%% file: hats-final.tex
\documentclass[a4paper,11pt,reqno]{amsart}

\usepackage[colorlinks=true, pdfstartview=FitV, linkcolor=blue,
citecolor=blue]{hyperref}

\usepackage{amssymb,bm}
\usepackage{amsmath}
\numberwithin{equation}{section}

\usepackage{graphicx,subcaption}
\usepackage{a4wide,rotating}
\usepackage{tikz}
\usetikzlibrary{graphs,arrows,shapes}

\tikzset{every picture/.style={line width=0.75pt}}

\newtheorem{theorem}{Theorem}[section]
\newtheorem{prop}[theorem]{Proposition}
\newtheorem{lemma}[theorem]{Lemma}
\theoremstyle{definition}

\newcommand{\ts}{\hspace{0.5pt}}
\newcommand{\nts}{\hspace{-0.5pt}}
\newcommand{\RR}{\mathbb{R}\ts}
\newcommand{\CC}{\mathbb{C}}
\newcommand{\ZZ}{\mathbb{Z}}
\newcommand{\NN}{\mathbb{N}}
\newcommand{\QQ}{\mathbb{Q}}
\newcommand{\XX}{\mathbb{X}}
\newcommand{\YY}{\mathbb{Y}}

\newcommand{\cK}{\mathcal{K}}
\newcommand{\cL}{\mathcal{L}}
\newcommand{\cO}{\mathcal{O}}
\newcommand{\cR}{\mathcal{R}}
\newcommand{\cT}{\mathcal{T}}

\newcommand{\vL}{\varLambda}
\newcommand{\vrho}{\varrho}
\newcommand{\vn}{\varnothing}

\newcommand{\CAP}{\mathrm{CAP}}
\newcommand{\CASPr}{\mathrm{CASPr}}
\newcommand{\ii}{\mathrm{i}}
\newcommand{\dd}{\,\mathrm{d}}
\newcommand{\ee}{\,\mathrm{e}}
\newcommand{\Mat}{\mathrm{Mat}}
\newcommand{\oplam}{\mbox{\Large $\curlywedge$}}

\newcommand{\card}{\ts\mathrm{card}\ts}

\newcommand{\inte}{\mathrm{int}}

\newcommand{\myfrac}[2]{\frac{\raisebox{-2pt}{$#1$}}
	{\raisebox{0.5pt}{$#2$}}}

\newcommand{\bhex}{_{\mathrm{hex}}}

\makeatletter
\renewcommand{\@captionfont}{\small}
\makeatother

\DeclareMathOperator{\dens}{dens}
\DeclareMathOperator{\vol}{vol}
\DeclareMathOperator{\sinc}{sinc}

\DeclareFontFamily{U}{mathx}{}
\DeclareFontShape{U}{mathx}{m}{n}{<-> mathx10}{}
\DeclareSymbolFont{mathx}{U}{mathx}{m}{n}
\DeclareMathAccent{\widecheck}{0}{mathx}{"71}

\DeclareMathAlphabet{\mathmybb}{U}{bbold}{m}{n}

\newcommand{\defeq}{\mathrel{\mathop:}=}

\begin{document}
	
\title[Diffraction of the Hat and Spectre tilings and some of their
relatives]{Diffraction of the Hat and Spectre tilings \\[3mm] and some
  of their relatives}
	
\author{Michael Baake}
\address{Fakult\"at f\"ur Mathematik, Universit\"at Bielefeld,
	\newline \indent Postfach 100131, 33501 Bielefeld, Germany}
\email{$\{$mbaake, gaehler, jmazac$\}$@math.uni-bielefeld.de}

\author{Franz G\"ahler}
	
\author{Jan Maz\'{a}\v{c}}
	
\author{Andrew Mitchell} \address{Faculty of Science,
          Technology, Engineering and Mathematics, The Open
          University, \newline \indent Milton Keynes, MK7 6AA, UK;
          \newline \indent present address: 
          Department of Mathematical Sciences,
          Loughborough University, \newline \indent
          Loughborough, LE11 3TU, UK }
\email{A.C.Mitchell@lboro.ac.uk}

\begin{abstract}
  The diffraction spectra of the Hat and Spectre monotile tilings,
  which are known to be pure point, are derived and computed
  explicitly. This is done via model set representatives of
  self-similar members in the topological conjugacy classes of the Hat
  and the Spectre tiling, which are the CAP and the CASPr tiling,
  respectively. This is followed by suitable reprojections of the
  model sets to represent the original Hat and Spectre tilings, which
  also allows to calculate their Fourier--Bohr coefficients
  explicitly. Since the windows of the underlying model sets have
  fractal boundaries, these coefficients need to be computed via an
  exact renormalisation cocycle in internal space.
\end{abstract}
	

\maketitle

\section{Introduction}
The recently discovered Hat and Spectre monotiles~\cite{Hat,Spectre}
give rise to tilings of the plane that are mean quasiperiodic (in the
sense of Weyl~\cite{LSS-long}) and possess pure-point dynamical
spectra~\cite{BGS,BGS2}.  Therefore, they have pure-point diffraction
\cite{BL} and, in fact, they are \emph{mutually locally derivable}
(MLD) with reprojections of regular model sets. Each of the latter
emerges from a~fully Euclidean \emph{cut-and-project scheme} (CPS)
with two-dimensional direct (or physical) and internal spaces. The
corresponding windows have been determined in~\cite{BGS,BGS2} and are
Rauzy fractals \cite[Ch.~7.4]{PyFo} with (some) boundaries of
non-integer Hausdorff dimension.

A numerical approximation to the diffraction of the Hat tiling was
obtained by Socolar~\cite{Soc} soon after its discovery, by
considering a large patch of a topologically conjugate tiling called
the \emph{Golden Key tiling} and its embedding in six-dimensional
space. Further numerical studies of large finite patches appeared in
\cite{Kap1} for the Hat tiling, with the correct conclusion of the
periodicity of the diffraction image, but an unclear interpretation of
diffuse parts, which can only be a finite-size effect. The
corresponding finite-size analysis for the Spectre tiling appeared in
\cite{Kap2}, which correctly gives a diffraction image without
translation symmetries. In our approach, we start from the established
result that the diffraction images, in the infinite size limit, are
pure point, with intensities that emerge from the \emph{Fourier--Bohr}
(FB) amplitudes by taking the square of their absolute values, which
leads to exact results.

To achieve maximal simplicity and transparency, we remove all
unnecessary dimensions and work with representative point sets that
consist of \emph{one} translation class only (with respect to the
intrinsically defined return module of each tiling). In doing so, it
is advantageous to also remove any superfluous dimension of internal
space, so that all windows have full dimension in it. Concretely,
although \cite{Soc} uses the projection method, the chosen embedding
has two unnecessary dimensions. Since the FB amplitudes needed for the
calculation of the diffraction are given via the Fourier transform of
(characteristic functions of) windows with fractal boundaries, they
are generally difficult to calculate, even approximately.  In
particular, there are known examples where the usual method of
finite-patch approximation converges really slowly; consult
\cite{BG-Rauzy2} for a~fully worked example.

For primitive unimodular inflation tilings, the underlying exact
inflation structure allows the Fourier transform of the associated
window to be represented as an infinite product of Fourier matrices,
called the \emph{Fourier cocycle}, with compact and exponentially fast
convergence~\cite{BG-Rauzy}.  In particular, this approach is
applicable even in cases where the window has fractal boundary, so can
be employed in our setting.  Specifically, we use an embedding with
minimal dimension in conjunction with an exact formula for the FB
coefficients, whose numerical computation can be done to any desired
precision in a~controlled way.  This permits accurate approximations
to the diffraction spectra. Due to the underlying Weyl mean
quasiperiodicity, we know that the superposition property holds on the
level of FB amplitudes~\cite{LSS-long}, so that an arbitrary set of
weights for the tile control points can be chosen. We present some
characteristic examples, calculated with a standard computer algebra
system, which uses exact integer arithmetic up to the final
(numerical) evaluation of the Fourier cocycle. \smallskip
	
The paper is organised as follows.  We first consider a guiding
example in one dimension to demonstrate the main techniques, which we
believe will also be of independent interest.  Section~\ref{sec:Hats}
concerns the Hat family of tilings.  We first recall the
cut-and-project description of the self-similar member of this family,
namely the CAP tiling~\cite{BGS}, for which we derive the diffraction
using the Fourier cocycle approach.  By utilising the description of
the Hat family of tilings as a reprojection (see for example
\cite{KelSad}) of the CAP tiling, we then obtain the diffraction for
the Hat tiling itself.  In Section~\ref{sec:spec}, we present the
analogous approach for the Spectre tiling, this time starting form the
self-similar CASPr tiling~\cite{BGS2}, once again followed by a
reprojection to cover a Delone set that is MLD with the Spectre
tiling.

\section{A guiding example in one dimension}\label{sec:example}

This section is meant to illustrate the methods and results that we
later apply to the Hat and the Spectre tilings, so we keep the
exposition informal. Even though some of the results are new, they all
follow from methods and tools that are known. As we proceed, we also
introduce some of the notation we later need, where~\cite{TAO} is our
guiding reference.
    
\subsection{A twisted silver mean inflation}
    
Consider the primitive, binary substitution rule
\begin{equation}
	\vrho: \qquad a\,\mapsto \, abb, \quad b\, \mapsto \, ab, 
	\label{eq:subst1}
\end{equation}
which we write more concisely as
$\vrho = \bigl(\vrho(a),\vrho(b)\bigr) = (abb,ab)$ from now on (and
analogously for other substitutions). The rule is extended to words
via the usual homomorphism property of $\vrho$~\cite[Ch.~4]{TAO}.  It
has the substitution matrix
\begin{equation}
	M \, = \, \begin{pmatrix}
		1&1\\2&1 \end{pmatrix}
		\label{eq:substmatrix}
\end{equation}
with \emph{Perron--Frobenius} (PF) eigenvalue $\lambda = 1+ \sqrt{2}$,
which is a \emph{Pisot--Vijayaraghavan} (PV) unit, and corresponding
left and right eigenvectors
\begin{equation}
  \langle u | \, = \, \bigl(\sqrt{2},1\bigr) \quad \mbox{and} \quad
  |v\rangle \,=\, \lambda^{-1}\bigl(1,\sqrt{2}\,\bigr)^{\top}
  \,= \, \bigl(\lambda-2,3-\lambda\bigr)^{\top}.
\label{eq:eigenvectors}
\end{equation}
Here, the right eigenvector is frequency (or statistically)
normalised, so that $\lambda-2$ and $3-\lambda$ are the relative
frequencies of the letters $a$ and $b$ in the bi-infinite fixed point
$w = \vrho(w)$, where
$w\, =\, \dots w^{}_{-2}\, w^{}_{-1}\, |\, w^{}_{0}\, w^{}_{1}\,
w^{}_{2} \dots$ is the limit of the iteration sequence
\[
  b|a\, \longmapsto\, ab|abb \, \longmapsto \, abbab|abbabab \,
  \longmapsto \, \cdots \, \longrightarrow \, w\,=\,\vrho(w).
\]
The marker $|$ is kept as the reference point; compare
\cite[Ch.~4]{TAO} for details.
	
The word $w$ is repetitive, which means that all finite subwords of
$w$ reappear with bounded gaps, see \cite[Sec.~4.2]{TAO} for
details. The corresponding \emph{discrete} hull,
\begin{equation}\label{eq:hull}
  \XX \, = \, \XX(w)\, = \, \overline{\left\{S^n w \, :
               \, n\in \ZZ \right\}},	
\end{equation}
is the closure of the shift orbit of $w$ in the product topology of
$\left\{a,b\right\}^{\ZZ}$. Here, $S$ denotes the usual left shift,
$(Sw)^{}_{m} \,=\,w^{}_{m+1}$.  The hull $\XX$ is a closed, minimal
subshift of $\left\{a,b\right\}^{\ZZ}$, which is aperiodic
\cite[Thm.~4.6]{TAO}. A central goal in the theory of aperiodic order
is to determine the dynamical and diffraction spectra of $\XX$, where
the latter refers to the support of the diffraction image of a generic
element of $\XX$ (which can be any in this case, due to
minimality). For this, a geometric approach is possible.  Since this
is not widely known, we briefly recall it here.

\begin{figure}
\tikzset{every picture/.style={line width=0.75pt}}        
\begin{tikzpicture}[x=0.75pt,y=0.75pt,yscale=-0.8,xscale=0.8]
  \draw [draw opacity=0][fill={rgb, 255:red, 0; green, 0; blue, 255 }
  ,fill opacity=0.7 ] (20,50) -- (80,50) -- (80,70) -- (20,70) --
  cycle ;
			
  \draw [line width=1.5] (20,50) -- (20,70) ; \draw [line width=1.5]
  (80,50) -- (80,70) ;
			
  \draw [draw opacity=0][fill={rgb, 255:red, 255; green, 0; blue, 32 }
  ,fill opacity=0.7 ] (20,20) -- (104,20) -- (104,40) -- (20,40) --
  cycle ; \draw [line width=1.5] (20,20) -- (20,40) ; \draw [line
  width=1.5] (104,20) -- (104,40) ;
			
  \draw [draw opacity=0][fill={rgb, 255:red, 0; green, 0; blue, 255 }
  ,fill opacity=0.7 ] (166,50) -- (310,50) -- (310,70) -- (166,70) --
  cycle ; \draw [line width=1.5] (166,50) -- (166,70) ; \draw [line
  width=1.5] (310,50) -- (310,70) ;
			
  \draw [draw opacity=0][fill={rgb, 255:red, 255; green, 0; blue, 32 }
  ,fill opacity=0.7 ] (166,20) -- (370,20) -- (370,40) -- (166,40) --
  cycle ; \draw [line width=1.5] (166,20) -- (166,40) ; \draw [line
  width=1.5] (370,20) -- (370,40) ;
			
  \draw [draw opacity=0][fill={rgb, 255:red, 255; green, 0; blue, 32 }
  ,fill opacity=0.7 ] (436,20) -- (520,20) -- (520,40) -- (436,40) --
  cycle ; \draw [line width=1.5] (436,20) -- (436,40) ; \draw [line
  width=1.5] (520,20) -- (520,40) ;
			
  \draw [draw opacity=0][fill={rgb, 255:red, 0; green, 0; blue, 255 }
  ,fill opacity=0.7 ] (580,20) -- (640,20) -- (640,40) -- (580,40) --
  cycle ; \draw [line width=1.5] (580,20) -- (580,40) ; \draw [line
  width=1.5] (640,20) -- (640,40) ;
			
  \draw [draw opacity=0][fill={rgb, 255:red, 0; green, 0; blue, 255 }
  ,fill opacity=0.7 ] (520,20) -- (580,20) -- (580,40) -- (520,40) --
  cycle ; \draw [line width=1.5] (520,20) -- (520,40) ; \draw [line
  width=1.5] (580,20) -- (580,40) ;
			
  \draw [draw opacity=0][fill={rgb, 255:red, 255; green, 0; blue, 32 }
  ,fill opacity=0.7 ] (436,50) -- (520,50) -- (520,70) -- (436,70) --
  cycle ; \draw [line width=1.5] (436,50) -- (436,70) ; \draw [line
  width=1.5] (520,50) -- (520,70) ;
			
  \draw [draw opacity=0][fill={rgb, 255:red, 0; green, 0; blue, 255 }
  ,fill opacity=0.7 ] (520,50) -- (580,50) -- (580,70) -- (520,70) --
  cycle ; \draw [line width=1.5] (520,50) -- (520,70) ; \draw [line
  width=1.5] (580,50) -- (580,70) ;
			
  \draw (56,25) node [anchor=north west][inner sep=0.75pt]
  {$\textcolor[rgb]{1,1,1}{a}$}; \draw (46,51.4) node [anchor=north
  west][inner sep=0.75pt] {$\textcolor[rgb]{1,1,1}{b}$}; \draw
  (476,25) node [anchor=north west][inner sep=0.75pt]
  {$\textcolor[rgb]{1,1,1}{a}$}; \draw (546,22.4) node [anchor=north
  west][inner sep=0.75pt] {$\textcolor[rgb]{1,1,1}{b}$}; \draw
  (606,22.4) node [anchor=north west][inner sep=0.75pt]
  {$\textcolor[rgb]{1,1,1}{b}$}; \draw (476,54) node [anchor=north
  west][inner sep=0.75pt] {$\textcolor[rgb]{1,1,1}{a}$}; \draw
  (546,51.4) node [anchor=north west][inner sep=0.75pt]
  {$\textcolor[rgb]{1,1,1}{b}$};
			
  \draw (120,60) -- (147,60) ; \draw [shift={(150,60)}, rotate = 180]
  [fill={rgb, 255:red, 0; green, 0; blue, 0 } ][line width=0.08] [draw
  opacity=0] (10.72,-5.15) -- (0,0) -- (10.72,5.15) -- (7.12,0) --
  cycle ; \draw (120,30) -- (147,30) ; \draw [shift={(150,30)}, rotate
  = 180] [fill={rgb, 255:red, 0; green, 0; blue, 0 } ][line
  width=0.08] [draw opacity=0] (10.72,-5.15) -- (0,0) -- (10.72,5.15)
  -- (7.12,0) -- cycle ; \draw (390,30) -- (417,30) ; \draw
  [shift={(420,30)}, rotate = 180] [fill={rgb, 255:red, 0; green, 0;
    blue, 0 } ][line width=0.08] [draw opacity=0] (10.72,-5.15) --
  (0,0) -- (10.72,5.15) -- (7.12,0) -- cycle ; \draw (390,60) --
  (417,60) ; \draw [shift={(420,60)}, rotate = 180] [fill={rgb,
    255:red, 0; green, 0; blue, 0 } ][line width=0.08] [draw
  opacity=0] (10.72,-5.15) -- (0,0) -- (10.72,5.15) -- (7.12,0) --
  cycle ;
\end{tikzpicture}
\vspace*{-2mm}
\caption{The geometric inflation rule with prototiles of the natural
  lengths as induced by the substitution $\vrho$. \label{fig:inflrule}
}
\end{figure}
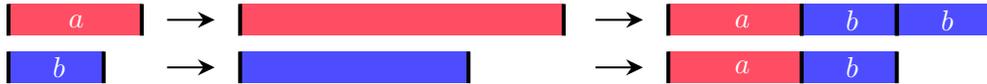

As an intermediate step, we consider the self-similar inflation rule
shown in Figure~\ref{fig:inflrule}.  The rule is induced by the action
of $\vrho$ on two prototiles in the form of intervals of length
$\sqrt{2}$ (for $a$) and 1 (for $b$), which are chosen from the
entries of $\langle u|$ in \eqref{eq:eigenvectors}. This turns the
(symbolic) fixed point $w$ into a self-similar aperiodic tiling of
$\RR$ with two types of intervals.
	
If $\cT^{}_{w}$ denotes the tiling induced by $w$, with $w^{}_0$
labelling the corresponding interval starting at 0, the
\emph{continuous} hull is
\begin{equation}
  \YY\,\defeq \, \overline{\left\{t+\cT^{}_{w} \, : \, t\in \RR\right\}},
\label{eq:conthull}
\end{equation} 
where $t+\cT^{}_{w}$ is the translate of $\cT^{}_{w}$ by $t$ and the
closure is taken in the local topology. In this topology, two tilings
are $\varepsilon$-close if they agree on the interval
$\bigl(-\tfrac{1}{\varepsilon},\tfrac{1}{\varepsilon}\bigr)$, possibly
after a~global translation of one of them by at most $\varepsilon$. As
is well known for primitive inflation rules on a finite prototile set,
$\YY$ is compact, and the translation action is continuous in the
local topology, so $(\YY,\RR)$ defines a \emph{topological dynamical
  system} (TDS). Since $\vrho$ is primitive, where we use $\vrho$ both
for the substitution \eqref{eq:subst1} and for the inflation in
Figure~\ref{fig:inflrule}, $\YY$ is minimal and possesses a unique
translation-invariant probability measure, say $\mu$, so that we can
look at the TDS $(\YY,\RR)$ in comparison with its unique counterpart,
the measure-theoretic dynamical system $(\YY,\RR,\mu)$. We first
determine its dynamical spectrum and then derive that of $(\XX,\ZZ)$
from it.

\subsection{Embedding and diffraction}
	
It is advantageous to represent any $\cT\in \YY$ by a Delone set that
emerges from $\cT$ by taking the left endpoint of each
interval. Clearly, the tilings and the derived point sets are MLD, see
\cite[Sec.~5.2]{TAO} for background on this concept, so we can work
with the Delone sets as a representative of the topological conjugacy
class defined by $\YY$. If $\vL = \vL^{}_a \, \dot{\cup} \, \vL^{}_b $
is the point set from the fixed point tiling $\cT^{}_{w}$, where we
keep track of the point type according to the corresponding interval
type, the fixed point property can be written as
\begin{equation}\label{eq:EMS}
\begin{split}
  \vL^{}_a \, & = \, \lambda \vL^{}_a \ \dot{\cup} \
  \lambda \vL^{}_b \, = \, \lambda \vL, \\[1mm]
  \vL^{}_b \, & = \, \lambda \vL^{}_a\! +\! \sqrt{2} \
  \ \ \dot{\cup} \ \ \lambda \vL^{}_a \! +\! (1{+}\sqrt{2}) \ \
  \dot{\cup} \ \ \lambda\vL^{}_b \! +\! \sqrt{2} \\[1mm]
  & = \, \lambda \vL \! +\!  \sqrt{2} \ \ \dot{\cup} \ \
  \lambda\vL^{}_a \!+\! \lambda.
\end{split}
\end{equation}
The right-hand side encodes the points of $\vL^{}_{a}$ and
$\vL^{}_{b}$ via the inflated versions of them, compare
\cite[Ch.~5]{Bernd}. More compactly, after renaming $a$ and $b$ by $1$
and $2$, this is
\[
  \vL_i \, = \bigcup_{j=1}^{2} \, \bigcup_{t \in T^{}_{ij}} \lambda
  \vL_j + t \ts ,
\]
with translations $t$ that are the entries of the set-valued
\emph{displacement matrix }
\[
  T\,=\, \begin{pmatrix} \{0\} & \{0\} \\[2mm]
    \{\sqrt{2},1+\sqrt{2}\,\} & \{\sqrt{2}\,\} \end{pmatrix},
\]
where $T_{ij}$ is the set of relative translations of all tiles of
type $i$ in a supertile of type $j$. We note that
$M_{ij} = \card (\,T_{ij})$.
	
By construction, all points of $\vL$ lie in $\ZZ[\sqrt{2}\,]$, the
ring of integers in the quadratic field $\QQ(\sqrt{2}\,)$, as do their
differences, which generate the return module. While $\ZZ[\sqrt{2}\,]$
is a dense subset of $\RR$, its Minkowski embedding into $\RR^2$, as
given by
\[
  \cL \,= \, \left\{(x,x^{\star}) \, : \, x\in\ZZ[\sqrt{2}\,] \right\}
  \, = \, \left\{ m\left(\begin{smallmatrix}
        1\\[1mm]1 \end{smallmatrix} \right) + n
    \left(\begin{smallmatrix} \sqrt{2}\\
        -\sqrt{2} \end{smallmatrix} \right) \, : \, m,n \in \ZZ
  \right\},
\]
is a lattice, where $\star$ denotes the non-trivial algebraic
conjugation in $\QQ(\sqrt{2}\,)$, which is the unique $\QQ$-linear
mapping defined by $\sqrt{2} \mapsto -\sqrt{2}$. Altogether, we have
an example of a Euclidean CPS, namely
\begin{equation}\label{eq:CPS}
\renewcommand{\arraystretch}{1.2}
  \begin{array}{ccccc@{}l}
    \RR & \xleftarrow{\;\;\; \pi \;\;\; }
    & \ \ \RR \nts\nts \times \nts\nts \RR^{}_{\inte}
    & \xrightarrow{\;\: \pi^{}_{\text{int}} \;\: }
    & \ \RR^{}_{\inte} & \\
    \cup & & \cup & & \cup  & \hspace*{-2ex}
                              \raisebox{1pt}{\text{\scriptsize dense}} \\
    \pi (\cL) & \xleftarrow{\;\,\ts 1:1 \;\,\ts } & \cL &
          \xrightarrow{ \qquad } &\pi^{}_{\text{int}} (\cL) & \\
    \| & & & & \| & \\ L = \ZZ[\sqrt{2}\,] &
          \multicolumn{3}{c}{\xrightarrow{\qquad\quad\qquad \,  \star \,
                                             \qquad\quad\qquad\ts}}
    &  {L_{}}^{\star\nts} = \ZZ[\sqrt{2}\,]  &  \end{array}
\renewcommand{\arraystretch}{1}
\end{equation}
which goes back to the work of Meyer~\cite{Mey72} and Moody
\cite{Moo97}. We will consistently keep track of which space is the
internal space by a subscript.
	
The structure of the CPS suggests replacing the difficult expansive
system of equations \eqref{eq:EMS}, for which no general solution
theory is known, by their $\star$-mapped version, namely
\begin{equation}\label{eq:IFS1}
\begin{split}
  W^{}_a \, & = \, \lambda^{\star} W^{}_a \ \cup \
  \lambda^{\star} W^{}_b \, ,  \\[1mm]
  W^{}_b \, & = \, \lambda^{\star} W^{}_a - \sqrt{2} \ \cup \
  \lambda^{\star} W^{}_a + \lambda^{\star} \ \cup \ \lambda^{\star}
  W_b - \sqrt{2} \ts ,
\end{split}
\end{equation}
where $W^{}_a$ and $W^{}_b$ are the closures of $\vL^{\star}_a$ and
$\vL^{\star}_b$, respectively, in the topology of
$\RR^{}_{\inte}$. Due to having taken the closure, the unions on the
right-hand sides need no longer be disjoint. The reason for this step
is that, due to $|\lambda^{\star}|<1$ (the PV property of $\lambda$),
Eq.~\eqref{eq:IFS1} defines a \emph{contractive} iterated function
system (IFS) on $\bigl(\cK\RR^{}_{\inte} \bigr)^2$, where
$\cK\RR^{}_{\inte}$ is the space of all non-empty, compact subsets of
$\RR^{}_{\inte}$, equipped with the Hausdorff
metric~\cite{Hutch81,Wicks}. So, by Banach's contraction principle,
there is a \emph{unique} pair $(W^{}_a,W^{}_b)$ of non-empty compact
sets that solves~\eqref{eq:IFS1}, and one can check explicitly that
they are given by
\begin{equation}\label{eq:IFS1_solution}
  W^{}_a \, = \, \Bigl[\myfrac{\sqrt{2}}{2} -1,\,
  \myfrac{\sqrt{2}}{2} \, \Bigr] \qquad \mbox{and} \qquad
  W^{}_{b}\,=\, \Bigl[-1-\myfrac{\sqrt{2}}{2},\,
  \myfrac{\sqrt{2}}{2}-1 \Bigr],
\end{equation} 
where
$W = W^{}_a \cup W^{}_b = \bigl[
-1-\tfrac{\sqrt{2}}{2},\tfrac{\sqrt{2}}{2}\,\bigr]$ is the union of
both, with
$ W^{}_a \cap W^{}_b = \bigl\{\tfrac{\sqrt{2}}{2}-1 \bigr\}$.
	
For a relatively compact $U\subset \RR^{}_{\inte}$, the point
set
\[
  \oplam (U) \, \defeq \, \{x\in L \, : \, x^{\star} \in U \}
\]
is called a \emph{cut-and-project set}. If $U$ has a non-empty
interior, it is a \emph{model set}. Such a model set is called
\emph{regular} if $\partial U$ has Lebesgue measure 0, and
\emph{proper} if $U$ is compact and the closure of its
interior. Therefore, $\oplam(W^{}_a)$ and $\oplam(W^{}_b)$ are proper,
regular model sets, as is $\oplam(W)$. As such, they all define
dynamical systems with pure-point dynamical spectrum and continuous
eigenfunctions~\cite{TAO,Lenz09}.
	
By construction, we know that $\vL^{}_a \subseteq \oplam(W^{}_a)$ and
$\vL^{}_b \subseteq \oplam(W^{}_b)$, as well as
$\vL^{} \subseteq \oplam(W)$. Moreover, in all three cases, the two
sets have the same density. For $\vL$, we obtain
\[
  \dens(\vL) \, = \, \langle u | v \rangle ^{-1}
  \, = \, \myfrac{2+\sqrt{2}}{4} \, = \, \myfrac{\lambda+1}{4}.
\]
With $\dens(\cL) = \tfrac{\sqrt{2}}{4}$, the uniform distribution
property in the window for regular model sets~\cite{Moody,Schl98}
gives
\[
  \dens\bigl(\oplam(W)\bigr) \, = \, \dens(\cL)\ts \vol(W)
  \, = \, \myfrac{2+\sqrt{2}}{4} \ts .
\]
Since the boundary points of $W^{}_a$ and $W^{}_b$ do not lie in
$L^{\star} = \ZZ[\sqrt{2}\,]$, we obtain
\[
  \vL^{}_a \, = \, \oplam(W^{}_a), \quad \vL^{}_b
  \,=\,\oplam(W^{}_b) \quad \mbox{and} \quad
  \vL \,=\,\oplam(W) \ts .
\]

We are now in the position to determine the spectrum of the tiling
$\cT^{}_w$ both in the dynamical and the diffraction sense, where the
dynamical spectrum is the group generated by the support of the
diffraction measure~\cite{BL}.  The general theory of diffraction of
regular model sets is well-known; see~\cite[Ch.~9]{TAO} for a detailed
survey. To determine it explicitly, some work is necessary. First, one
has to find the support, the Fourier module $L^{\circledast}$, and
then the intensities of the Bragg peaks on $L^{\circledast}$. The
second step is usually done by computing the FB coefficients of the
structure and then taking the squares of their absolute values.
	
The Fourier module $L^{\circledast}$ can be obtained from the dual
lattice $\cL^{\ast}$ by taking its $\pi$-projection. In our guiding
example, we have
\begin{equation}\label{eq:Spectrum_SM}
  L^{\circledast} \,= \, \pi(\cL^{\ast}) \,= \,
  \myfrac{\sqrt{2}}{4}\,\ZZ[\sqrt{2}\,]. 
\end{equation}

\noindent
While non-generic extinctions are possible, so that the FB amplitudes
for some or even infinitely many elements of $L^{\circledast}$ vanish,
no generic weighting of the two point types will lead to a proper
subgroup of $L^{\circledast}$; see \cite[Rem.~9.10]{TAO} for a related
discussion. $L^{\circledast}$ thus is the smallest (additive) group
that contains all locations with non-trivial Bragg peaks.
	
Now, suppose that one assigns weights $\alpha,\, \beta\in \CC$ to all
points of type $a$ and $b$, respectively. Then, the diffraction
measure $\widehat{\gamma}$ reads
\begin{equation}\label{eq:diff}
  \widehat{\gamma} \, = \, \sum_{k\in L^{\circledast}}\bigl|
  \alpha A^{}_{\vL_a}(k) + \beta A^{}_{\vL_b}(k) \bigr|^2 \delta_k
  \, = \, \sum_{k\in L^{\circledast}}I^{}_{\vL}(k)\, \delta_k,
\end{equation}
where the amplitudes $A^{}_{\vL_i}(k)$ are the FB amplitudes or
\emph{coefficients} defined by
\[
  A^{}_{\vL^{}_{i}}(k) \, \defeq \, \lim_{r\to\infty} \myfrac{1}{2r}
  \sum_{\substack{x\in \vL^{}_{i}\\ |x|\, \leqslant \,r}} \ee^{-2\pi \ii k x } .
\]
The limits exist for all $k\in \RR$ (see~\cite{BH} for a recent
elementary proof), and are given by
\begin{equation}\label{eq:FBcoeff}
  A^{}_{\vL_i}(k) \, = \,\begin{cases}  H^{}_{i}(k^{\star}),
    & \mbox{if}\ k\in L^{\circledast},\\
	0, &\mbox{otherwise}, \end{cases}
\end{equation}
where
\[
  H^{}_{i}(k^{}_{\inte}) \, =\, \frac{\dens(\vL)}{\vol(W)}
  \ \widecheck{\bm{1}^{}_{W^{}_i}}(k^{}_{\inte})
\] 
is the (scaled) inverse Fourier transform of the characteristic
function of the window $W^{}_i$.

The amplitudes can be calculated explicitly for
$k\in L^{\circledast}$, where one obtains
\begin{equation}\label{eq:FB}
\begin{split}
  H^{}_{a}(k^{}_{\inte}) &\,=\, \myfrac{\sqrt{2}}{4}
  \int_{\frac{\sqrt{2}}{2}-1}^{\frac{\sqrt{2}}{2}} \ee^{2\pi\ii k^{}_{\inte}y}
  \dd y \, = \, \myfrac{\sqrt{2}}{4} \ee^{\pi\ii k^{}_{\inte}(\lambda -2)}
  \, \sinc(\pi k^{}_{\inte})\quad \mbox{and} \\[1.5mm]
  H^{}_{b}(k^{}_{\inte}) &\, = \, \myfrac{\sqrt{2}}{4}
  \int_{\frac{-\sqrt{2}}{2}-1}^{\frac{\sqrt{2}}{2}-1} \ee^{2\pi\ii k^{}_{\inte}y}
  \dd y \, = \, \myfrac{1}{2} \ee^{-2\pi\ii k^{}_{\inte}}
  \sinc\bigl(\pi k^{}_{\inte}(\lambda-1)\bigr)\ts ,
\end{split}
\end{equation}
with $\sinc(z) = \tfrac{\sin(z)}{z}$. These amplitudes also constitute
a set of eigenfunctions for the Koopman operator acting on $\YY$,
because they satisfy
\[
  A^{}_{t+\vL^{}_i}(k) \, = \, \ee^{-2\pi\ii k t} A^{}_{\vL^{}_i}(k)\ts
\]
for any $t\in\RR$. Thus, unless they vanish, the FB coefficients are
eigenfunctions for the dynamical eigenvalue $\ee^{-2\pi\ii k t}$ with
$k\in L^{\circledast}$, which adds important dynamical information.

\begin{figure}
\includegraphics[width=0.85\textwidth]{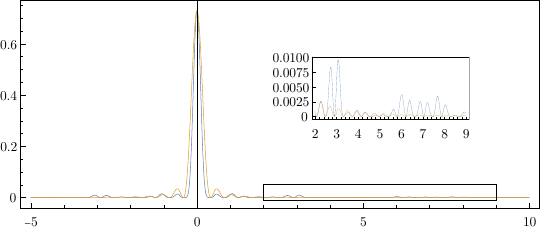}
\caption{Diffraction intensities for model sets
    with equal weights belonging to
    the substitutions $\varrho$, and $\widetilde{\vrho}$ from
    Section~\ref{sec:variation} (blue depicts
    $|\widetilde{H}(k^{}_{\inte})|^2$ and
    yellow $|H(k^{}_{\inte})|^2$). The intensity of the central
    peak is the same in both cases and given by the point set density,
     $|\widetilde{H}(0)|^2 = |H(0)|^2 = \dens(\vL)^2 =
    \tfrac{\lambda^2}{8} \approx 0.72855\dots$. For the approximation
    of $|\widetilde{H}(k^{}_{\inte})|^2$, the cocycle was used
    with $n=20$; see Section~\ref{sec:variation} for further details.
    \label{fig:intensities}}
\end{figure}

\begin{figure}
\includegraphics[width=0.85\textwidth]{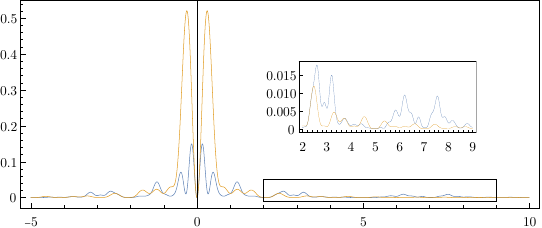}
\caption{Diffraction intensities for model sets arising from
  the substitutions $\varrho$, and $\widetilde{\vrho}$ from
  Section~\ref{sec:variation} with weighting (blue for
  $|\widetilde{H}(k^{}_{\inte})|^2$ and yellow for
  $|H(k^{}_{\inte})|^2$) for weighted sets $\widetilde{\vL}$ and $\vL$
  with weights $\alpha = \sqrt{2}$ and $\beta = -1$. The weights are
  chosen so that the intensity of the central peak vanishes, so
  $|\widetilde{H}(0)|^2 = |H(0)|^2 = 0$. For the approximation of
  $|\widetilde{H}(k^{}_{\inte})|^2$, the cocycle was again used with
  $n=20$.  \label{fig:intensitiesZeroCP}}
\end{figure}

Let us define $H = \alpha H^{}_{a} + \beta H^{}_{b}$ for the total
amplitude of the weighted system.  Choosing $\alpha = \beta = 1$, the
total intensity $I^{}_{\vL}(k)$ reads for all $k\in L^{\circledast}$
\[
  I^{}_{\vL}(k) \,= \, |H(k^{\star})|^2 \, =\,
  \myfrac{\lambda^2}{8}\sinc^2(\pi \lambda k^{\star}) \ts .
\]
For general weights $\alpha,\, \beta \in \CC$, the total intensity
becomes for all $k\in L^{\circledast}$
\begin{equation*}
\begin{split}
  I^{}_{\vL}(k) & = \, |H(k^{\star})|^2 \, =\, \myfrac{|\alpha|^2}{8}
  \sinc^2\bigl(\pi k^{\star}\bigr) + \myfrac{|\beta|^2}{8}
  \sinc^2\bigl(\pi k^{\star}(\lambda-1)\bigr)  \\
  & \quad + \myfrac{\sqrt{2}}{4}|\alpha\overline{\beta}|\,
  \cos(\pi k^{\star} \lambda + \phi)\,\sinc\bigl(\pi k^{\star}\bigr)
  \,\sinc\bigl(\pi k^{\star}(\lambda-1)\bigr)
\end{split}
\end{equation*}
with $\phi = \arg(\alpha\overline{\beta})$. These functions are shown
in Figures~\ref{fig:intensities} and
\ref{fig:intensitiesZeroCP}. Thus, we have an explicit formula for the
diffraction measure~$\widehat{\gamma}$ from Eq.~\eqref{eq:diff}.

\subsection{Shape change and reprojection}

So far, we have determined the spectra of the self-similar version of
the (geometric) binary sequence. To return to the original symbolic
binary sequence, we employ a shape change, all in the framework of
\emph{deformed model sets}~\cite{BL05,BerDun00}. Here, the only
relevant shape change amounts to changing the lengths of the two
intervals in such a way that the overall point density remains the
same. Such shape changes, according to results of Clark and Sadun
\cite{CS1,CS2}, lead to topologically conjugate dynamical systems. We
note that for sufficiently nice model sets (Euclidean setting and
polygonal window), the converse also holds. In other words, all
topological conjugacies are (MLD with) reprojections, see
\cite{KelSad} for detailed treatment. The only remaining degree of
freedom is a global change of scale, which changes all spectral
properties in a controlled way.  So let us fix this scale. If the
tiles $a$ and $b$ have lengths $\ell_a$ and $\ell_b$, the point
density is given by $\lambda\cdot(\ell_a+\sqrt{2}\,\ell_b)^{-1}$,
which must be equal to $\dens(\vL)=\tfrac{1}{4}(2+\sqrt{2})$, so
\begin{equation}\label{eq:def_cond}
  \ell_a + \sqrt{2}\,\ell_b \, = \, 2\sqrt{2}.
\end{equation}
Here, $\ell_a=\sqrt{2}$ and $\ell_b=1$ correspond to $\vL$, while
choosing $\ell_a = \ell_b = 4-2\sqrt{2} \approx 1.17157\dots$ gives
equal lengths and thus a scaled version of the initial symbolic case
(embedded into $\RR$ by a suspension with a constant roof function).
As long as \eqref{eq:def_cond} is satisfied, the dynamical spectrum
remains the same due to topological conjugacy. Note that we always
keep track of the interval type, which guarantees aperiodicity also
when $\ell_a = \ell_b$, as we have it in the symbolic setting.

\begin{figure}
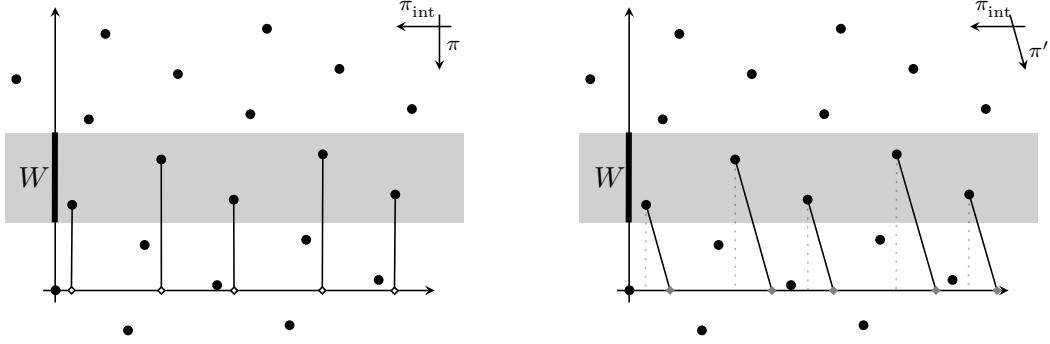

\centering
   \include{reprojection}
   \vspace*{-10mm}
   \caption{An illustration of the reprojection technique. While the
     left figure shows the initial cut-and-project scheme with two
     orthogonal projections $\pi$, $\pi^{}_{\inte}$, the right
     shows the change of the first projection. The strip
     $(\RR\times W) \cap \cL$ (shaded area) in the underlying space
     remains unchanged, but the resulting projection sets are
     different. \label{fig:reprojection}}
\end{figure}
   
To understand how the diffraction measure changes, we interpret the
shape change as a~reprojection from the same CPS and then apply the
formula for the diffraction of deformed model sets. The reprojection
takes the same lattice points as before, that is
$(\RR\times W) \ts \cap \ts \cL$, but projects them back to $\RR$ with
a different projection, say $\pi'$, as illustrated in Figure
\ref{fig:reprojection}. The projections are linear mappings, so they
can be represented by matrices. In our case, one has (in the standard
basis)
\[
  \pi \,=\, \bigl(1 \ \ 0\bigr), \qquad \pi^{}_{\inte} \,=\,
  \bigl(0 \ \ 1\bigr), \quad \mbox{and} \quad \pi'\,=\, \bigl(1 \ \
  \lambda^{-2}\bigr) \,=\, \bigl(1 \ \ 3{-}2\sqrt{2}\,\bigr).
\]
In particular,
\begin{equation}
  \pi'\left(\begin{smallmatrix}
      \sqrt{2}\\-\sqrt{2} \end{smallmatrix} \right)  \, = \,
  4-2\sqrt{2}, \qquad \mbox{and} \qquad 
  \pi'\left(\begin{smallmatrix}
      1\\[1mm] 1 \end{smallmatrix} \right) \, = \, 4-2\sqrt{2},
\label{eq:reprojection}
\end{equation}
as desired. Therefore, one can write the reprojected set $\vL'$ as
\[
  \vL' \, = \, \left\{ x + Dx^{\star} \, : \, x\in \vL \right\} \ts ,
\]
with the linear mapping $D: \RR^{}_{\inte} \rightarrow \RR$
being defined as $Dy = \lambda^{-2} y$ for all
$y\in \RR^{}_{\inte}$.
	
As mentioned before, the dynamical spectrum remains the same. Now, we
derive the new FB coefficients for $k\in L^{\circledast}$, the
amplitudes $A^{}_{\vL'}(k)$, from the definition,
\begin{equation}
\begin{split}
  A^{}_{\vL'}(k) & =\, \lim_{r\to\infty} \,\myfrac{1}{2r} \sum_{x\in \vL'_r}
  \ee^{-2\pi\ii k x} \, =\, \lim_{r\to\infty} \, \myfrac{1}{2r} \sum_{x\in \vL_r}
  \ee^{-2\pi\ii k (x+\lambda^{-2}x^{\star})} \\
  & =\, \lim_{r\to\infty}\, \myfrac{1}{2r} \sum_{x\in \vL_r}
  \ee^{2\pi\ii (k^{\star}x^{\star} - k\lambda^{-2}x^{\star} )} \\
  & =\, \myfrac{\dens(\vL)}{\vol(W)} \int_{W}
  \ee^{2\pi\ii (k^{\star} - k\lambda^{-2})y}  \dd y \, = \,
  \myfrac{\dens(\vL)}{\vol(W)} \,  \widecheck{\bm{1}^{}_{\,W}}\,
  (k^{\star} - k\lambda^{-2}),\\[2pt]
  & = \, H(k^{\star} - D^{\top}k),
\end{split}
\label{eq:FBdeformed}
\end{equation}
with $\vL_r \defeq \vL\cap [-r,r]$, and analogously for
$A^{}_{\vL'_{a}}$, and $A^{}_{\vL'_{b}}$.  We first used results from
\cite{BerDun00}, in the second row the fact that $kx$ is an algebraic
integer for all $k\in L^{\circledast}$ and all $x\in\vL$, and the
third line follows from standard equidistribution results in the
window; compare~\cite[Ch.~7]{TAO} or~\cite[Prop.~2.1]{Schl98}.  This
implies that the diffraction of the reprojected model set can be
computed from Eq.~\eqref{eq:FB}.
	
One has $\vL' = (4-2\sqrt{2}\,)\ZZ$. If the weights are chosen as
$\alpha= \beta = 1$, the diffraction measure is a periodic measure
supported on $\tfrac{2+\sqrt{2}}{4}\ZZ$, the dual lattice to
$(4-2\sqrt{2}\,)\ZZ$. Concretely, Eq.~\eqref{eq:FBdeformed} for the
amplitudes becomes
\[
  A^{}_{\vL'}(k) \, = \, \begin{cases}
    H'(k^{\star})\, \defeq \, \frac{\lambda +1}{4}, &\mbox{if} \
    k\in \frac{2+\sqrt{2}}{4}\,\ZZ, \\
    0, & \mbox{otherwise,}	\end{cases}
\]
which is the expected result. Note that the support of the diffraction
measure is $\tfrac{2+\sqrt{2}}{4}\,\ZZ$ and hence a rank-$1$ submodule
of the initial Fourier module $L^{\circledast}$ from
\eqref{eq:Spectrum_SM}.
	
If the weights $\alpha, \beta$ differ, the aperiodic structure
survives and is present in the diffraction picture, as one can see in
Figure~\ref{fig:DiffDefSM}. Moreover, one recovers the entire original
Fourier module $L^{\circledast}$. Nevertheless, since the weighted
point set is supported on a lattice, it follows from
\cite[Thm.~10.3.]{TAO} that the diffraction measure remains periodic
with its period given by the dual lattice $\tfrac{2+\sqrt{2}}{4}\ZZ$,
and the aperiodic nature manifests itself in the diffraction measure
restricted to the fundamental domain, for example to the interval
$[0,\tfrac{2+\sqrt{2}}{4})$.

\begin{figure}
\centering
\includegraphics[width=0.6\linewidth]{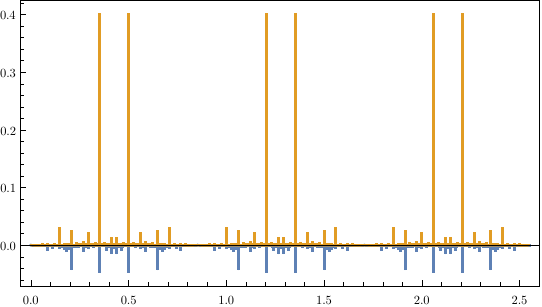}
\caption{Diffraction of the deformed model sets $\vL'$ (yellow/top)
  and $\widetilde{\vL}'$ (bottom/blue) with weights
  $\alpha = \sqrt{2}$ and $\beta = -1$. They are chosen so that the
  central peak vanishes. The figure shows the intensities at all
  points $k = \tfrac{\sqrt{2}}{4}(m+n\sqrt{2}\,)$ with
  $m,n\in \{-250,\,\dots,\, 250\}$ in three fundamental domains of the
  dual lattice $\tfrac{2+\sqrt{2}}{4}\ZZ$.}
\label{fig:DiffDefSM}
\end{figure}

\subsection{Some variations on the guiding example}\label{sec:variation}

The substitution matrix $M$ from \eqref{eq:substmatrix} is compatible
with six substitutions, namely with
\[
  (abb,ab),\ (bab,ba), \ (bba,ba),\ (bab,ab),
\]
which all define the same discrete hull $\XX$, and with the remaining
two,
\[
  (bba,ab)\ \, \mbox{and} \ \, (abb,ba),
\]
which form an enantiomorphic (or mirror) pair of systems. The first
claim easily follows from~\cite[Prop.~4.6]{TAO} in conjunction with
the palindromicity of $(abb,ab)$, compare~\cite[Lemma~4.5]{TAO}, while
the remaining two are different from the previous four (as they
contain $aa$ while the others do not), with hulls that are not
reflection symmetric (they differ in the occurrence of $abbaa$ versus
$aabba$, for example).
	
Let us thus take a closer look at $\widetilde{\vrho} = (bba,ab)$,
where we can construct a fixed point of $\widetilde{\vrho}\,^2$ from
the legal seed $a|a$ via
\[
  a|a\, \longmapsto\, bba|bba \, \longmapsto \, ababbba|ababbba \,
  \longmapsto\cdots \longmapsto \widetilde{w}\,\
  \longmapsto\,\widetilde{\vrho}(\widetilde{w})\,\longmapsto\,
  \widetilde{w}
\]
with $\widetilde{w} = \dots a|a \dots $. Here, $\widetilde{w}$ and
$\widetilde{\vrho}(\widetilde{w})$ are equal to the left of the marker
but differ on the right of it in a way that will show up later in more
detail. In other words, $\widetilde{w}$ and
$\widetilde{\vrho}(\widetilde{w})$ form an asymptotic pair; see
\cite[Ch.~4]{TAO} or~\cite{PyFo}.
	
The step from here to the Delone sets works in complete analogy to our
initial example. One can work with the CPS from \eqref{eq:CPS}, and
the displacement matrix for $\widetilde{\vrho}$ reads
\[
  \widetilde{T}\,=\, \begin{pmatrix}
  \{2\} & \{0\} \\[2mm] \{0,1\} & \{\sqrt{2}\,\} \end{pmatrix}.
\]
Strictly speaking, we should start with the displacement matrix for
$\widetilde{\vrho}\,^2$, as we have to deal with a fixed point of
$\widetilde{\vrho}\,^2$. However, one then finds that both fixed
points, after $\star$-map, lead to the same contractive IFS. This
means that we may work with the IFS induced by $\widetilde{\vrho}$
instead, which is simpler and reads
\begin{equation}\label{eq:IFS2}
\begin{split}
  \widetilde{W}^{}_a \, & = \, \lambda^{\star} \widetilde{W}^{}_a +2\
  \cup \ \lambda^{\star} \widetilde{W}^{}_b \ts,  \\[1mm]
  \widetilde{W}^{}_b \, & = \, \lambda^{\star} \widetilde{W}^{}_a \
  \cup \ \lambda^{\star} \widetilde{W}^{}_a + 1 \ \cup \
  \lambda^{\star} \widetilde{W}_b - \sqrt{2} \ts .
\end{split}
\end{equation}
This IFS is again contractive, so it defines a unique pair of compact
sets $(\widetilde{W}^{}_a,\widetilde{W}^{}_b)$ of positive Lebesgue
measure that solve \eqref{eq:IFS2}. However, this time,
$\widetilde{W}^{}_{a}$ and $\widetilde{W}^{}_{b}$ are not intervals
(as the substitutions rule is not invertible \cite{BEIR07}), but
Cantorvals~\cite{BGM-Cant}; they are topologically regular sets with a
boundary of Hausdorff dimension
\[
  \operatorname{dim}^{}_{\mathrm{H}} (\partial \widetilde{W}^{}_a) \,
  = \, \operatorname{dim}^{}_{\mathrm{H}} (\partial
  \widetilde{W}^{}_b) \, = \,
  \myfrac{\log(x_{\max})}{\log(\lambda)}\, \approx \,
  0.89745\dots
\]
with $x_{\max}$ the largest root of $x^3-2x^2-1$. A
visualisation of the windows $W_a$ and $W_b$ is presented in Figure
\ref{fig:window_sm}. The dimension can be computed in various ways;
compare~\cite{BGG,FFIW,Bernd,ST}. The boundaries have zero
Lebesgue measure~\cite[Cor.~6.66]{Bernd} where $\widetilde{W}^{}_{a}$
and $\widetilde{W}^{}_{b}$ have no interior points in common, though
they share many boundary points. These boundary points, in particular,
distinguish the two different fixed points $\widetilde{w}$ and
$\widetilde{\vrho}(\widetilde{w})$.

\begin{figure}
\includegraphics[width=0.9\textwidth]{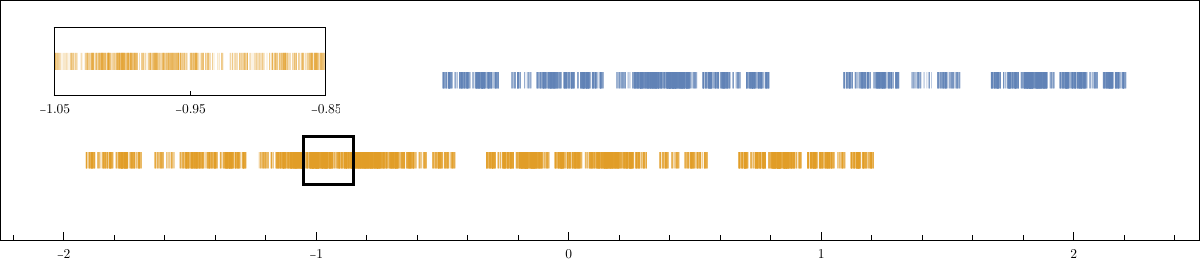}
\caption{The windows $\widetilde{W}^{}_a$ (blue/top) and
  $\widetilde{W}^{}_b$ (yellow/bottom) for the tiling given by
  $\widetilde{\vrho}$. The inlay shows a stretched view of the marked
  region. Note that both windows are subsets of $\RR$ and
  measure-theoretically disjoint, though this is almost impossible to
  illustrate due to the large Hausdorff dimension of their
  boundaries.}
  \label{fig:window_sm}
\end{figure}
	
The diffraction measure of
$\widetilde{\vL} = \widetilde{\vL}^{}_a \, \dot{\cup} \,
\widetilde{\vL}^{}_b$ with weights $\alpha,\beta \in \CC$, as above,
reads
\[
  \widehat{\gamma} \, = \, \sum_{k\in L^{\circledast}}\bigl|\alpha
  A^{}_{\widetilde{\vL}_a}(k) + \beta A^{}_{\widetilde{\vL}_b}(k)
  \bigr|^2 \delta_k
\]
with the same Fourier module $L^{\circledast}$ as before and the
non-zero FB coefficients read for all $k\in L^{\circledast}$ and
$i\in \{a,b\}$
\[
  A^{}_{\widetilde{\vL}^{}_i}(k) \, = \,
  \widetilde{H}^{}_{i}(k^{\star})
\]
with
$\widetilde{H}^{}_{i}(k^{}_{\inte}) =
\tfrac{\dens(\widetilde{\vL})}{\vol(\widetilde{W})} \
\widecheck{\bm{1}^{}_{\,\widetilde{W}^{}_i}}\,(k^{}_{\inte})$.  It is
difficult to calculate the Fourier transform of sets
like~$\widetilde{W}^{}_{i}$, which are \emph{Rauzy fractals}.
Fortunately, there exists a~method due to Baake and
Grimm~\cite{BG-Rauzy} based on a~cocycle approach. One defines the
\emph{internal Fourier cocycle}, which is a matrix cocycle induced by
the inflation as follows. First, consider the inverse Fourier
transform of the matrix of Dirac measures at positions given by the
entries of $\widetilde{T}^{\star}$. For
$k^{}_{\inte} \in \RR^{}_{\inte}$, the matrix elements are defined by
\[
  \underline{B}^{}_{ij}(k^{}_{\inte}) \, =\, \sum_{t \in
    \widetilde{T}^{}_{ij}}\ee^{2\pi\ii t^{\star} k^{}_{\inte}} ,
\]
which is abbreviated as $\underline{B}(k^{}_{\inte}) =
\widecheck{\delta_{\widetilde{T}^{\star}}}$. For $\widetilde{\vrho}$,
we have
\[
  \underline{B}(k^{}_{\inte}) \, = \, \begin{pmatrix}
    \ee^{4\pi\ii k^{}_{\inte}} & 1 \\[1mm]
    1+\ee^{2\pi \ii k^{}_{\inte}} & \ee^{-2\pi\ii
      k^{}_{\inte} \sqrt{2}} \end{pmatrix}\ts,
\]
where $\underline{B}(0)$ is the substitution matrix of
$\widetilde{\vrho}$.  Now, one defines the matrix cocycle for
$n\in\NN$ via
\[
  \underline{B}^{(n)}(k^{}_{\inte}) \, \defeq \,
  \underline{B}(k^{}_{\inte})\,\underline{B}(\lambda^{\star}k^{}_{\inte})\,
  \cdots\,
  \underline{B}\bigl((\lambda^{\star})^{n-1}k^{}_{\inte}\bigr),
\]
and, further, one considers the matrix function
$C(k^{}_{\inte})$
\begin{equation}\label{eq:matrix_limit}
  C(k^{}_{\inte}) \, \defeq \, \lim_{n\to\infty} \lambda^{-n}\,
  \underline{B}^{(n)}(k^{}_{\inte}).
\end{equation}
The function $C(k^{}_{\inte})$ is well defined and continuous,
as the sequence
$ \bigl(\lambda^{-n}\underline{B}^{(n)}(k^{}_{\inte})\bigr)_n$
converges compactly on $\RR$~\cite[Thm.~4.6]{BG-Rauzy}. Moreover, the
convergence of \eqref{eq:matrix_limit} is exponentially fast, which
makes it effectively computable to any desired precision. Note that
$C(k^{}_{\inte})$ is of rank smaller than or equal to $1$, so one
can represent it as
\[
  C(k^{}_{\inte}) \,= \, |\,c(k^{}_{\inte})\,\rangle
  \langle u | \ts ,
\]
with the left PF eigenvector $\langle u|$ from
\eqref{eq:eigenvectors}. It turns out that the vector of functions
$|\,c(k^{}_{\inte})\, \rangle$ has components
\[
  c_i(k^{}_{\inte}) \, = \, \eta\,
  \widecheck{\bm{1}^{}_{\,\widetilde{W}^{}_i}}\,(k^{}_{\inte})
  \qquad \mbox{for some } \ \eta>0,
\] 
which provides the desired quantities; see~\cite[Sec.~4]{BG-Rauzy}
for details.

\begin{figure}
\centering
\begin{subfigure}{.5\textwidth}
	\centering
	\includegraphics[width=.77\linewidth]{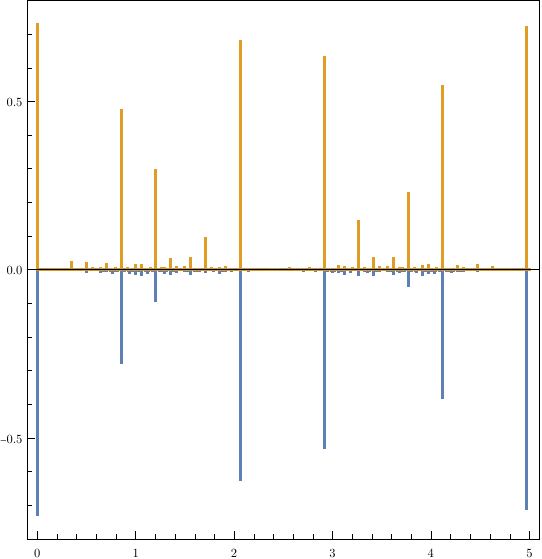}
	\caption{Equal weights $\alpha= \beta =1$.}
	\label{fig:DiffSM}
\end{subfigure}%
\begin{subfigure}{.5\textwidth}
	\centering
	\includegraphics[width=.77\linewidth]{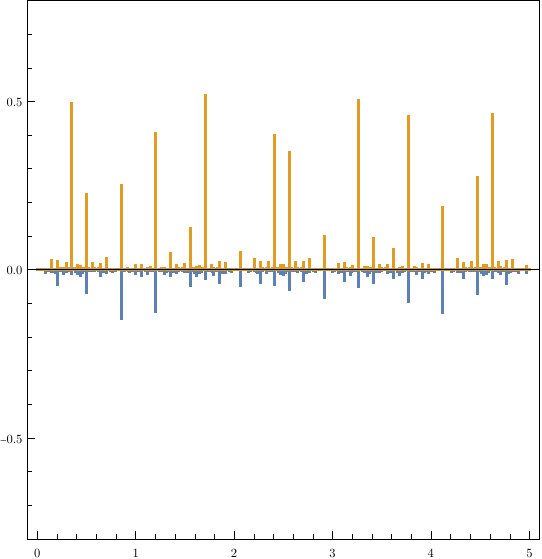}
	\caption{Weights $\alpha = \sqrt{2}$, $\beta =-1$.}
	\label{fig:DiffZeroCP}
\end{subfigure}
\caption{Diffraction images of the model sets $\vL$ and
  $\widetilde{\vL}$, with two different choices of weights. (A) shows
  the case of equal weights, while (B) depicts the diffraction when
  the weights are chosen such that the central peak at $0$
  vanishes. Both pictures show all Bragg peaks in positions
  $k\in L^{\circledast}\cap[0,5]$ with intensity~$\geqslant
  10^{-3}$.} \label{fig:DiffSMboth}
\end{figure}
	
\begin{figure}
\centering
\begin{subfigure}[b]{0.80\textwidth}
	\includegraphics[width=1\linewidth]{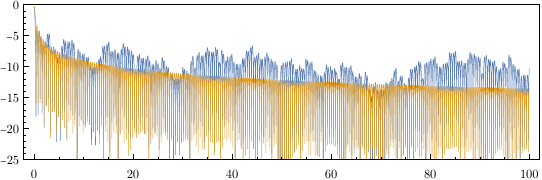}
	\caption{}
	\label{fig:LogInt1} 
\end{subfigure}
\begin{subfigure}[b]{0.80\textwidth}
	\includegraphics[width=1\linewidth]{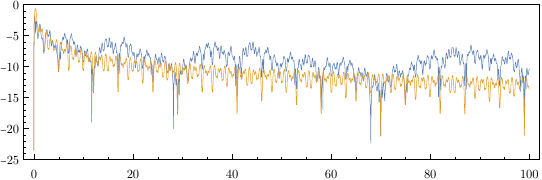}
	\caption{}
	\label{fig:LogInt2}
\end{subfigure}
\caption[Two logarithms of the intensities]{The functions
  $\log \bigl(|\widetilde{H}(k^{}_{\inte})k^{}_{\inte}|^2\bigr)$
  (blue) and $\log \bigl(|H(k^{}_{\inte})|^2\bigr)$ (yellow) and their
  values for $k\in [0,100]$. The values are cropped at -25 for
  presentation purposes. Graph~(A) shows the case with equal weights
  ($\alpha = \beta = 1$), whereas graph (B) depicts the case with
  $\alpha = \sqrt{2}$ and $\beta = -1$. Both graphs illustrate the
  slower decay of the intensities for the window with a boundary of
  non-trivial Hausdorff dimension.}
\end{figure}

Figure~\ref{fig:intensities} compares the continuous counterparts of
the intensity functions $|H(k^{}_{\inte})|^2$ and
$|\widetilde{H}(k^{}_{\inte})|^2$ with equal weights $\alpha= \beta=1$
and Figure~\ref{fig:DiffSM} shows the diffraction of both
structures. On the level of intensities, one can recognise that the
decay of $I_{\widetilde{\vL}}(k)$ is slower than that of
$I^{}_{\vL}(k)$, which supports the conjectured non-trivial relation
between the boundary dimension of the window and the decay rate of the
diffraction measure~\cite{LGJJ93}. Note that this point is a subtle
one, as it emerges from the \emph{boundary dimension} of a set that
itself has full dimension, and is thus not covered by the usual decay
estimates as given in \cite[Thm.~3.10]{Mat}. To illustrate the
significantly slower decay, we include plots of the logarithms of the
intensities on a larger scale in Figure~\ref{fig:LogInt1}.  If the
weights are chosen as $\alpha = \sqrt{2}$ and $\beta = -1$, the
central peak vanishes. Figure~\ref{fig:intensitiesZeroCP} shows the
diffraction intensities in such a case. Again, the slower decay can be
observed as in the previous case (compare Figure~\ref{fig:LogInt2}).
	
As above, we can recover the spectrum of the original symbolic
sequence $\widetilde{w}$ by a reprojection. Since we are using the
same CPS, the reprojection from \eqref{eq:reprojection} still applies,
and Eq.~\eqref{eq:FBdeformed} remains valid. The only difference is
the method for obtaining the Fourier transform of the window. If the
weights are both equal, one ends up with the lattice
$(4-2\sqrt{2}\,)\ZZ$ as before. Figure~\ref{fig:DiffDefSM} shows the
diffraction of the deformed model sets $\vL'$ and $\widetilde{\vL}'$
with weights chosen so that the central peak vanishes.

At this point, we hope that the reader is well prepared to embark on
the analogous programme in two dimensions, which we require to tackle
the Hat and the Spectre tilings.

\section{CAPs and Hats (and their relatives)}\label{sec:Hats}
 
Recall that the Hat tiling, discovered by David Smith and his
coauthors~\cite{Hat}, is an aperiodic tiling of the plane using a
disk-like prototile and its flipped version. Thus, it provides a
partial solution to the monotile problem. In fact, there exists a
continuum of monotiles related to the Hat, which have become known as
the Hat family of tilings.  Soon after this discovery, Baake,
G\"{a}hler and Sadun~\cite{BGS} showed that all elements of this
family give rise to topologically conjugate dynamical systems (up to
scale and rotations). In the topological conjugacy class, there exists
a self-similar relative of the Hat tiling called the CAP
tiling. Further, they proved that the CAP tiling is MLD to a Euclidean
model set and showed that the Hat tiling is a reprojection of the CAP
tiling, as is every other member of the Hat family (after choosing an
appropriate scale and orientation).
	
We aim to provide more details on these connections. In particular, we
derive the explicit reprojection and deformation mappings. Then, using
the cocycle approach, we calculate the diffraction and dynamical
spectrum of the CAP tiling and, consequently, the spectra of the Hat
tiling. Here, the cocycle method is required for this system because
its window has some parts with fractal boundary. In what follows, we
mimic the strategy of our one-dimensional guiding example from
Section~\ref{sec:example}. Where possible, we keep an informal style
and notation for better readability.

\subsection{The embedding of the CAP tiling}
	
Recall that the CAP tiling is built from 4 prototiles, each of which
appears in 6 orientations. Therefore, there are altogether 24
prototiles up to translations. Figure~\ref{fig:CAPsubst} shows the
substitution rule that can be turned into a proper stone inflation
rule with fractiles~\cite[Fig.~3]{BGS} and inflation factor $\tau^2$.
\begin{figure}
\centering
\begin{subfigure}{.75\textwidth}
	\centering
	\includegraphics[width=.65\linewidth]{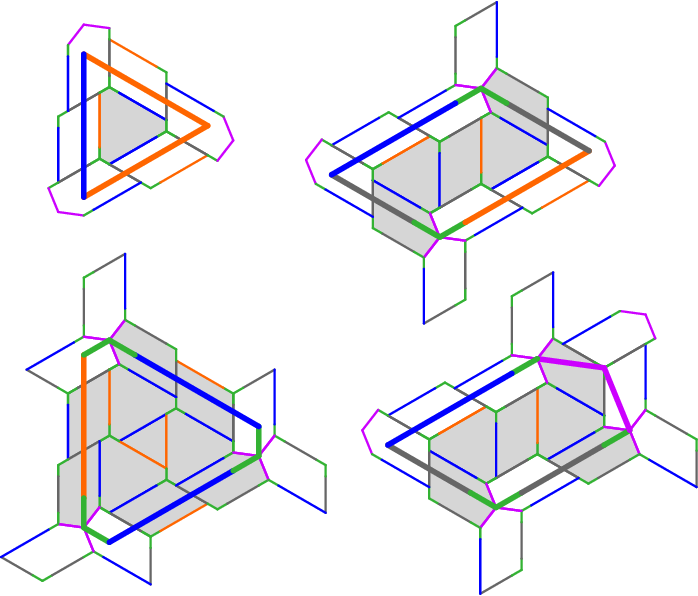}
\caption{}
	\label{fig:CAPsubst}
\end{subfigure}%
\begin{subfigure}{.25\textwidth}
	\centering
	\includegraphics[width=0.33\linewidth]{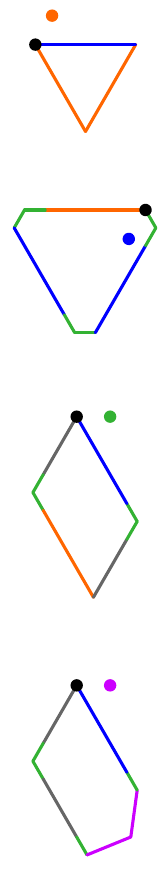}
\caption{}
	\label{fig:CAPcp}
	\end{subfigure}
        \caption{Panel (A) shows the inflation rule for the CAP
          tiling. Note that the grey polygons constitute the level-1
          supertile, whereas the white ones are uniquely determined by
          the grey patch so that the inflation rule is border
          forcing. Panel (B) shows the position of the control point
          for every tile; see~\cite{BGS} for details.}
		\label{fig:CAPsubstAll}
\end{figure}
	
Note that the control points of the tiles, as shown in
Figure~\ref{fig:CAPcp}, do not all lie inside the tiles. They are
chosen so that they form a single orbit under the translation action
of the return module. In~\cite{BGS}, it was derived that the return
module is the principal ideal in the ring $\ZZ[\tau,\xi]$ generated by
$(1-\xi)(\tau - \xi)$, with $\xi$ a primitive $6^{\mathrm{th}}$ root
of unity.  The return module (and hence the ideal
$(1-\xi)(\tau - \xi)\,\ZZ[\tau,\xi]$, which equals
$\tau^2(1-\xi)(\tau - \xi)\,\ZZ[\tau,\xi] = (3\tau
+2-\xi)\,\ZZ[\tau,\xi]$ as $\tau$ is a unit) is generated by 4
elements,
\begin{equation}\label{eq:generators}
\begin{split}
  \bm{u}^{}_1 & \,=\, 3\tau + 2- \xi, \\
  \bm{u}^{}_2 & \,=\, 2\tau+1-\tau\xi+\xi,\\
  \bm{u}^{}_3 & \,=\, 1+3\tau\xi+\xi \, = \, \xi \bm{u}^{}_1, \\
  \bm{u}^{}_4 & \,=\, \tau-1 + \tau\xi + 2\xi  \, = \, \xi \bm{u}^{}_2. \\
\end{split}
\end{equation}
	
The module $\cR^{}_{_{\CAP}} = \ZZ\bm{u}^{}_1 \oplus \cdots
\oplus \ZZ\bm{u}^{}_4$ can be lifted via a Minkowski embedding into
$\RR^4 \simeq \CC^2$ to obtain the lattice
$\cL' = \left\{(u, \, u^{\star}) \, : \, u \in
  \cR^{}_{_{\CAP}} \right\}$, with a $\star$-map that
follows from the Minkowski embedding. Let us explain this in more
detail. The generating vectors of $\cL'$ constitute the columns of a
matrix $V^{}_{\RR} \in \Mat(\RR,4)$ or
$V^{}_{\CC} \in \CC^{2\times4}$, which read
\begin{equation}\label{eq:latticeR}
  V^{}_{\RR} \, = \, \myfrac{1}{4}\begin{pmatrix}
    12+6\sqrt{5} & 9+3\sqrt{5} & 9+3\sqrt{5} & 3+3\sqrt{5}\\[1mm] 
    -2\sqrt{3} & \sqrt{3}-\sqrt{15} & 5\sqrt{3}+ 3\sqrt{15} &
    5\sqrt{3}+\sqrt{15} \\[1mm] 
    12-6\sqrt{5} & 9-3\sqrt{5} & 9-3\sqrt{5} & 3-3\sqrt{5}\\[1mm] 
    2\sqrt{3} & -\sqrt{3}+\sqrt{15} & -5\sqrt{3}+ 3\sqrt{15} &
    -5\sqrt{3}+\sqrt{15} \end{pmatrix}
\end{equation}
and
\begin{equation}\label{eq:latticeC}
  V^{}_{\CC} \, = \, \begin{pmatrix}
    3\tau+2-\xi & 2\tau+1-\tau\xi+\xi & 1+3\tau\xi+\xi &
    \tau-1 + \tau\xi + 2\xi \\
    -3\tau+4+\xi & -\tau +3 -\tau\xi & -3\tau +5 +3\tau\xi-4\xi &
    -2\tau+3+\tau\xi -3\xi \end{pmatrix},
\end{equation}
respectively.  We will tactically switch between the real and complex
descriptions. Using the lattice $\cL'$, we obtain the CPS
\begin{equation}\label{eq:CPSHat}
\renewcommand{\arraystretch}{1.2}
\begin{array}{ccccc@{}l}
  \RR^2\simeq \CC & \hspace*{-7ex}\xleftarrow{\quad \ \pi \ \quad }
  & \RR^2 \nts\nts \times \nts\nts \RR^2_{\inte}  \simeq \CC
    \nts\nts \times \nts\nts \CC^{}_{\inte}
  & \xrightarrow{\quad \pi^{}_{\text{int}} \quad }
  & \hspace*{-7ex} \RR^2_{\inte}  \simeq \CC^{}_{\inte}  & \\
  \cup & & \cup & & \hspace*{-8ex}\cup
  & \hspace*{-10ex}
  \raisebox{1pt}{\text{\scriptsize dense}} \\
  \pi (\cL') &\hspace*{-7ex} \xleftarrow{\quad \ts 1:1 \, \quad }
  & \cL' & \xrightarrow{ \qquad \ \quad } & \hspace*{-7ex}
               \pi^{}_{\text{int}} (\cL') & \\
  \| & & & & \hspace*{-8ex} \| & \\ 
  L = (3\tau{+}2{-}\xi)\,\ZZ[\tau,\xi]
     & \multicolumn{3}{c}{\hspace*{-8ex}
      \xrightarrow{\qquad\quad\quad \ \ \star \ \quad\quad\qquad}}
  & \hspace*{-7ex}  {L_{}}^{\star\nts} = (-3\tau{+}4{+}\xi)\,\ZZ[\tau,\xi]
  &  \end{array}
\renewcommand{\arraystretch}{1}
\end{equation}
with the star map, in complex formulation, being given by
\[
  (a+b\tau+c\xi+d\tau\xi)^{\star} \, = \, a+b+c+d - (b+d)\tau
  -(c+d)\xi + d\tau \xi, \quad \mbox{for} \ a,b,c,d \in\QQ \ts .
\]
This is the Galois isomorphism of $\QQ(\tau,\xi)$ that fixes
$\QQ(\sqrt{-15}\ts)$ but no other subfield of $\QQ(\tau,\xi)$. In
other words, the star map is a composition of the non-trivial
algebraic conjugations in $\QQ(\tau)$ and $\QQ(\xi)$.
	
As in the guiding example, we construct the set-valued displacement
matrix $T^{}_{_\CAP}$, which is $24$-dimensional in this case. It has
the block structure
\begin{equation}\label{eq:TCAP}
  T^{}_{_\CAP} \, = \, 
  \begin{pmatrix}
    \vn & T^{}_{12} & \vn & \vn \\
    T^{}_{21} & T^{}_{22} & T^{}_{23} & T^{}_{24} \\
    \vn & T^{}_{32} & T^{}_{33} & T^{}_{34} \\
    \vn & T^{}_{42} & T^{}_{43} & T^{}_{44} \end{pmatrix},
\end{equation}
where each entry represents a $6{\times}6$ matrix and $\vn$ is
the $6{\times}6$ block of empty sets. The remaining matrices
$T^{}_{ij}$ are listed in Appendix~A.  As before,
$T^{}_{_\CAP}$ determines an IFS with \emph{linear} scaling
factor $\tau^{-1}$ on
$\bigl(\cK\RR^2_{\inte}\bigr)^{24}$, whose solution
provides the windows for a model set which (possibly after removing
points of density 0) is MLD with the CAP tiling, as discussed in
detail in~\cite{BGS}. The total window and its subdivisions are shown
in Figure~\ref{fig:CAPwindow}.
	
\begin{prop}[\cite{BGS}]
  The CAP tiling is MLD with a Euclidean model set derived from the
  CPS~\eqref{eq:CPSHat} and the hexagonal window with fractal
  boundaries shown in Figure~$\ref{fig:CAPwindow}$. \qed
\end{prop}

\begin{figure}[ht]
  \includegraphics[width=0.6\textwidth]{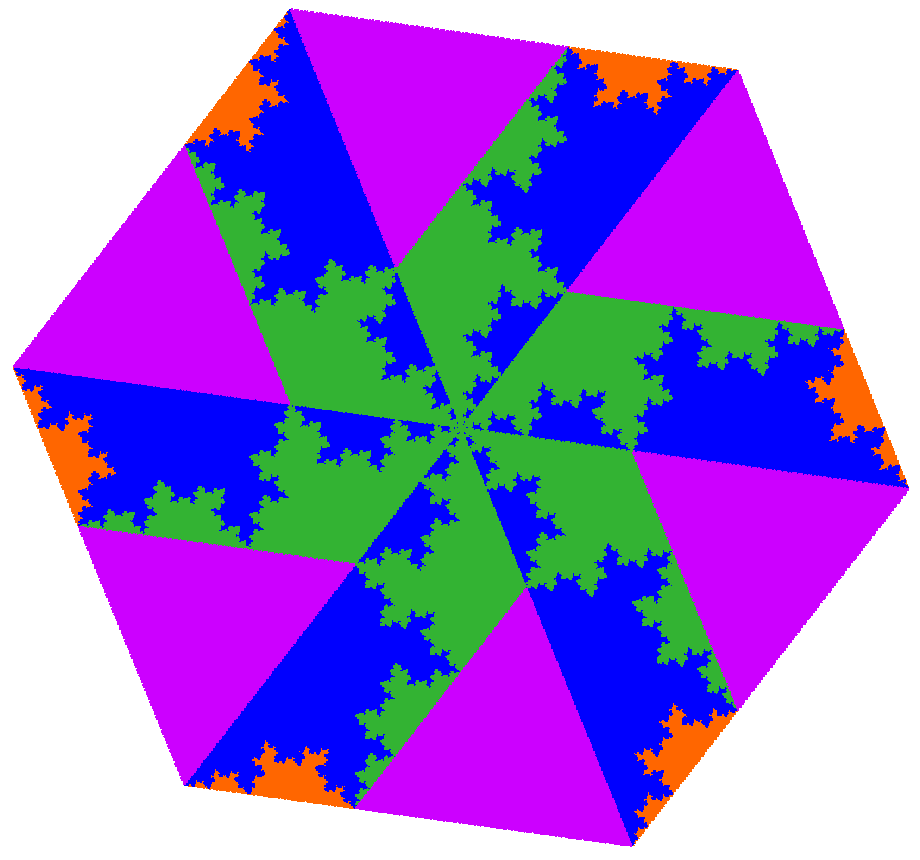}
  \caption{The window for the control points of the CAP tiling. The
    four different colours correspond to the four different shapes of
    prototiles. Note that there are two types of boundaries in the
    interior of the hexagon.}
  \label{fig:CAPwindow}
\end{figure}
	
The total window is a hexagon rotated by
$\arccos\bigl(\tfrac{3 +\sqrt{5}}{4} \sqrt{\frac{3}{3
    \sqrt{5}+7}}\bigr) \approx 52.24^{\circ}$ relative to the window
shown in~\cite{BGS}, which is due to our choice of a basis. For the
total window, it is possible to compute its Fourier transform
explicitly. Nevertheless, since some of the subwindows have fractal
boundary parts with Hausdorff \cite{MW} dimension
$\tfrac{\log(2+\sqrt{3})}{2\log(\tau)}\approx 1.3683764$, as shown in
\cite{BGS}, one has to use the cocycle approach to obtain the FB
coefficients for a general choice of weights.  Thus, for all
$k^{}_{\inte}\in\RR^2_{\inte} \simeq
\CC^{}_{\inte}$, we define the internal Fourier matrix
$\underline{B}(k^{}_{\inte})$ as
\begin{equation}\label{eq:CAPfourier}
  \underline{B}(k^{}_{\inte})^{}_{ij} \, = \, \sum_{t \in T^{}_{ij}}
  \ee^{2\pi\ii \langle t^{\star} | k^{}_{\inte} \rangle }
  \qquad  \mbox{with} \ T = T^{}_{_{\CAP}}, 
\end{equation}
where $\langle \, \cdot \, | \, \cdot \, \rangle$ denotes the standard
scalar product in $\RR^2_{\inte}$.
	
Since the tiles come in six orientations and the inflation rule
respects the orientation, the displacement matrix as well as the
(internal) Fourier matrix must reflect this fact. In particular, we
have the following symmetry properties of the matrices
$T^{}_{_{\CAP}}$ and $\underline{B}(k^{}_{\inte})$.
	
\begin{lemma}
  For the displacement matrix\/ $T^{}_{_\CAP}$ and the
  corresponding internal Fourier matrix \eqref{eq:CAPfourier}, one has
  the symmetry relations
\[
  S^{\top}  \underline{B}(k^{}_{\inte}) \, S \,=\,
  \underline{B}(\xi k^{}_{\inte}) \quad \mbox{and} \quad
  S\,T^{}_{_\CAP}\, S^{\top} \,= \, \xi\, T^{}_{_\CAP} \ts ,
\]
with the permutation matrix\/
$S = \mathmybb{1}^{}_4\, \otimes\, C \in \Mat\,(24,\ZZ)$, where\/
$\mathmybb{1}^{}_4$ is the \emph{4D} identity matrix, $\otimes$ is the
Kronecker product and $C$~stands for the companion matrix of\/ $X^6-1$,
\[
  C \,=\, \left(\begin{smallmatrix} 0&0&0&0&0&1\\1&0&0&0&0&0
      \\0&1&0&0&0&0\\ 0&0&1&0&0&0 \\ 0&0&0&1&0&0 \\ 0&0&0&0&1&0
    \end{smallmatrix}\right).
\] 
\end{lemma}

\begin{proof}
  The claims follow from the structure of the matrices and from an
  explicit computation using
  $\langle t^{\star}\,|\,\xi k^{}_{\inte}\rangle = \langle
  \xi^{-1} t^{\star} \,|\, k^{}_{\inte}\rangle = \langle (\xi
  t)^{\star} \,|\, k^{}_{\inte} \rangle$.
\end{proof}
	
This symmetry relation provides a good consistency check for numerical
calculations, which can be implemented easily. It detects the position
of eventual mistake, in particular in $\underline{B}(k^{}_{\inte})$.
	
Now, we have all the ingredients needed to discuss the spectral
properties of the CAP tiling. In~\cite[Lemma 10]{BGS}, it was proved
that the CAP tiling is pure-point diffractive. The authors also
derived the Fourier module
\[
  L^{\circledast}_{_{\CAP}} \, = \,
  \myfrac{(1+\xi)(\tau-\xi)\ts\ii}{3\sqrt{15}}\, \ZZ[\tau,\xi],
\]
which agrees with the dynamical spectrum of the CAP tiling dynamical
system. For the diffraction intensities, it is sufficient to compute
the Fourier transform of the windows using the cocycle. The result is
shown in Figure \ref{fig:diffCAP}. To obtain a first (and approximate)
impression, one can replace the hexagonal total window with a circular
one (of the same area) and use the explicit Fourier transform of a
circle in terms of Bessel functions $J_{\nu}$;
see~\cite[Rem.~9.15]{TAO}. Note that the uncoloured CAP point set is
\emph{not} MLD with the coloured one, so it is insufficient to only
work with the total window, as we demonstrate in Figures
\ref{fig:CAPdiff} and \ref{fig:CAPdiffZeroCP}.

\begin{figure}[ht]
\begin{subfigure}{.5\textwidth}
	\centering
	\includegraphics[width=0.9\textwidth]{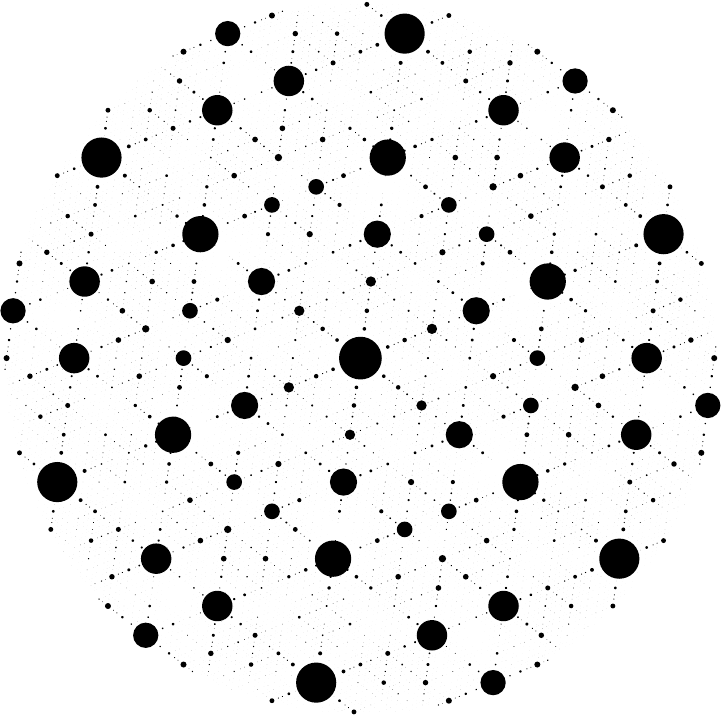}
	\caption{}
	\label{fig:CAPdiff}
\end{subfigure}%
\begin{subfigure}{.5\textwidth}
	\centering
	\includegraphics[width=0.9\linewidth]{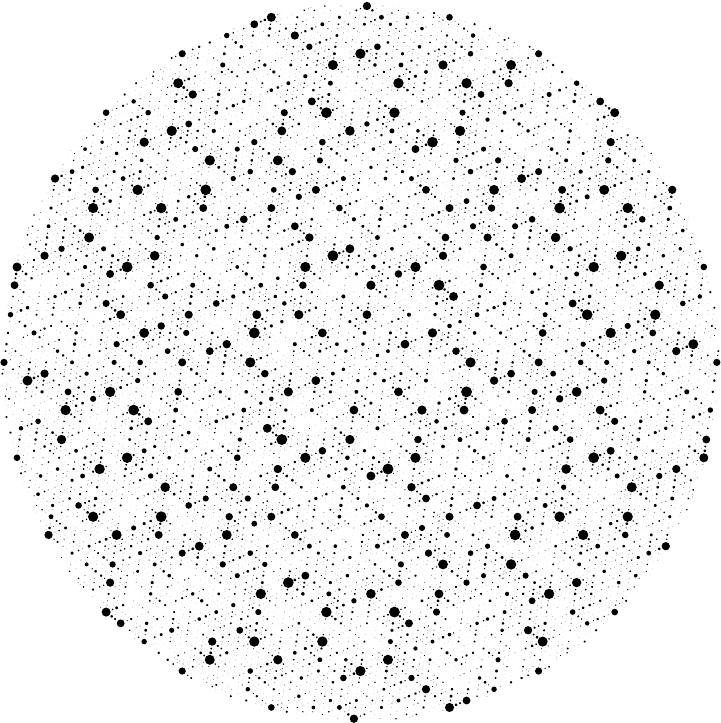}
	\caption{}
	\label{fig:CAPdiffZeroCP}
\end{subfigure}
\caption[LoF entry]{Diffraction pattern of the CAP tiling in the
  centred ball of radius 0.6, with two different sets of weights. The
  radii of the black disks are proportional to their
  intensities. Panel A shows the case with equal weights, where the
  intensity of the central peak equals
  $\dens(\vL_{_\CAP})^2 = \tfrac{1}{75 \tau^4}\approx 0.001945$. The
  diffraction exhibits the sixfold symmetry and mirror symmetries as
  well. The second brightest peaks are located at
  $\tfrac{1}{30}\bigl(\sqrt{5} + \ii \sqrt{3}\,(5+2\sqrt{5}\,) \bigr)$
  and all its $\xi$-multiples.
			
  Figure (B) shows the diffraction pattern for weights
  $(0,\, 0,\, \tau,\, -1)$, chosen so that the central peak
  vanishes. The Bragg peaks are not as bright as in Figure (A), so in
  order to obtain a more visible pattern, we magnified all intensities
  by a factor of 4. One can still see the sixfold symmetry, but the
  mirror symmetry is broken, which demonstrates the chiral nature of
  the CAP tiling (as manifest in the window in
  Figure~\ref{fig:CAPwindow}). To create the figure, 15 iterations of
  the cocycle were used.} \label{fig:diffCAP}
\end{figure}

\subsection{Shape changes --- from CAPs to Hats}

Let us now explain the reprojection of the CAP tiling that results in
the Hat tiling. To be more precise, we start with the (coloured) set
of control points, which is MLD to the CAP tiling, and we modify it to
a different coloured point set, which is MLD to the Hat tiling. This
is then a deformed model set in the sense of~\cite{BL05,BerDun00},
which can now be used, as outlined in~\cite{BGS}. Moreover, the
authors also derived the return module for the Hat tiling, which reads
\begin{equation}\label{eq:retModHat}
  \myfrac{\sqrt{5}}{4} (1+\xi)(\tau-\xi)^3 \,\ZZ[\xi] \ts .
\end{equation}
It is a scaled and rotated triangular lattice, thus is of rank $2$ ---
in comparison to $\cR^{}_{_{\CAP}}$, which is of rank $4$. The
generators for \eqref{eq:retModHat} can be chosen as
\begin{equation}\label{eq:generatorsHat}
\begin{split}
  \bm{v}^{}_1 & \,=\, \myfrac{\sqrt{5}}{4}\xi (1+\xi)(\tau-\xi)^3, \\
  \bm{v}^{}_2 & \,=\, \myfrac{\sqrt{5}}{4}\xi^2 (1+\xi)(\tau-\xi)^3. \\
\end{split}
\end{equation}

Let us derive the reprojection, and the deformation mapping from the
CAP to the Hat tiling. First, we identify the generators of the return
module of the CAP tiling (and due to our choice of the control points,
we do not need to pay attention to the type of the points!), for
example as indicated in Figure~\ref{fig:CAPgen}. Then, we find the
corresponding patch in the Hat tiling and identify the same return
vectors as shown in Figure~\ref{fig:Hatgen}. The reprojection is
chosen so that all control points lie on the anti-Hats (meaning the
reflected Hats), and their colouring determines the
neighbourhood. Thus, the coloured point set is MLD with the Hat
tiling.

\begin{figure}
\centering
\begin{subfigure}{.5\textwidth}
	\centering
	\include{CAPgen}
           \vspace*{-10mm}
	\caption{}
	\label{fig:CAPgen}
\end{subfigure}%
\begin{subfigure}{.5\textwidth}
	\centering
	\includegraphics[width=.8\linewidth]{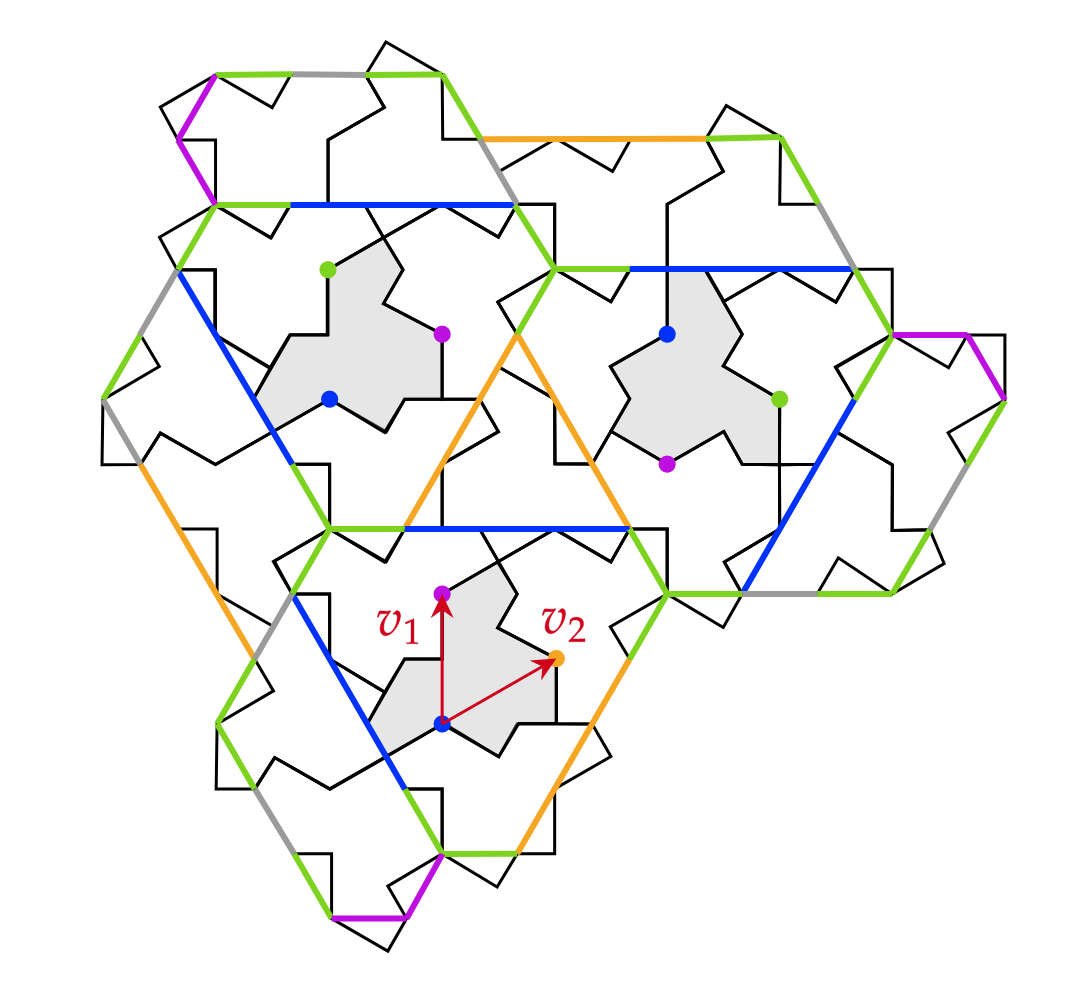}
	\caption{}
	\label{fig:Hatgen}
\end{subfigure}
\caption{Patch of the CAP tiling with its control points (A). The 4
  generators of the return module are indicated. Panel (B) shows the
  deformed version of the same patch together with the control points
  and the underlying Hat tiling. Note that one has to deform the tiles
  first, then identify the new return module, which is a triangular
  lattice in this case. Then, one can \emph{decide} how to choose the
  control points accordingly. In our case, we decided to choose the
  control point so that they all lie on the anti-Hats as shown in
  Figure (B). We also include the generators of the return module of
  the Hat tiling and indicate the generators by arrows.}
	\label{fig:generators}
\end{figure}

Now, the reprojection map acts on the level of the generators of the
return modules as
\begin{equation*}\label{eq:reprojModules}
  \begin{split}
    \bm{u}^{}_1 & \longmapsto \ 3\bm{v}^{}_{1} - \bm{v}^{}_{2}, \\
    \bm{u}^{}_2 & \longmapsto \ \bm{v}^{}_{1}, \\
    \bm{u}^{}_3 & \longmapsto \ \bm{v}_{1}+ 2\bm{v}_{2} , \\
    \bm{u}^{}_4 & \longmapsto \ \bm{v}_{2}. \\
\end{split}
\end{equation*} 
This determines the entire reprojection. As in our guiding example, we
can thus employ the matrix description via the reprojection matrix
$\pi' \in \RR^{2\times 4}\simeq \CC^{1\times 2}$, acting on the
lattice generators as
\[
   \pi'\,V^{}_{\CC} \, = \, \bigl(\begin{matrix}
     3\bm{v}^{}_{1} - \bm{v}^{}_{2} & \bm{v}^{}_{1} &
     \bm{v}_{1} + 2\bm{v}_{2}  & \bm{v}_{2} \end{matrix}\bigr),
\]
or the real version
\begin{equation}\label{eq:reprojectionCAP}
  \pi'\,V^{}_{\RR} \, = \, \myfrac{1}{16} \begin{pmatrix}
    15+45\sqrt{5} & 18\sqrt{5} & 45+9\sqrt{5} & 15+9\sqrt{5}\\[1mm]
    -25\sqrt{3}+9\sqrt{15} & -10\sqrt{3} & -5\sqrt{3}+27\sqrt{15} &
    -5\sqrt{3}+9\sqrt{15} \end{pmatrix}.
\end{equation}
Since $V^{}_{\RR}$ is an invertible matrix, one can multiply
\eqref{eq:reprojectionCAP} from the right by $V^{-1}_{\RR}$ to obtain
\[
  \pi' \, = \, \begin{pmatrix}
 	1 & 0 & -\frac{11}{16} & \frac{3\sqrt{15}}{16} \\[2mm]
 	0 & 1 & \frac{3\sqrt{15}}{16} & \frac{11}{16}  \end{pmatrix}.
\]
Since the reprojection can be considered as a special case of a
deformation of a model set, we have
$\bm{v}^{}_{i} \, =\, \bm{u}^{}_{i} + D\bm{u}^{\star}_{i}$, for
$i\in \{1,2,3,4\}$, where $D$ is the desired deformation mapping form
$\RR^2_{\inte}$ to $\RR^2$. Again, this can be rewritten compactly
using the matrices and for the real version, then giving
$\pi' \,=\,\pi + D \pi^{}_{\inte}$.  The projections with respect to
the standard basis read
$ \pi = \bigl( \begin{smallmatrix} 1 & 0 & 0 & 0 \\0 & 1& 0 & 0
\end{smallmatrix} \bigr)$ and
$ \pi^{}_{\inte} = \bigl(\begin{smallmatrix} 0 & 0 & 1 & 0 \\0 & 0 & 0
  & 1
\end{smallmatrix}\bigr)$, so 
\[
  D \,= \, \myfrac{1}{16} \begin{pmatrix}
    -11 & 3\sqrt{15} \\[1mm]
    3\sqrt{15} & 11 \end{pmatrix}.
\]
We summarise the above derivation as follows.
\begin{theorem}
  The set of control points of the Hat tiling is a deformed model set
  obtained from the control points of the CAP tiling using
\[
  D \,= \, \myfrac{1}{16} \begin{pmatrix}
    -11 & 3\sqrt{15} \\[1mm]
    3\sqrt{15} & 11  \end{pmatrix}
\]
as the linear deformation mapping. Moreover, the Hat control points
are a reprojection of the CAP tiling control points using the
projection\/ $\pi' \,=\,\pi + D \pi^{}_{\inte}$. \qed
\end{theorem}

The matrix $D$ allows the computation of the diffraction of the Hat
tiling from the amplitude functions of the CAP tiling. It plays a role
similar to the scaling factor $\lambda^{-2}$ in our guiding example,
now with $D: \RR^{2}_{\inte} \to \RR^2$.

Let us now move to the spectral properties of the Hat tiling. Due to
the topological conjugacy, its Fourier module is the same as that of
the CAP tiling, namely $L^{\circledast}_{_{\CAP}}$. It (of course)
contains the dual of the return module of the Hat tiling, which is
\begin{equation}\label{eq:dualModuleHat}
  \myfrac{(1+\xi)(\tau-\xi)^3 \, \ii}{3\sqrt{15}}\, \ZZ[\xi] \ts . 	
\end{equation}
$L^{\circledast}_{_{\CAP}}$ gives the dynamical spectrum of the Hat
tiling.

Further, one has an additional similarity to the one-dimensional
guiding example. As already suggested by the return module
\eqref{eq:retModHat}, the set of control points of the Hat tiling
forms a~subset of the triangular lattice. It follows from
\cite[Thm.~10.3]{TAO} that the corresponding diffraction measure is
lattice-periodic. The lattice of periods is the dual of the underlying
lattice. In our case, it is given by \eqref{eq:dualModuleHat}.

Since we already know the dynamical spectrum of the Hat tiling, we can
proceed to compute the FB coefficients. Suppose that tiles of type $i$
come with weight $\alpha^{}_{i}\in\CC$. Then, the FB coefficients
vanish for $k\notin L^{\circledast}_{_{\CAP}}$, while the remaining
ones are given via the inverse Fourier transform of the windows as
\begin{equation}\label{eq:FBdefHat}
  A^{}_{_{\mathrm{Hat}}}(k) \, = \,  H_{_{\CAP}}(k^{\star} - D^{\top} k),
\end{equation}
with
\[
  H_{_{\CAP}}(k^{}_{\inte}) \, = \,
  \frac{\dens\bigl(\vL^{}_{_{\CAP}}\bigr)}{\vol\bigl(
    W^{}_{_{\CAP}} \bigr)}\sum_{i}\alpha^{}_{i} \ts
  \widecheck{\bm{1}^{}_{W_{_{\CAP,i}}}}\bigl(k^{}_{\inte}\bigr),
\]
where $W_{_{\CAP,i}}$ stands for the part of the window from
Fig.~\ref{fig:CAPwindow} corresponding to points of type~$i$.

Note that the set of arguments
$\{k^{\star} - D^{\top} k \, : \, k\in L^{\circledast} \}$ forms a
lattice in $\RR^2_{\inte}$ --- in contrast to
$L^{\circledast}_{_{\CAP}}$ itself, which is a dense subset of
$\RR^2$. The lattice is $\tfrac{1}{12}(\tau-2+3\xi)\ts\ZZ[\xi]$ and
provides additional insight into the periodic nature of the
diffraction measure.  Figure~\ref{fig:HatDiff_total} shows the
diffraction spectra of the Hat tiling with two different sets of
weights, together with a~fundamental domain and generators of the
lattice of periods.

\begin{figure}
\begin{subfigure}{.5\textwidth}
	\centering
	\includegraphics[width=0.9\linewidth]{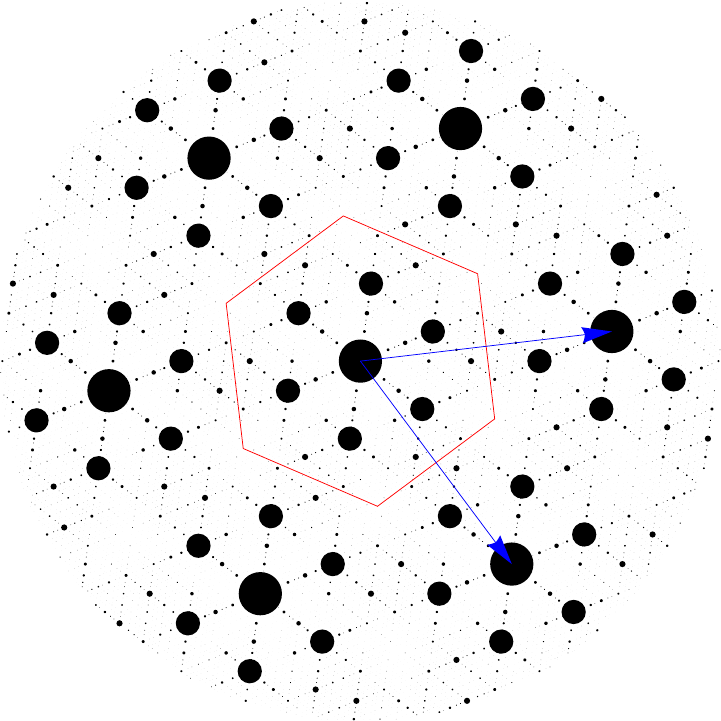}
	\caption{}
	\label{fig:HatDiff}
\end{subfigure}%
\begin{subfigure}{.5\textwidth}
	\centering
	\includegraphics[width=0.9\linewidth]{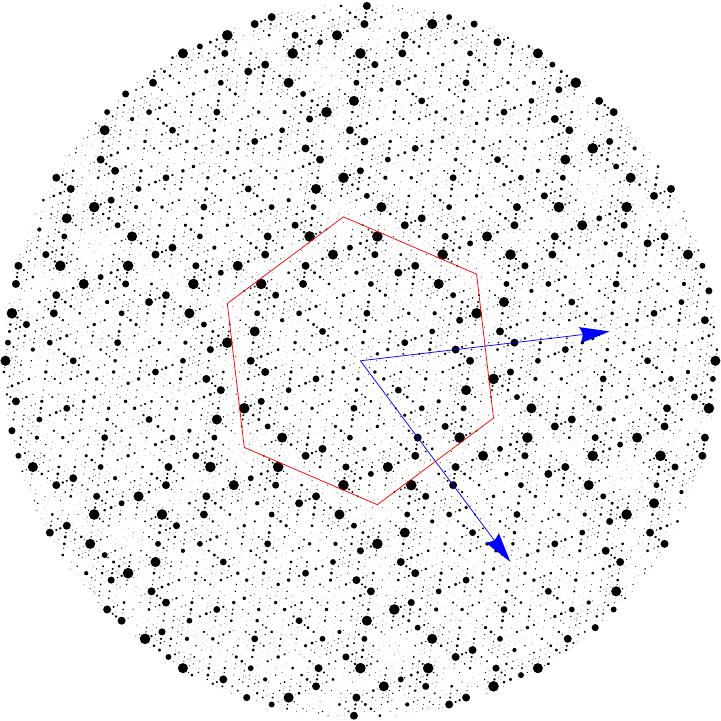}
	\caption{}
	\label{fig:HatDiffZeroCP}
\end{subfigure}
\caption{Diffraction of the Hat tiling with equal weights (A), and
  with weights $(0,\, 0,\, \tau, -1)$ chosen such that the central
  peak vanishes (B). Both pictures show a~fundamental domain (hexagon)
  of the support of the diffraction measure, which is lattice-periodic
  with periods \eqref{eq:dualModuleHat}. Two fundamental periods are
  indicated by the blue arrows. The intensities of the Bragg peaks are
  proportional to the area of the disks, and the intensity of the
  brightest one in (A) is $\tfrac{1}{75 \tau^4}\approx 0.001945$. In
  both cases, one has sixfold symmetry, while the absence of any
  mirror symmetry can be seen in the fundamental domain. $15$
  iterations of the cocycle were used in the computation. }
	\label{fig:HatDiff_total}
\end{figure}

The formula for the FB coefficients \eqref{eq:FBdefHat} holds for the
entire class of deformations, among them all affine ones. The proof
mimics the one given for the one-dimensional silver mean case in
\eqref{eq:FBdeformed}. We state it as a theorem for linear maps (the
translation part of affine mappings adds an additional phase factor,
which does not play a role for the intensities); it can also be found
implicitly in a slightly different form in~\cite[Thm.~2.6]{BerDun00}.

\begin{theorem}
  Let\/ $\vL\subset \RR^n$ be a model set arising from a Euclidean
  CPS\/ $(\RR^n, \, \RR^n_{\inte}, \, \cL)$ with window\/ $W$, and
  let\/ $\vartheta: \RR^n_{\inte} \longrightarrow \RR^n$ be a linear
  mapping, represented by the matrix\/ $D$ with respect to the
  standard bases in\/ $\RR^n_{\inte}$ and\/ $\RR^n$. Then, for\/
  $k\in L^{\circledast}$, the FB coefficients\/
  $A^{}_{\vL^{}_{\vartheta}}(k)$ of the deformed model set\/
  $\vL^{}_{\vartheta} = \{ x+ \vartheta (x^{\star}) \, : \, x\in \vL
  \}$ are given by
\[
    A^{}_{\vL^{}_{\vartheta}}(k) \, = \, H(k^{\star} - D^{\top}k),
\]
with \pushQED{\qed}
\[
  H(k^{}_{\inte}) \, = \, \frac{\dens(\vL)}{\vol( W)}\,\ts
  \widecheck{\bm{1}^{}_{W}}\bigl(k^{}_{\inte}\bigr). \qedhere
\]
\popQED
\end{theorem}

\section{Spectre}\label{sec:spec}

The Spectre tiling was constructed shortly after the discovery of the
Hat tiling by the same team of authors~\cite{Spectre}. The Spectre,
which looks a little like a malicious cat, is an aperiodic monotile
with respect to translations and rotations. Notably, in contrast to
the Hat tiling, a~reflected copy of the prototile is not required.
The Hat and Spectre tilings are closely related, as the latter was
constructed using two tiles from the Hat tiling family. Nevertheless,
the combinatorics of the Spectre is rather different from that of the
Hat.

Despite the fact that one needs only translations and rotations of a
single Spectre tile, the Spectre tiling forms two LI-classes. This
manifests itself in distinct (and rationally independent) frequencies
for Spectres rotated by $\tfrac{\pi}{6}$ relative to each
other. Although Spectres occur in 12 orientations, the LI-classes have
six-fold symmetry only. It is thus tempting to speak of Spectres and
Shadow-Spectres, whose frequency ratio is $(4+\sqrt{15})^2$.

Smith et al.~\cite{Spectre} provided a combinatorial inflation for
marked hexagons, which gives rise to the Spectre tiling. This
inflation acts on nine different hexagons
($\Gamma,\, \Delta,\,\Theta,\,\Lambda,\,\Xi,\,\Pi,\,
\Sigma,\,\Phi,\,\Psi$), each appearing in six different orientations,
which gives 54 translational prototiles in total. One can assign
control points to five of them ($\Theta,\,\Xi,\,\Sigma,\,\Phi,\,\Psi$)
such that the resulting point set is MLD to the Spectre
tiling. Moreover, based on the combinatorial inflation, a self-similar
version of the Spectre tiling was derived \cite{BGS2}, called
CASPr. It is topologically conjugate (but not MLD) to the Spectre, and
it possesses a model set description, which we employ in what
follows. The cut-and-project description of the CASPr tiling is more
complex than that of the CAP tiling, which plays the analogous role
for the Hat tiling~\cite{BGS}. Although the leading eigenvalue of the
inflation matrix is $\lambda = 4+\sqrt{15}$, a PV unit, the
corresponding linear scaling is $\sqrt{4+\sqrt{15}}$. Due to the
chiral nature of the tiling, a reflection is required when the
substitution rule is applied once; see~\cite{BGS2} for further
details.  The underlying number field is the quartic number field
$\QQ(\xi,\lambda) = \QQ(\alpha)$ with
$\alpha = \sqrt{5}\ee^{\frac{2\pi \ii}{12}}$, which satisfies
$\alpha^4 = 5\alpha^2-25$. In contrast to the Hat tiling,
$\QQ(\alpha)$ has class number 2, see entry 4.0.3600.3 of
\cite{LMFDB}, which makes the description of the return module in
terms of ideals more difficult.

The generators of the return module
$\cR^{}_{_{\CASPr}}$ can be chosen as follows
\begin{equation}\label{eq:CASPr_generators}
\begin{split}
  g^{}_{1} \ & = \ -1-\xi+\lambda -2\xi\lambda,\\
  g^{}_{2} \ & = \ 1-2\xi+2\lambda -\xi\lambda \ = \ \xi\ts g^{}_{1} ,\\
  g^{}_{3} \ & = \ -2+\xi+2\lambda +2\xi\lambda,\\
  g^{}_{4} \ & = \ -2-2\xi-\lambda +2\xi\lambda,
\end{split}
\end{equation}
so $\cR^{}_{_{\CASPr}} \subset \ZZ[\lambda,\xi]$, the latter being of
index $3$ in the ring of integers
$\cO^{}_{\QQ(\alpha)} = \langle 1,\, \alpha, \, \tfrac{\alpha^2}{5},
\, \tfrac{\alpha^3}{5} \rangle$. Alternatively, $g^{}_{4}$ could be
replaced by $\xi\ts g^{}_{3}$, making the basis more symmetric, but
this would also require changes to Figure~\ref{fig:reprogen}.  Note
that, although $\ZZ[\lambda,\xi]$ is \emph{not} an ideal in
$\cO^{}_{\QQ(\alpha)}$, the return module $\cR^{}_{_{\CASPr}}$
\emph{is} an ideal in $\ZZ[\lambda,\xi]$ as well as in
$\cO^{}_{\QQ(\alpha)}$. Surprisingly, it possesses the same set of
generators in both cases, so we can write (without confusion)
\begin{equation}\label{eq:CASPR_ideals}
  \cR^{}_{_{\CASPr}} \, = \,
  \bigl ( g^{}_{1},\, g^{}_{3} \bigr) \, = \, \bigl ( g^{}_{1} \bigr)
  + \bigl ( g^{}_{3} \bigr).  
\end{equation}

The representation of
$\cR^{}_{_{\CASPr}} \subset \ZZ[\lambda,\xi]$ in
$\RR^2$ can be chosen as
\begin{equation}\label{eq:genSpec}
\begin{split}
  1\ &\longmapsto \ \bm{w}^{}_{1}\,=\,\biggl(\begin{matrix} 1 \\ 0
  \end{matrix}\biggr)\, , \\
  \xi\ &\longmapsto \
  \bm{w}^{}_{2}\,=\,\myfrac{1}{2}\biggl(\begin{matrix} 1 \\ \sqrt{3}
	\end{matrix}\biggr)\,, \\
        \lambda \ & \longmapsto \ \bm{w}^{}_{3}
        \,=\,\biggl(\begin{matrix} 4 + \sqrt{15} \, \\ 0
	\end{matrix}\biggr) \, =\, \lambda\bm{w}^{}_{1} , \\
        \xi \lambda & \longmapsto \
        \bm{w}^{}_{4}\,=\,\myfrac{1}{2}\biggl(\begin{matrix} 4 +
          \sqrt{15} \\ 4\sqrt{3} + 3\sqrt{5}
	\end{matrix}\biggr)\, =\, \lambda\bm{w}^{}_{2}. 
\end{split}
\end{equation} 
With this parametrisation, the expansive mapping of the inflation rule
reads
\[
  R \,=\, \myfrac{1}{6} \begin{pmatrix} 9+2\sqrt{15} & -\sqrt{3} \\
    -\sqrt{3} & -9-2\sqrt{15}  \end{pmatrix}\!.
\]
It can be understood as a concatenation of the reflection about the
$x$-axis, a rotation by
$\theta = - \arccos\bigl(\tfrac{9+2\sqrt{15}}{6\sqrt{\lambda}}\bigr)
\approx -5.9^{\circ}$, and a linear scaling by $\sqrt{\lambda}$. As
such, $R$ is a matrix square root of $\lambda\ts \mathmybb{1}_2$.

The $\star$-images of the generators $\bm{w}_i$ are given by the
embedding of $1$, $\overline{\xi}$, $\lambda'$ and
$\overline{\xi} \lambda'$, where $\overline{\,\cdot \,}$ denotes
complex conjugation and $':\QQ(\lambda) \to \QQ(\lambda)$ is the
non-trivial field automorphism $(a+b\lambda)'\mapsto a+8b -
b\lambda$. The concatenation of these two maps defines the $\star$-map
in $\QQ(\lambda, \xi)$, which is the non-trivial Galois isomorphism
fixing the subfield $\QQ(\sqrt{-5}\ts)$ of $\QQ(\lambda, \xi)$. For
the embedding of the generators, we obtain
\begin{equation}\label{eq:genSpecStar}
\begin{split}
  \bm{w}^{\star}_{1}\ & =\,\biggl(\begin{matrix} 1 \\ 0
  \end{matrix}\biggr)\, , \\
  \bm{w}^{\star}_{2}\ & =\,\myfrac{1}{2}\biggl(\begin{matrix} 1 \\
    -\sqrt{3}
  \end{matrix}\biggr)\,, \\
  \bm{w}^{\star}_{3} \ & =\,\biggl(\begin{matrix} 4 - \sqrt{15} \\ 0
  \end{matrix}\biggr) \, , \\
  \bm{w}^{\star}_{4} \ & =\,\myfrac{1}{2}\biggl(\begin{matrix} 4 -
    \sqrt{15} \\ -4\sqrt{3} + 3\sqrt{5} \end{matrix}\biggr)\, ,
\end{split}
\end{equation} 
which describes the entire $\star$-map due to $\QQ$-linearity. The
induced matrix $R^{\star}$ becomes
\[
  R^{\star} \,=\, \myfrac{1}{6} \begin{pmatrix} 9-2\sqrt{15} &
    \sqrt{3} \\ \sqrt{3} & -9+2\sqrt{15}
      \end{pmatrix}\!,
\]
and is a matrix square root of $\lambda^{\star}\ts \mathmybb{1}_2$.

Using this embedding, one obtains a Euclidean CPS with the lattice
given by the embedding of the return module \eqref{eq:genSpec}. Its
basis matrix reads
\begin{equation}\label{eq:lattSpectre}
  B \, = \, \myfrac{3}{2}\begin{pmatrix}
    -1 & -5-\sqrt{15} & 7+2\sqrt{15}  & -2 \\[2pt]
    -3\sqrt{3} -2\sqrt{5} & -\sqrt{3}-\sqrt{5} &
    3\sqrt{3} +2\sqrt{5} & 2\sqrt{3}+2\sqrt{5}\\[2pt]
    -1 & -5+\sqrt{15} & 7-2\sqrt{15}  & -2 \\[2pt]
    3\sqrt{3} -2\sqrt{5} & \sqrt{3}-\sqrt{5} &
    -3\sqrt{3} +2\sqrt{5} & -2\sqrt{3}+2\sqrt{5} \end{pmatrix}.
\end{equation}
The lattice has density $\tfrac{1}{3645}$. Via the form
\eqref{eq:genSpecStar}, one can take the $\star$-image of the tiling
control points to obtain the windows. Moreover, the control points
satisfy renormalisation equations as in the previous examples. In this
case, we obtain 54 equations for 54 point sets, 30 of which then give
the window; see~\cite{BGS2} for further details.

The total window is simply connected, with sixfold symmetry, but
without any mirror symmetry. It has fractal boundaries and, in
contrast to the Hat tiling, there are no other types of
boundaries. The window is shown in Figure~\ref{fig:Spectrewindow}.

\begin{figure}
\includegraphics[angle=90, origin=c,
  width=0.65\textwidth]{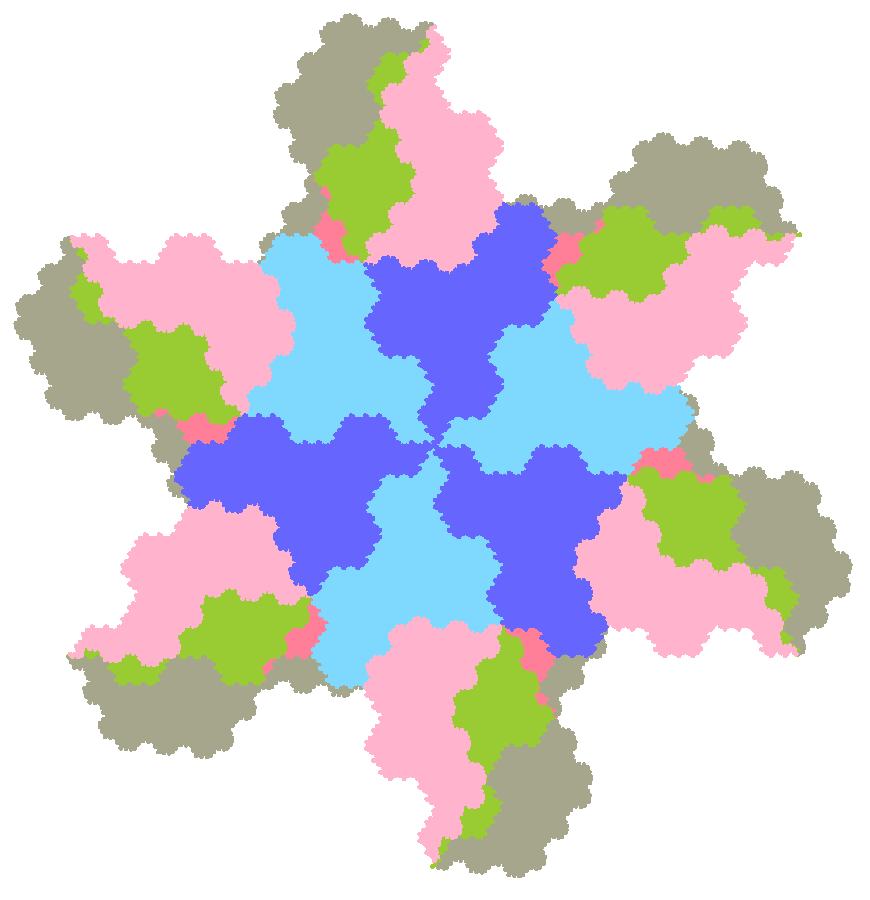}
\caption{The window for the control points of the Spectre tiling. The
  five different colours correspond to the five different types of
  control points. We used two shades of blue to visualise the six
  different regions in the central blue part. The window has sixfold
  symmetry, but no mirror symmetry.}
	\label{fig:Spectrewindow}
\end{figure} 

Since the window is a fundamental domain of a hexagonal lattice in
$\RR^2_{\inte}$, its volume is easily computable~\cite{BGS2} and reads
$\tfrac{135}{2}(4\sqrt{3}-3\sqrt{5}\,) \approx 14.85$. By a density
argument~\cite{Schl98}, we know that the control points of the CASPr
tiling form a subset of the model set with window from
Figure~\ref{fig:Spectrewindow}, where both have the same density. This
tiny difference stems from boundary points of the window and does not
affect the FB coefficients. This also implies that the diffraction and
dynamical spectra of the Spectre tiling are both pure point
\cite{BGS2}.

\begin{prop}[\cite{BGS2}]
  The CASPr tiling is MLD with a Euclidean model set derived from the
  CPS arising from the Minkowski embedding of\/
  $\mathcal{R}^{}_{_{\mathrm{CASPr}}}$ and the window with fractal
  boundaries shown in Figure~$\ref{fig:Spectrewindow}$. \qed
\end{prop}

\begin{figure}
\centering
\includegraphics[width=0.75\linewidth]{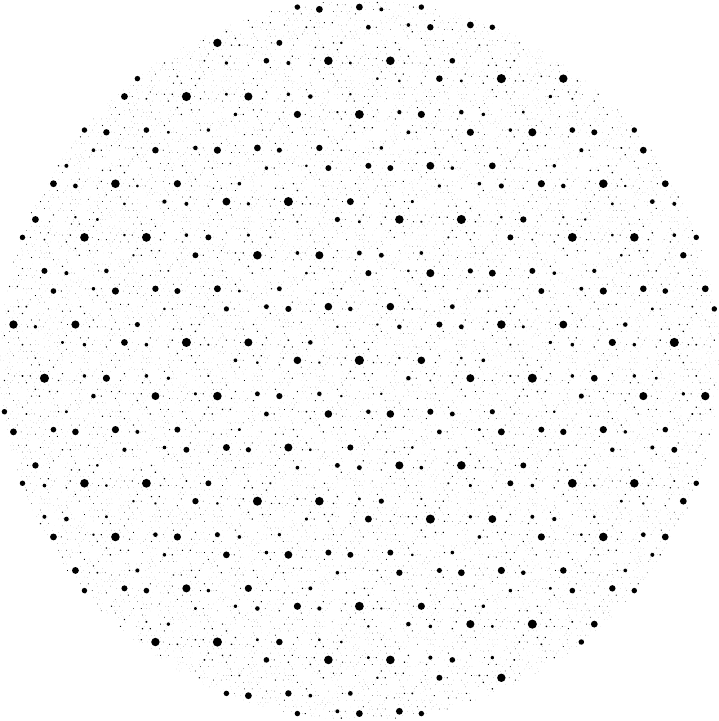}
\caption{Diffraction of the CASPr tiling control points with equal
  weights. The picture shows a disk of radius 0.5 around $k=0$. The
  intensities of the Bragg peaks are proportional to the area of the
  disks and the intensity of the brightest one equals
  $\tfrac{31-8\sqrt{15}}{972} \approx 0.000016597\dots$. One has
  sixfold rotational symmetry, but no reflection
  symmetry. $15$~iterations of the cocycle were used for the
  computation.}  \label{fig:CASPr-diff}
\end{figure}

Now, for the Fourier module (and the dynamical spectrum), one
considers the dual basis matrix $B^{\ast} = \bigl(B^{-1}\bigr)^{\top}$
of \eqref{eq:lattSpectre} and its $\pi$-projection. The generators
written as columns of a~matrix read
\[
  \myfrac{1}{90}\begin{pmatrix}
    -5+2\sqrt{15} & -10 +2\sqrt{15} & -5+2\sqrt{15} & -5+\sqrt{15} \\[2pt]
    -5\sqrt{3}+2\sqrt{5} & 10\sqrt{3}-8\sqrt{5} & 5\sqrt{3}-4\sqrt{5}
    & -5\sqrt{3}+5\sqrt{5} \end{pmatrix},
\]
and the Fourier module $L^{\circledast}_{_{\CASPr}}$ (the dynamical
spectrum) can be expressed as
\[
  L^{\circledast}_{_{\CASPr}} \, = \, \myfrac{\ii \sqrt{5}}{135}
  \ts\cR^{}_{_{\CASPr}} \ts ,
\]
with $\cR^{}_{_{\CASPr}}$ from \eqref{eq:CASPR_ideals}, so it forms a
fractional ideal. We refer the reader to~\cite{BGS2} for further
details and a number-theoretic description of the return and Fourier
modules.

For the diffraction amplitudes, we again employ the cocycle
method. This time, the displacement matrix has size $54\times 54$, but
since several metatiles form clusters, one can reduce the dimension to
$30 \times 30$, see~\cite{BGS2} for further details.

For the total intensity, one has to consider the weighted sum, so one
obtains
\[
  I^{}_{_{\CASPr}}(k) \, = \, \begin{cases}
    \bigl|H^{}_{_{\CASPr}}(k^{\star})\bigr|^2, & \mbox{if}
    \ k\in L^{\circledast}_{_{\CASPr}}, \\
    0, & \mbox{otherwise,}  \end{cases}
\]
with
\[
  H_{_{\CASPr}}(k^{}_{\inte}) \, = \,
  \frac{\dens\bigl(\vL^{}_{_{\CASPr}}\bigr)}{\vol\bigl(
    W^{}_{_{\CASPr}} \bigr)}\sum_{i}\alpha^{}_{i} \ts
  \widecheck{\bm{1}^{}_{W_{_{\CASPr,i}}}}\bigl(k^{}_{\inte}\bigr),
\]
where $\alpha^{}_{i}\in \CC$ denotes the weight of tiles of type
$i$. We note that, in order to obtain the diffraction intensities, one
only has to choose non-zero weights for $30$ (instead of $54$)
elements.  The diffraction pattern for equal weights (near $0$) is
shown in Figure~\ref{fig:CASPr-diff}. It reflects the properties of
the Spectre tiling: It exhibits sixfold rotational symmetry, while
mirror symmetry is absent.

The CASPr tiling can be reprojected, and various tilings related to
the Spectre tiling can be obtained. First, we start by recovering a
tiling by regular hexagons, which is combinatorially equivalent to the
tiling by Spectre clusters and plays a pivotal role in all cohomological
considerations in~\cite{BGS2}. The deformation matrix reads
\begin{equation}\label{eq:def_hex}
    D^{}_{\bhex} \, = \, \begin{pmatrix}
        -1 & 0 \\ 0 & 1 \end{pmatrix},
\end{equation}
and the new return module becomes a scaled and rotated hexagonal
lattice of rank $2$ (as one would expect from a hexagonal tiling),
\[
  \cR^{}_{\bhex} \, = \, 6\ts\ii\sqrt{5}\,\ZZ[\xi] \, = \,
  (8-16\xi-2\lambda+4\lambda\xi)\,\ZZ[\xi] \ts .
\]
The reprojection of the generators $g^{}_{i}$ is shown in
Figure~\ref{fig:hex_deformation}. We note at this point that the
resulting tiling consists of \emph{combinatorial} hexagons (which are
not regular hexagons, but geometric shapes rather similar to the CASPr
tiles, see \cite[Fig.~9]{BGS2}), but within the MLD class of this
tiling, one also finds a tiling with regular hexagons.

\begin{figure}
\begin{subfigure}{.5\textwidth}
\centering
\includegraphics[ width=0.875\linewidth]{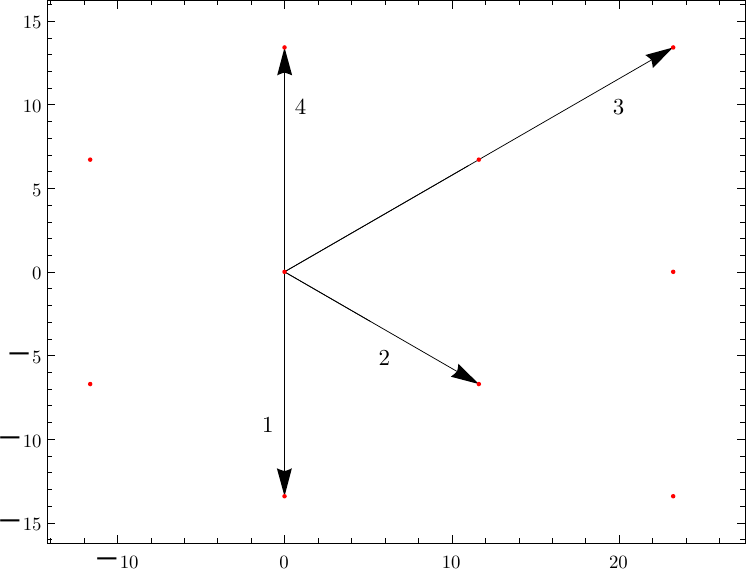}
\caption{}
	\label{fig:hex_deformation}
\end{subfigure}%
\begin{subfigure}{.5\textwidth}
	\centering
	\includegraphics[width=0.9\linewidth]{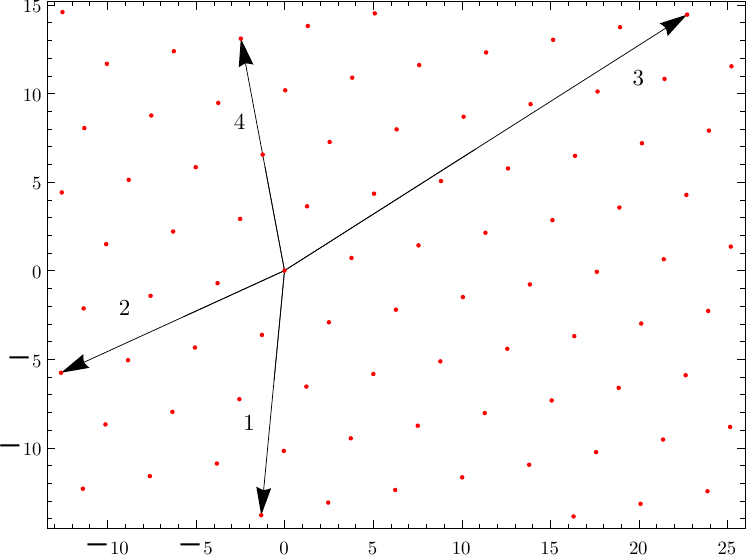}
	\caption{}
	\label{fig:HT_deformation}
\end{subfigure}
\caption{Reprojected generators $g^{}_{i}$
  (Eq.~\eqref{eq:CASPr_generators}) of the return module of the CASPr
  tiling. Figure (A) shows the deformation to the regular hexagon
  tiling and the underlying lattice $6\ts \ii\sqrt{5}\,\ZZ[\xi]$ (red
  dots). Figure (B) shows the generators for the Hat--Turtle tiling
  and the lattice
  $\tfrac{2}{67} \bigl( -185 -206\xi +20\lambda +44\lambda\xi
  \bigr)\ZZ[\xi]$.}
  \label{fig:reprogen}
\end{figure}

This implies that the control points form a~lattice subset, and hence
the diffraction image is lattice-periodic with the lattice of periods
being given by $\cR^{\ast}_{\bhex}$
\begin{equation}\label{eq:hex_periods}
  \cR^{\ast}_{\bhex} \, = \, \myfrac{\sqrt{15}}{45} \, \ZZ[\xi]
  \, = \, \myfrac{\lambda -4}{45} \, \ZZ[\xi] \ts . 
\end{equation}
We illustrate the diffraction image of the hexagon tiling in
Figure~\ref{fig:hex_diff}.

\begin{figure}
\centering
\includegraphics[width=0.5\linewidth]{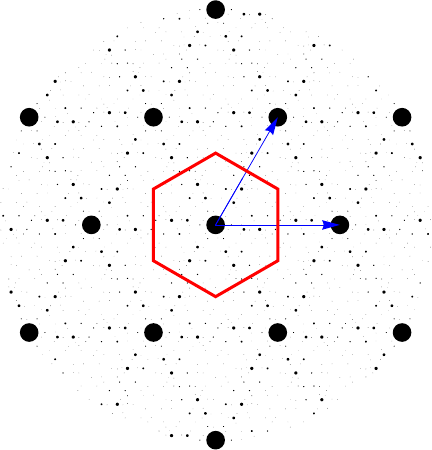}
\caption{Diffraction image of the aperiodic hexagon tiling. The
  diffraction measure is lattice periodic, with the lattice of periods
  given by \eqref{eq:hex_periods}. Its generators are indicated by the
  blue arrows, and a fundamental domain with the red hexagon. The
  picture shows the intensities of the Bragg peaks in the intersection
  of the Fourier module with a ball of radius~0.15 in $\RR^2$. The
  intensity of the central peak is the same as for the CASPr
  tiling. $10$ iterations of the cocycle were used for the
  computation; see \cite{Jan} for further
  details. } \label{fig:hex_diff}
\end{figure}

The reprojection to another lattice tiling --- the Hat--Turtle (HT)
tiling~\cite{Spectre} --- is more complicated, and the lattice is much
finer.  The deformation matrix reads
\begin{equation}\label{eq:def_ht}
  D^{}_{_{\mathrm{HT}}} \, = \, \myfrac{1}{201} \begin{pmatrix}
    44\sqrt{15}-231 & 80\sqrt{3}-84\sqrt{5}\\
    80\sqrt{3}-84\sqrt{5} & 231-44\sqrt{15} \end{pmatrix}.
\end{equation}

\begin{figure}
\centering
\includegraphics[width=0.8\linewidth]{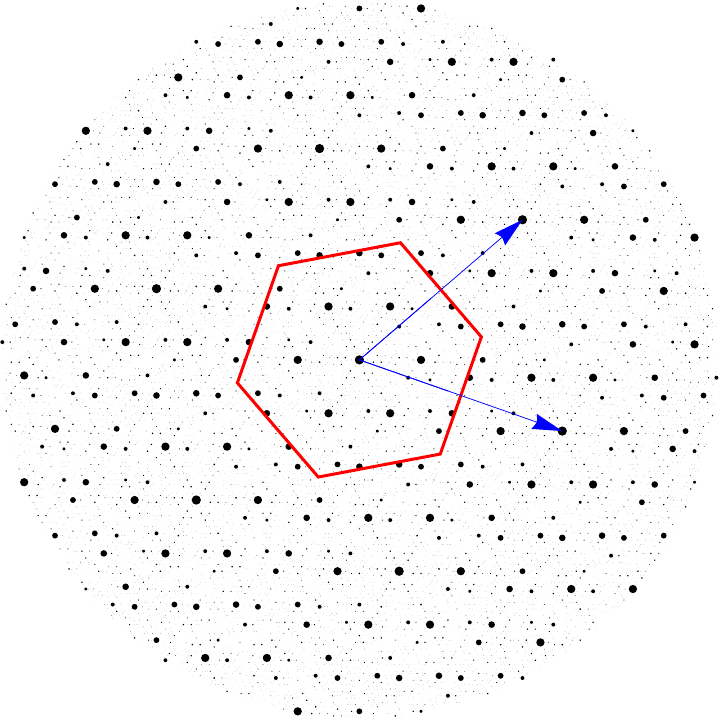}
\caption{Diffraction image of the Hat--Turtle tiling. The diffraction
  is lattice periodic, with the lattice of periods given by
  \eqref{eq:HT_per}, whose generators are indicated by the blue
  arrows. A fundamental domain is marked by the red hexagon. The
  picture shows the intensities of the Bragg peaks restricted to the
  intersection of the Fourier module with a ball of radius~0.5 in
  $\RR^2$. The intensity of the central peak is the same as in the
  undeformed case. $10$ iterations of the cocycle
  were used for the computation.} \label{fig:HT_diff}
\end{figure}

The reprojection of $\cR^{}_{_{\CASPr}}$ yields a $\ZZ$-module
$\cR^{}_{_{\mathrm{HT}}}$ of rank $2$, as the HT tiling is again
a~lattice tiling, this time with return module
\[
  \cR^{}_{_{\mathrm{HT}}} \, = \, \myfrac{1}{67}\bigl(-240 +
  84\sqrt{15} -30\ts\ii \sqrt{3} +132\ts\ii \sqrt{5}\, \bigr) \ZZ[\xi]
  \, = \, \myfrac{2}{67} \bigl( -185 -206\xi +20\lambda +44\lambda\xi
  \bigr)\ZZ[\xi] \ts .
\]
For its dual module, one finds
\begin{equation}\label{eq:HT_per}
  \cR^{\ast}_{_{\mathrm{HT}}} \, = \, \myfrac{1}{45}
  \bigl(5+2\sqrt{15}-2\ts \ii\sqrt{5} \, \bigr)\ZZ[\xi]
  \, = \, \myfrac{1}{135} \bigl(-17+16\xi+8\lambda -4\lambda\xi
  \bigr)\ZZ[\xi] \ts , 
\end{equation}
which provides the lattice of periods of the diffraction pattern as in
the previous case. Note that the length of the period is approximately
$3.5$ times larger than in the case of the hexagon tiling. The lattice
constant of $\cR^{\ast}_{_{\mathrm{HT}}}$ reads
$\sqrt{\frac{4\lambda+5}{405}}$.  Figure~\ref{fig:HT_diff} shows the
intensities of the diffraction measure around the origin.

\begin{figure}
\centering
\includegraphics[width=0.8\linewidth]{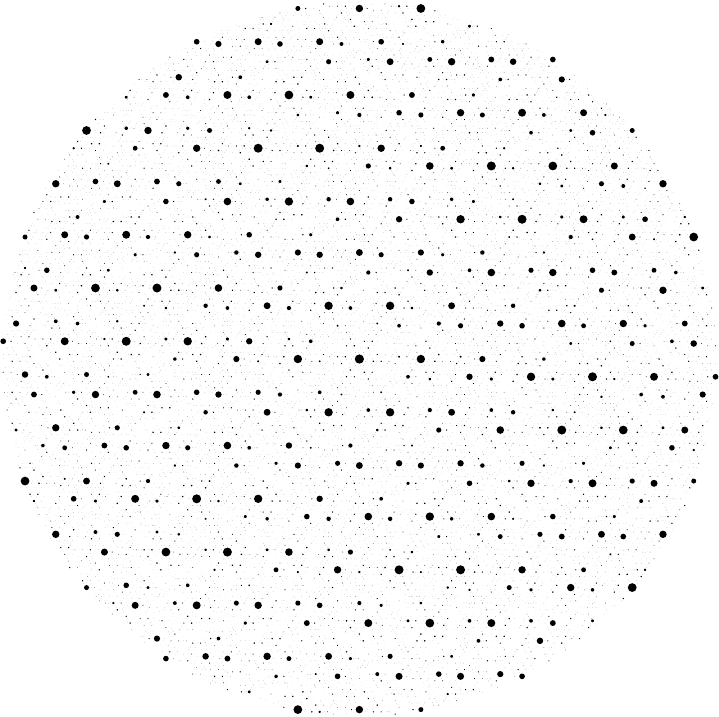}
\caption{Diffraction image of the Spectre tiling. The picture shows
  the intensities of the Bragg peaks restricted to the intersection of
  the Fourier module with a ball of radius~0.5 in $\RR^2$. The
  intensity of the central peak is the same as in the CASPr tiling. To
  compute the image, 10 iterations of the cocycle were used. }
	\label{fig:Spectre_diff}
\end{figure}
      
Finally, the Spectre tiling is a further deformation of the HT tiling
in the following sense. Hats and Turtles have edges of two different
lengths. The Spectre tiling is obtained by rescaling these edges such
that their lengths become equal (maintaining the edge
directions). Combining the deformation of the HT tiling with this
additional deformation results in a~relatively simple deformation
matrix, one obtains
\[
  D^{}_{_{\mathrm{Sp}}} \, = \, \myfrac{3-\sqrt{15}}{6} \begin{pmatrix}
    \sqrt{5} & 1 \\  1 & -\sqrt{5} \end{pmatrix}.
\]
The generators \eqref{eq:CASPr_generators} are reprojected to the
generators $h^{}_{i}$ of the return module of the Spectre tiling and
read
\begin{align*}
  h^{}_{1} & = \, \myfrac{3}{2}\begin{pmatrix}
    -1+4\sqrt{3}-3\sqrt{5} \\[1pt] 12-3\sqrt{3}-2\sqrt{5}-3\sqrt{15}
  \end{pmatrix}, \\
  h^{}_{2} & = \, \myfrac{1}{6}\begin{pmatrix}
    60-98\sqrt{3}+75\sqrt{5}+15\sqrt{15}\, \\[1pt]
                                 30+6\sqrt{3}+9\sqrt{5}-7\sqrt{15}
    \end{pmatrix}, \\
    h^{}_{3} & = \, \myfrac{1}{2}\begin{pmatrix}
      21-42\sqrt{3}+33\sqrt{5}+6\sqrt{15}\, \\[1pt]
      -6+9\sqrt{3}+6\sqrt{5}+3\sqrt{15} \end{pmatrix}, \\
    h^{}_{4} & = \, \myfrac{3}{2}\begin{pmatrix}
      -6-3\sqrt{3}+3\sqrt{5} \\[1pt] -33+6\sqrt{3}+6\sqrt{5}+9\sqrt{15}
    \end{pmatrix}. \\
\end{align*}
When interpreted as complex numbers, the generators belong to
$\QQ(\beta)$ with $\beta^8-3\beta^6+8\beta^4-3\beta^2+1=0$, a number
field containing $\xi,\, \lambda$ and the twelfth root of unity, which
is not surprising due to the geometry of the Spectre tiles.  We note
that the control points of the HT tiling and the Spectre tiling do not
differ too much, so the diffraction pattern looks very similar in both
cases. The diffraction of the Spectre is shown in
Figure~\ref{fig:Spectre_diff}. We also include a comparison of the
diffraction of the HT and Spectre tilings around a~point
$\tfrac{10}{45}\bigl(5+2\sqrt{15}, \, -2\sqrt{5}\ts \bigr)^{\top}$ to
show the significant difference in both diffraction patterns, see
Figure~\ref{fig:HT_Spectre}.

\begin{figure}
\centering
\includegraphics[width=0.95\linewidth]{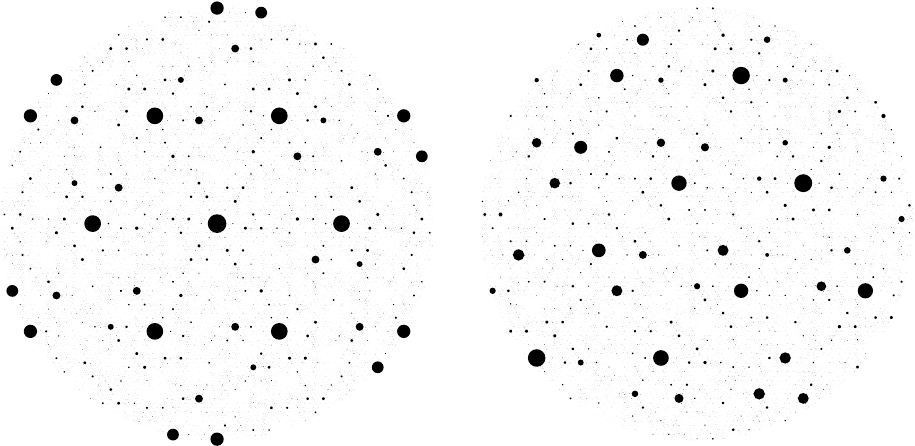}
\caption{Diffraction image of the Hat--Turtle (left) and Spectre
  (right) tiling. Both pictures show the intensities of the Bragg
  peaks restricted to the intersection of the Fourier module with a
  ball of radius~0.15 around
  $\tfrac{10}{45}\bigl(5+2\sqrt{15}, \, -2\sqrt{5}\ts \bigr)^{\top}$
  in $\RR^2$. The central point is a lattice point in the lattice of
  periods of the diffraction measure of HT tiling given by
  \eqref{eq:HT_per}. $10$ iterations of the cocycle were
  used for these figures.} \label{fig:HT_Spectre}
\end{figure}

We summarise the observations from the last paragraphs in the
following theorem.

\begin{theorem}
  The set of control points of the Spectre, Hat--Turtle, and
  (combinatorial) hexagon tilings are deformed model sets, and can be
  interpreted as reprojections. They are obtained from the set of
  control points of the CASPr tiling via the deformation mappings
\allowdisplaybreaks
\begin{align*}
D^{}_{_{\mathrm{Sp}}} \, &= \, \myfrac{3-\sqrt{15}}{6} 
\begin{pmatrix}\sqrt{5} & 1 \\1 & -\sqrt{5}
\end{pmatrix}, \\[2pt]
D^{}_{_{\mathrm{HT}}} \,& = \, \myfrac{1}{201} 
\begin{pmatrix} 44\sqrt{15}-231 & 80\sqrt{3}-84\sqrt{5}\\
80\sqrt{3}-84\sqrt{5} & 231-44\sqrt{15}
\end{pmatrix}, \\[2pt]
D^{}_{\bhex} \,& = \, \begin{pmatrix}-1 & 0 \\ 0 & 1
\end{pmatrix}.  
\end{align*}    
The Fourier module of all these tilings is\/ 
$L^{\circledast}_{_{\mathrm{CASPr}}}$, and the diffraction
intensities are
\[
  I^{}_{\bullet}(k) \, = \, \begin{cases}
    \bigl|H^{}_{_{\CASPr}}(k^{\star}-D^{\top}_{\bullet} k)\bigr|^2, &
    \mbox{if}
    \ k\in L^{\circledast}_{_{\CASPr}}, \\
    0, & \mbox{otherwise,} \end{cases}
\] 
where\/ $\bullet \in \{\mathrm{Sp}, \, \mathrm{HT}, \, \mathrm{hex} \}$.
\qed
\end{theorem}

\bigskip

\section*{Appendix A ---
   \texorpdfstring{$T^{}_{\CAP}$}{TCAP}}

Here, we give the non-empty matrix blocks for \eqref{eq:TCAP}.
\label{sec:AppA}
\[
  T^{}_{12} \,= \,\begin{pmatrix}
    \bigl\{\substack{{-}\tau{+}1\\{-}\tau\xi{-}2\xi}\}&
    \vn&\vn&\vn&\vn&\vn\\
    \vn& \bigl\{\substack{\tau{+}2\\{-}2\tau\xi{-}\xi}\bigr\}
    &\vn&\vn&\vn&\vn\\
    \vn&\vn& \bigl\{\substack{2\tau{+}1\\
      {-}\tau\xi{+}\xi}\bigr\}&\vn&\vn&\vn\\
    \vn&\vn&\vn& \bigl\{\substack{\tau{-}1\\
      {+}\tau\xi{+}2\xi}\bigr\}& \vn&\vn\\
    \vn&\vn&\vn&\vn&
    \bigl\{\substack{{-}\tau{-}2\\{+}2\tau\xi{+}\xi}\bigr\} &\vn\\
    \vn&\vn&\vn&\vn&\vn&
    \bigl\{\substack{{-}2\tau{-}1\\{+}\tau\xi{-}\xi}\bigr\}\\
\end{pmatrix},
\]

\[
T^{}_{21} \, = \, 
\begin{pmatrix}
  \vn & \bigl\{\substack{4\tau{+}1 \\+\tau\xi{+}\xi}\bigr\}&
  \vn&\vn&\vn&\vn\\
  \vn&\vn&  \bigl\{\substack{-\tau{-}1\\
    +5\tau\xi{+2\xi}}\bigr\}&\vn&\vn&\vn\\
  \vn&\vn&\vn& \bigl\{\substack{-5\tau{-}2\\
    +4\tau\xi{+}\xi}\bigr\}&\vn&\vn\\
  \vn&\vn&\vn&\vn &
  \bigl\{\substack{-4\tau{-}1 \\-\tau\xi{-}\xi}\bigr\}& \vn\\
  \vn&\vn&\vn&\vn&\vn&
  \bigl\{\substack{\tau{+}1\\-5\tau\xi{-}2\xi}\bigr\}\\
  \bigl\{\substack{5\tau{+}2\\-4\tau\xi{-}\xi}\bigr\}&\vn&
  \vn&\vn&\vn&\vn\\
\end{pmatrix},
\]

\[
T^{}_{22} \, = \, 
\begin{pmatrix}
  \bigl\{0,\substack{3\tau+1\\-3\tau\xi-2\xi}\bigr\} &\vn&
  \bigl\{\substack{4\tau{+}1 \\+\tau\xi{+}\xi}\bigr\}& \vn&
  \vn&\vn\\
  \vn&\bigl\{0,\substack{3\tau+2\\-\xi}\bigr\} &\vn &
  \bigl\{\substack{-\tau{-}1\\+5\tau\xi{+2\xi}}\bigr\}&\vn
  &\vn\\
  \vn&\vn&\bigl\{0,\substack{1\\+3\tau\xi+\xi}\bigr\} &\vn&
  \bigl\{\substack{-5\tau{-}2\\+4\tau\xi{+}\xi}\bigr\}
  &\vn\\
  \vn&\vn&\vn&\bigl\{0,\substack{-3\tau-1\\
    +3\tau\xi+2\xi}\bigr\}& \vn & \bigl\{\substack{-4\tau{-}1
    \\-\tau\xi{-}\xi}\bigr\}\\
  \bigl\{\substack{\tau{+}1\\-5\tau\xi{-}2\xi}\bigr\}&\vn&
  \vn&\vn&\bigl\{0,\substack{-3\tau-2\\+\xi}\bigr\}&
  \vn \\
  \vn&\bigl\{\substack{5\tau{+}2\\-4\tau\xi{-}\xi}\bigr\}&
  \vn&\vn&\vn&\bigl\{0,\substack{-1\\
    -3\tau\xi-\xi}\bigr\}\\
\end{pmatrix},
\]

\[
T^{}_{23} \, = \, 
\begin{pmatrix}
  \vn & \bigl\{ \substack{8\tau+4\\-4\tau\xi-2\xi}\bigr\} &
  \bigl\{ \substack{4\tau+1\\+\tau\xi+\xi}\bigr\}&\vn &\vn &\vn\\
  \vn&\vn&\bigl\{ \substack{4\tau+2\\+4\tau\xi+2\xi}\bigr\} &
  \bigl\{ \substack{-\tau-1\\+5\tau\xi+2\xi}\bigr\} &\vn &\vn\\
  \vn &\vn &\vn & \bigl\{ \substack{-4\tau-2\\+8\tau\xi+4\xi} \bigr\}
  & \bigl\{ \substack{-5\tau-2\\+4\tau\xi+\xi} \bigr\} & \vn\\
  \vn & \vn &\vn &\vn& \bigl\{
  \substack{-8\tau-4\\+4\tau\xi+2\xi}\bigr\}
  & \bigl\{ \substack{-4\tau-1\\-\tau\xi-\xi}\bigr\}\\
  \bigl\{ \substack{\tau+1\\-5\tau\xi-2\xi}\bigr\} &\vn &\vn&\vn&\vn
  &\bigl\{ \substack{-4\tau-2\\-4\tau\xi-2\xi}\bigr\} \\
  \bigl\{ \substack{4\tau+2\\-8\tau\xi-4\xi} \bigr\}& \bigl\{
  \substack{5\tau+2\\-4\tau\xi-\xi} \bigr\} &
  \vn&\vn &\vn &\vn  \\
\end{pmatrix},
\]

\[
T^{}_{24} \, = \, 
\begin{pmatrix}
  \vn & \bigl\{ \substack{8\tau+4\\-4\tau\xi-2\xi}\bigr\} &
  \bigl\{ \substack{4\tau+1\\+\tau\xi+\xi}\bigr\} & \vn & \vn & \vn\\
  \vn & \vn & \bigl\{ \substack{4\tau+2\\+4\tau\xi+2\xi}\bigr\} &
  \bigl\{\substack{-\tau-1\\+5\tau\xi+2\xi} \bigr\} &\vn &\vn\\
  \vn & \vn & \vn &\bigl\{ \substack{-4\tau-2\\+8\tau\xi+4\xi}\bigr\}
  & \bigl\{ \substack{-5\tau-2\\+4\tau\xi+\xi} \bigr\} & \vn \\
  \vn & \vn & \vn &\vn & \bigl\{
  \substack{-8\tau-4\\+4\tau\xi+2\xi}\bigr\}
  & \bigl\{ \substack{-4\tau-1\\-\tau\xi-\xi}\bigr\} \\
  \bigl\{\substack{\tau+1\\-5\tau\xi-2\xi} \bigr\} &\vn &\vn&\vn & \vn
  &
  \bigl\{ \substack{-4\tau-2\\-4\tau\xi-2\xi}\bigr\} \\
  \bigl\{ \substack{4\tau+2\\-8\tau\xi-4\xi}\bigr\} & \bigl\{
  \substack{5\tau+2\\-4\tau\xi-\xi} \bigr\} &
  \vn&\vn & \vn& \vn  \\
\end{pmatrix},
\]

\[
T^{}_{32} \, = \, 
\begin{pmatrix}
  \vn & \vn & \vn & \bigl\{ \substack{ -4\tau -3 \\+5\tau\xi +3 \xi}
  \bigr\} & \bigl\{ \substack{-\tau-1 \\+5\tau\xi + 2\xi} \bigr\} &
  \bigl\{ \substack{ -5\tau -3 \\+\tau\xi} \bigr\} \\
  \bigl\{ \substack{ -\tau\\-4\tau\xi-3\xi} \bigr\} & \vn & \vn &\vn &
  \bigl\{ \substack{ -5\tau-3 \\+\tau\xi}\bigr\} &
  \bigl\{ \substack{-5\tau -2\\+4\tau\xi + \xi}\bigr\}\\
  \bigl\{ \substack{-4\tau-1\\-\tau\xi -\xi} \bigr\} & \bigl\{
  \substack{ 4\tau +3 \\-5\tau\xi-3\xi}\bigr\} &\vn &\vn &\vn
  & \bigl\{ \substack{-\tau\\-4\tau\xi-3\xi}\bigr\}\\
  \bigl\{ \substack{ 4\tau +3 \\-5\tau\xi -3 \xi} \bigr\} & \bigl\{
  \substack{\tau+1 \\-5\tau\xi - 2\xi} \bigr\} &
  \bigl\{ \substack{ 5\tau +3 \\-\tau\xi} \bigr\} &\vn & \vn & \vn  \\
  \vn & \bigl\{ \substack{ 5\tau+3 \\-\tau\xi}\bigr\} & \bigl\{
  \substack{5\tau +2\\-4\tau\xi - \xi}\bigr\}&
  \bigl\{ \substack{ \tau\\+4\tau\xi+3\xi} \bigr\} & \vn & \vn \\
  \vn &\vn & \bigl\{ \substack{\tau\\+4\tau\xi+3\xi}\bigr\}& \bigl\{
  \substack{4\tau+1\\+\tau\xi +\xi} \bigr\} &
  \bigl\{ \substack{ -4\tau -3 \\+5\tau\xi+3\xi}\bigr\} &\vn \\
\end{pmatrix},
\]

\[ 
T^{}_{33} \,=\, T^{}_{34} \, = \, 
\begin{pmatrix}
  \vn & \vn & \vn & \vn & \vn &
  \bigl\{ \substack{-5\tau-3\\+\tau\xi} \bigr\}\\
  \bigl\{ \substack{ -\tau \\-4\tau\xi -3\xi } \bigr \} & \vn & \vn
  &\vn &\vn &\vn \\
  \vn & \bigl\{ \substack{4\tau +3 \\ - 5\tau\xi-3\xi } \bigr\} & \vn
  &\vn & \vn &\vn \\
  \vn & \vn & \bigl\{ \substack{5\tau+3\\-\tau\xi} \bigr\}
  &\vn & \vn & \vn \\
  \vn &\vn &\vn &\bigl\{ \substack{ \tau \\+4\tau\xi +3\xi } \bigr \}
  & \vn & \vn \\ \vn & \vn &\vn& \vn &
  \bigl\{ \substack{-4\tau -3 \\ + 5\tau\xi+3\xi} \bigr\} & \vn  \\
\end{pmatrix},
\]

\[
T^{}_{42} \, = \, 
\begin{pmatrix}
  \vn & \bigl\{ \substack{ 3\tau +1 \\-3\tau\xi-2\xi} \bigr\} & \vn &
  \bigl\{ \substack{-\tau-2\\+2\tau\xi+\xi} \bigr\} &
  \vn & \bigl\{ \substack{-2\tau-2\\-2\tau\xi -2\xi} \bigr\}\\
  \bigl\{ \substack{ 2\tau +2 \\ -4\tau\xi-4\xi}\bigr\} &\vn & \bigl\{
  \substack{3\tau +2 \\ -\xi} \bigr\} & \vn &
  \bigl\{ \substack{-2\tau-1\\+\tau\xi-\xi} \bigr\} & \vn\\
  \vn & \bigl\{ \substack{4\tau+4\\-2\tau\xi-2\xi} \bigr\} & \vn &
  \bigl\{ \substack{1\\+3\tau\xi+\xi} \bigr\} &
  \vn & \bigl\{ \substack{-\tau+1\\-\tau\xi-2\xi}\bigr\} \\
  \bigl\{ \substack{\tau+2\\-2\tau\xi-\xi} \bigr\} & \vn & \bigl\{
  \substack{2\tau+2\\+2\tau\xi +2\xi} \bigr\}& \vn &
  \bigl\{ \substack{ -3\tau -1 \\+3\tau\xi+2\xi} \bigr\} & \vn \\
  \vn & \bigl\{ \substack{2\tau+1\\-\tau\xi+\xi} \bigr\} & \vn
  &\bigl\{ \substack{ -2\tau -2 \\ +4\tau\xi+4\xi}\bigr\}
  &\vn &\bigl\{ \substack{-3\tau -2 \\ +\xi} \bigr\} \\
  \bigl\{ \substack{-1\\-3\tau\xi-\xi} \bigr\} &\vn & \bigl\{
  \substack{\tau-1\\+\tau\xi+2\xi}\bigr\} & \vn & \bigl\{
  \substack{-4\tau-4\\+2\tau\xi+2\xi} \bigr\} &\vn
\end{pmatrix},
\]

\[
T^{}_{43} \, = \, 
\begin{pmatrix}
  \vn & \bigl\{ \substack{ 3\tau +1\\-3\tau\xi-2\xi} \bigr\} & \vn &
  \vn&
  \bigl\{ \substack{-9\tau-6\\+6\tau\xi+3\xi} \bigl\} & \vn\\
  \vn& \vn& \bigl \{ \substack{3\tau+2\\-\xi} \bigr\} & \vn &
  \vn & \bigl\{ \substack{-6\tau -3 \\-3\tau\xi-3\xi}\bigr\}\\
  \bigl\{ \substack{3\tau +3\\-9\tau\xi-6\xi} \bigr\} & \vn & \vn &
  \bigl\{ \substack{1\\+3\tau\xi+\xi} \bigr\} &
  \vn & \vn\\
  \vn& \bigl\{ \substack{9\tau+6\\-6\tau\xi-3\xi} \bigl\} & \vn& \vn &
  \bigl\{ \substack{ -3\tau -1\\+3\tau\xi-2\xi} \bigr\} &\vn & \\
  \vn & \vn & \bigl\{ \substack{6\tau +3 \\+3\tau\xi+3\xi}\bigr\} &
  \vn & \vn & \bigl \{ \substack{-3\tau-2\\+\xi} \bigr\} \\
  \bigl\{ \substack{-1\\-3\tau\xi-\xi} \bigr\} &\vn & \vn & \bigl\{
  \substack{-3\tau -3\\+9\tau\xi+6\xi} \bigr\} & \vn & \vn
\end{pmatrix},
\]

\[
T^{}_{44} \, = \, 
\begin{pmatrix}
  \bigl\{ \substack{2\tau+1\\-7\tau\xi-5\xi}\bigr\} & \bigl\{
  \substack{ 3\tau +1\\-3\tau\xi-2\xi} \bigr\} & \vn & \vn&
  \bigl\{ \substack{-9\tau-6\\+6\tau\xi+3\xi} \bigl\} & \vn\\
  \vn& \bigl\{\substack{7\tau+5\\-5\tau\xi-4\xi}\bigr\} & \bigl \{
  \substack{3\tau+2\\-\xi} \bigr\} & \vn &
  \vn & \bigl\{ \substack{-6\tau -3 \\-3\tau\xi-3\xi}\bigr\}\\
  \bigl\{ \substack{3\tau +3\\-9\tau\xi-6\xi} \bigr\} & \vn &
  \bigl\{\substack{5\tau+4\\+2\tau\xi+\xi}\bigr\} & \bigl\{
  \substack{1\\+3\tau\xi+\xi} \bigr\} & \vn & \vn\\
  \vn& \bigl\{ \substack{9\tau+6\\-6\tau\xi-3\xi} \bigl\} & \vn&
  \bigl\{ \substack{-2\tau-1\\+7\tau\xi+5\xi} \bigr\} &
  \bigl\{ \substack{ -3\tau -1\\+3\tau\xi-2\xi} \bigr\} &\vn & \\
  \vn & \vn & \bigl\{ \substack{6\tau +3 \\+3\tau\xi+3\xi}\bigr\} &
  \vn & \bigl\{\substack{-7\tau-5\\+5\tau\xi+4\xi} \bigr\} &
  \bigl \{ \substack{-3\tau-2\\+\xi} \bigr\} \\
  \bigl\{ \substack{-1\\-3\tau\xi-\xi} \bigr\} &\vn & \vn & \bigl\{
  \substack{-3\tau -3\\+9\tau\xi+6\xi} \bigr\} & \vn & \bigl\{
  \substack{-5\tau-4\\-2\tau\xi-\xi} \bigr\}
\end{pmatrix}.
\]
	
\bigskip

\section*{Acknowledgements}
	
It is our pleasure to thank Lorenzo Sadun for his cooperation and
valuable comments on the manuscript.  This work was supported by the
German Research Council (Deutsche Forschungsgemeinschaft, DFG) under
CRC 1283/2 (2021 - 317210226).  AM also acknowledges support from
EPSRC grant EP/Y023358/1 and thanks Bielefeld University for
hospitality during an extended research visit in Winter 2023. JM
thanks University of Birmingham for hospitality during his research
visit in July 2024. We are grateful to an anonymous referee for his
thoughtful comments which helped us to improve the presentation.

\end{document}

%% file: reprojection.tex
\tikzset{every picture/.style={line width=0.6pt}} 

\begin{tikzpicture}[x=0.6pt,y=0.6pt,yscale=-1,xscale=1]

\draw  [draw opacity=0][fill={rgb, 255:red, 155; green, 155; blue, 155 }  ,fill opacity=0.47 ] (8.35,129.6) -- (297.53,129.6) -- (297.53,186.32) -- (8.35,186.32) -- cycle ;
\draw[stealth-]  (39.88,50.57) -- (39.88,236.83) ;
\draw[stealth-]    (279.13,229.26) -- (32.31,229.26) ;

\draw  [fill={rgb, 255:red, 0; green, 0; blue, 0 }  ,fill opacity=1 ] (68.95,67.37) .. controls (68.95,65.94) and (70.11,64.78) .. (71.54,64.78) .. controls (72.97,64.78) and (74.13,65.94) .. (74.13,67.37) .. controls (74.13,68.8) and (72.97,69.96) .. (71.54,69.96) .. controls (70.11,69.96) and (68.95,68.8) .. (68.95,67.37) -- cycle ;
\draw  [fill={rgb, 255:red, 0; green, 0; blue, 0 }  ,fill opacity=1 ] (58.47,121.28) .. controls (58.47,119.85) and (59.63,118.69) .. (61.06,118.69) .. controls (62.49,118.69) and (63.65,119.85) .. (63.65,121.28) .. controls (63.65,122.72) and (62.49,123.88) .. (61.06,123.88) .. controls (59.63,123.88) and (58.47,122.72) .. (58.47,121.28) -- cycle ;
\draw  [fill={rgb, 255:red, 0; green, 0; blue, 0 }  ,fill opacity=1 ] (47.99,175.2) .. controls (47.99,173.77) and (49.15,172.61) .. (50.58,172.61) .. controls (52.01,172.61) and (53.17,173.77) .. (53.17,175.2) .. controls (53.17,176.63) and (52.01,177.79) .. (50.58,177.79) .. controls (49.15,177.79) and (47.99,176.63) .. (47.99,175.2) -- cycle ;
\draw  [fill={rgb, 255:red, 0; green, 0; blue, 0 }  ,fill opacity=1 ] (37.51,229.11) .. controls (37.51,227.68) and (38.67,226.52) .. (40.1,226.52) .. controls (41.53,226.52) and (42.69,227.68) .. (42.69,229.11) .. controls (42.69,230.54) and (41.53,231.7) .. (40.1,231.7) .. controls (38.67,231.7) and (37.51,230.54) .. (37.51,229.11) -- cycle ;
\draw  [fill={rgb, 255:red, 0; green, 0; blue, 0 }  ,fill opacity=1 ] (83.19,254.44) .. controls (83.19,253) and (84.35,251.84) .. (85.79,251.84) .. controls (87.22,251.84) and (88.38,253) .. (88.38,254.44) .. controls (88.38,255.87) and (87.22,257.03) .. (85.79,257.03) .. controls (84.35,257.03) and (83.19,255.87) .. (83.19,254.44) -- cycle ;
\draw  [fill={rgb, 255:red, 0; green, 0; blue, 0 }  ,fill opacity=1 ] (93.67,200.52) .. controls (93.67,199.09) and (94.83,197.93) .. (96.27,197.93) .. controls (97.7,197.93) and (98.86,199.09) .. (98.86,200.52) .. controls (98.86,201.95) and (97.7,203.12) .. (96.27,203.12) .. controls (94.83,203.12) and (93.67,201.95) .. (93.67,200.52) -- cycle ;
\draw  [fill={rgb, 255:red, 0; green, 0; blue, 0 }  ,fill opacity=1 ] (104.15,146.61) .. controls (104.15,145.18) and (105.31,144.02) .. (106.75,144.02) .. controls (108.18,144.02) and (109.34,145.18) .. (109.34,146.61) .. controls (109.34,148.04) and (108.18,149.2) .. (106.75,149.2) .. controls (105.31,149.2) and (104.15,148.04) .. (104.15,146.61) -- cycle ;
\draw  [fill={rgb, 255:red, 0; green, 0; blue, 0 }  ,fill opacity=1 ] (114.63,92.69) .. controls (114.63,91.26) and (115.79,90.1) .. (117.23,90.1) .. controls (118.66,90.1) and (119.82,91.26) .. (119.82,92.69) .. controls (119.82,94.13) and (118.66,95.29) .. (117.23,95.29) .. controls (115.79,95.29) and (114.63,94.13) .. (114.63,92.69) -- cycle ;
\draw  [fill={rgb, 255:red, 0; green, 0; blue, 0 }  ,fill opacity=1 ] (160.32,118.02) .. controls (160.32,116.59) and (161.48,115.43) .. (162.91,115.43) .. controls (164.34,115.43) and (165.5,116.59) .. (165.5,118.02) .. controls (165.5,119.45) and (164.34,120.61) .. (162.91,120.61) .. controls (161.48,120.61) and (160.32,119.45) .. (160.32,118.02) -- cycle ;
\draw  [fill={rgb, 255:red, 0; green, 0; blue, 0 }  ,fill opacity=1 ] (12.78,95.96) .. controls (12.78,94.53) and (13.94,93.37) .. (15.37,93.37) .. controls (16.81,93.37) and (17.97,94.53) .. (17.97,95.96) .. controls (17.97,97.39) and (16.81,98.55) .. (15.37,98.55) .. controls (13.94,98.55) and (12.78,97.39) .. (12.78,95.96) -- cycle ;
\draw  [fill={rgb, 255:red, 0; green, 0; blue, 0 }  ,fill opacity=1 ] (149.84,171.93) .. controls (149.84,170.5) and (151,169.34) .. (152.43,169.34) .. controls (153.86,169.34) and (155.02,170.5) .. (155.02,171.93) .. controls (155.02,173.36) and (153.86,174.53) .. (152.43,174.53) .. controls (151,174.53) and (149.84,173.36) .. (149.84,171.93) -- cycle ;
\draw  [fill={rgb, 255:red, 0; green, 0; blue, 0 }  ,fill opacity=1 ] (139.36,225.85) .. controls (139.36,224.41) and (140.52,223.25) .. (141.95,223.25) .. controls (143.38,223.25) and (144.54,224.41) .. (144.54,225.85) .. controls (144.54,227.28) and (143.38,228.44) .. (141.95,228.44) .. controls (140.52,228.44) and (139.36,227.28) .. (139.36,225.85) -- cycle ;
\draw  [fill={rgb, 255:red, 0; green, 0; blue, 0 }  ,fill opacity=1 ] (195.52,197.26) .. controls (195.52,195.82) and (196.68,194.66) .. (198.12,194.66) .. controls (199.55,194.66) and (200.71,195.82) .. (200.71,197.26) .. controls (200.71,198.69) and (199.55,199.85) .. (198.12,199.85) .. controls (196.68,199.85) and (195.52,198.69) .. (195.52,197.26) -- cycle ;
\draw  [fill={rgb, 255:red, 0; green, 0; blue, 0 }  ,fill opacity=1 ] (206,143.34) .. controls (206,141.91) and (207.16,140.75) .. (208.6,140.75) .. controls (210.03,140.75) and (211.19,141.91) .. (211.19,143.34) .. controls (211.19,144.77) and (210.03,145.94) .. (208.6,145.94) .. controls (207.16,145.94) and (206,144.77) .. (206,143.34) -- cycle ;
\draw  [fill={rgb, 255:red, 0; green, 0; blue, 0 }  ,fill opacity=1 ] (216.48,89.43) .. controls (216.48,88) and (217.64,86.84) .. (219.08,86.84) .. controls (220.51,86.84) and (221.67,88) .. (221.67,89.43) .. controls (221.67,90.86) and (220.51,92.02) .. (219.08,92.02) .. controls (217.64,92.02) and (216.48,90.86) .. (216.48,89.43) -- cycle ;
\draw  [fill={rgb, 255:red, 0; green, 0; blue, 0 }  ,fill opacity=1 ] (262.17,114.75) .. controls (262.17,113.32) and (263.33,112.16) .. (264.76,112.16) .. controls (266.19,112.16) and (267.35,113.32) .. (267.35,114.75) .. controls (267.35,116.18) and (266.19,117.35) .. (264.76,117.35) .. controls (263.33,117.35) and (262.17,116.18) .. (262.17,114.75) -- cycle ;
\draw  [fill={rgb, 255:red, 0; green, 0; blue, 0 }  ,fill opacity=1 ] (251.69,168.67) .. controls (251.69,167.23) and (252.85,166.07) .. (254.28,166.07) .. controls (255.71,166.07) and (256.87,167.23) .. (256.87,168.67) .. controls (256.87,170.1) and (255.71,171.26) .. (254.28,171.26) .. controls (252.85,171.26) and (251.69,170.1) .. (251.69,168.67) -- cycle ;
\draw  [fill={rgb, 255:red, 0; green, 0; blue, 0 }  ,fill opacity=1 ] (241.21,222.58) .. controls (241.21,221.15) and (242.37,219.99) .. (243.8,219.99) .. controls (245.23,219.99) and (246.39,221.15) .. (246.39,222.58) .. controls (246.39,224.01) and (245.23,225.17) .. (243.8,225.17) .. controls (242.37,225.17) and (241.21,224.01) .. (241.21,222.58) -- cycle ;
\draw  [fill={rgb, 255:red, 0; green, 0; blue, 0 }  ,fill opacity=1 ] (185.04,251.17) .. controls (185.04,249.74) and (186.2,248.58) .. (187.64,248.58) .. controls (189.07,248.58) and (190.23,249.74) .. (190.23,251.17) .. controls (190.23,252.6) and (189.07,253.76) .. (187.64,253.76) .. controls (186.2,253.76) and (185.04,252.6) .. (185.04,251.17) -- cycle ;
\draw  [fill={rgb, 255:red, 0; green, 0; blue, 0 }  ,fill opacity=1 ] (170.8,64.1) .. controls (170.8,62.67) and (171.96,61.51) .. (173.39,61.51) .. controls (174.82,61.51) and (175.98,62.67) .. (175.98,64.1) .. controls (175.98,65.54) and (174.82,66.7) .. (173.39,66.7) .. controls (171.96,66.7) and (170.8,65.54) .. (170.8,64.1) -- cycle ;
\draw    (50.58,175.2) -- (50.1,229.62) ;
\draw    (208.6,143.34) -- (208.25,229.22) ;
\draw    (254.28,168.67) -- (253.76,229.03) ;
\draw    (106.75,146.61) -- (106.8,229.29) ;
\draw    (152.43,171.93) -- (152.6,228.84) ;
\draw  [fill={rgb, 255:red, 255; green, 255; blue, 255 }  ,fill opacity=1 ] (50.18,226.96) -- (52.38,229.15) -- (50.18,231.34) -- (47.99,229.15) -- cycle ;
\draw  [fill={rgb, 255:red, 255; green, 255; blue, 255 }  ,fill opacity=1 ] (106.8,227.1) -- (109,229.29) -- (106.8,231.49) -- (104.61,229.29) -- cycle ;
\draw  [fill={rgb, 255:red, 255; green, 255; blue, 255 }  ,fill opacity=1 ] (152.7,227.11) -- (154.89,229.3) -- (152.7,231.49) -- (150.51,229.3) -- cycle ;
\draw  [fill={rgb, 255:red, 255; green, 255; blue, 255 }  ,fill opacity=1 ] (208.25,227.03) -- (210.44,229.22) -- (208.25,231.41) -- (206.05,229.22) -- cycle ;
\draw  [fill={rgb, 255:red, 255; green, 255; blue, 255 }  ,fill opacity=1 ] (253.84,227.24) -- (256.03,229.43) -- (253.84,231.62) -- (251.65,229.43) -- cycle ;

\draw[-stealth]   (282.13,55.14) -- (282.13,89.99) ;
\draw[-stealth]    (289.7,62.71) -- (254.85,62.71) ;

\draw [line width=2.25]    (39.71,129.5) -- (39.62,186.26) ;

\draw  [draw opacity=0][fill={rgb, 255:red, 155; green, 155; blue, 155 }  ,fill opacity=0.47 ] (370.16,129.6) -- (659.34,129.6) -- (659.34,186.32) -- (370.16,186.32) -- cycle ;
\draw[stealth-]   (401.69,50.57) -- (401.69,236.83) ;

\draw[stealth-]    (640.93,229.26) -- (394.12,229.26) ;

\draw  [fill={rgb, 255:red, 0; green, 0; blue, 0 }  ,fill opacity=1 ] (430.75,67.37) .. controls (430.75,65.94) and (431.91,64.78) .. (433.35,64.78) .. controls (434.78,64.78) and (435.94,65.94) .. (435.94,67.37) .. controls (435.94,68.8) and (434.78,69.96) .. (433.35,69.96) .. controls (431.91,69.96) and (430.75,68.8) .. (430.75,67.37) -- cycle ;
\draw  [fill={rgb, 255:red, 0; green, 0; blue, 0 }  ,fill opacity=1 ] (420.27,121.28) .. controls (420.27,119.85) and (421.43,118.69) .. (422.87,118.69) .. controls (424.3,118.69) and (425.46,119.85) .. (425.46,121.28) .. controls (425.46,122.72) and (424.3,123.88) .. (422.87,123.88) .. controls (421.43,123.88) and (420.27,122.72) .. (420.27,121.28) -- cycle ;
\draw  [fill={rgb, 255:red, 0; green, 0; blue, 0 }  ,fill opacity=1 ] (409.79,175.2) .. controls (409.79,173.77) and (410.95,172.61) .. (412.39,172.61) .. controls (413.82,172.61) and (414.98,173.77) .. (414.98,175.2) .. controls (414.98,176.63) and (413.82,177.79) .. (412.39,177.79) .. controls (410.95,177.79) and (409.79,176.63) .. (409.79,175.2) -- cycle ;
\draw  [fill={rgb, 255:red, 0; green, 0; blue, 0 }  ,fill opacity=1 ] (399.31,229.11) .. controls (399.31,227.68) and (400.47,226.52) .. (401.91,226.52) .. controls (403.34,226.52) and (404.5,227.68) .. (404.5,229.11) .. controls (404.5,230.54) and (403.34,231.7) .. (401.91,231.7) .. controls (400.47,231.7) and (399.31,230.54) .. (399.31,229.11) -- cycle ;
\draw  [fill={rgb, 255:red, 0; green, 0; blue, 0 }  ,fill opacity=1 ] (445,254.44) .. controls (445,253) and (446.16,251.84) .. (447.59,251.84) .. controls (449.02,251.84) and (450.18,253) .. (450.18,254.44) .. controls (450.18,255.87) and (449.02,257.03) .. (447.59,257.03) .. controls (446.16,257.03) and (445,255.87) .. (445,254.44) -- cycle ;
\draw  [fill={rgb, 255:red, 0; green, 0; blue, 0 }  ,fill opacity=1 ] (455.48,200.52) .. controls (455.48,199.09) and (456.64,197.93) .. (458.07,197.93) .. controls (459.5,197.93) and (460.66,199.09) .. (460.66,200.52) .. controls (460.66,201.95) and (459.5,203.12) .. (458.07,203.12) .. controls (456.64,203.12) and (455.48,201.95) .. (455.48,200.52) -- cycle ;
\draw  [fill={rgb, 255:red, 0; green, 0; blue, 0 }  ,fill opacity=1 ] (465.96,146.61) .. controls (465.96,145.18) and (467.12,144.02) .. (468.55,144.02) .. controls (469.98,144.02) and (471.14,145.18) .. (471.14,146.61) .. controls (471.14,148.04) and (469.98,149.2) .. (468.55,149.2) .. controls (467.12,149.2) and (465.96,148.04) .. (465.96,146.61) -- cycle ;
\draw  [fill={rgb, 255:red, 0; green, 0; blue, 0 }  ,fill opacity=1 ] (476.44,92.69) .. controls (476.44,91.26) and (477.6,90.1) .. (479.03,90.1) .. controls (480.46,90.1) and (481.62,91.26) .. (481.62,92.69) .. controls (481.62,94.13) and (480.46,95.29) .. (479.03,95.29) .. controls (477.6,95.29) and (476.44,94.13) .. (476.44,92.69) -- cycle ;
\draw  [fill={rgb, 255:red, 0; green, 0; blue, 0 }  ,fill opacity=1 ] (522.12,118.02) .. controls (522.12,116.59) and (523.28,115.43) .. (524.72,115.43) .. controls (526.15,115.43) and (527.31,116.59) .. (527.31,118.02) .. controls (527.31,119.45) and (526.15,120.61) .. (524.72,120.61) .. controls (523.28,120.61) and (522.12,119.45) .. (522.12,118.02) -- cycle ;
\draw  [fill={rgb, 255:red, 0; green, 0; blue, 0 }  ,fill opacity=1 ] (374.59,95.96) .. controls (374.59,94.53) and (375.75,93.37) .. (377.18,93.37) .. controls (378.61,93.37) and (379.77,94.53) .. (379.77,95.96) .. controls (379.77,97.39) and (378.61,98.55) .. (377.18,98.55) .. controls (375.75,98.55) and (374.59,97.39) .. (374.59,95.96) -- cycle ;
\draw  [fill={rgb, 255:red, 0; green, 0; blue, 0 }  ,fill opacity=1 ] (511.64,171.93) .. controls (511.64,170.5) and (512.8,169.34) .. (514.24,169.34) .. controls (515.67,169.34) and (516.83,170.5) .. (516.83,171.93) .. controls (516.83,173.36) and (515.67,174.53) .. (514.24,174.53) .. controls (512.8,174.53) and (511.64,173.36) .. (511.64,171.93) -- cycle ;
\draw  [fill={rgb, 255:red, 0; green, 0; blue, 0 }  ,fill opacity=1 ] (501.16,225.85) .. controls (501.16,224.41) and (502.32,223.25) .. (503.76,223.25) .. controls (505.19,223.25) and (506.35,224.41) .. (506.35,225.85) .. controls (506.35,227.28) and (505.19,228.44) .. (503.76,228.44) .. controls (502.32,228.44) and (501.16,227.28) .. (501.16,225.85) -- cycle ;
\draw  [fill={rgb, 255:red, 0; green, 0; blue, 0 }  ,fill opacity=1 ] (557.33,197.26) .. controls (557.33,195.82) and (558.49,194.66) .. (559.92,194.66) .. controls (561.35,194.66) and (562.51,195.82) .. (562.51,197.26) .. controls (562.51,198.69) and (561.35,199.85) .. (559.92,199.85) .. controls (558.49,199.85) and (557.33,198.69) .. (557.33,197.26) -- cycle ;
\draw  [fill={rgb, 255:red, 0; green, 0; blue, 0 }  ,fill opacity=1 ] (567.81,143.34) .. controls (567.81,141.91) and (568.97,140.75) .. (570.4,140.75) .. controls (571.83,140.75) and (572.99,141.91) .. (572.99,143.34) .. controls (572.99,144.77) and (571.83,145.94) .. (570.4,145.94) .. controls (568.97,145.94) and (567.81,144.77) .. (567.81,143.34) -- cycle ;
\draw  [fill={rgb, 255:red, 0; green, 0; blue, 0 }  ,fill opacity=1 ] (578.29,89.43) .. controls (578.29,88) and (579.45,86.84) .. (580.88,86.84) .. controls (582.31,86.84) and (583.47,88) .. (583.47,89.43) .. controls (583.47,90.86) and (582.31,92.02) .. (580.88,92.02) .. controls (579.45,92.02) and (578.29,90.86) .. (578.29,89.43) -- cycle ;
\draw  [fill={rgb, 255:red, 0; green, 0; blue, 0 }  ,fill opacity=1 ] (623.97,114.75) .. controls (623.97,113.32) and (625.13,112.16) .. (626.57,112.16) .. controls (628,112.16) and (629.16,113.32) .. (629.16,114.75) .. controls (629.16,116.18) and (628,117.35) .. (626.57,117.35) .. controls (625.13,117.35) and (623.97,116.18) .. (623.97,114.75) -- cycle ;
\draw  [fill={rgb, 255:red, 0; green, 0; blue, 0 }  ,fill opacity=1 ] (613.49,168.67) .. controls (613.49,167.23) and (614.66,166.07) .. (616.09,166.07) .. controls (617.52,166.07) and (618.68,167.23) .. (618.68,168.67) .. controls (618.68,170.1) and (617.52,171.26) .. (616.09,171.26) .. controls (614.66,171.26) and (613.49,170.1) .. (613.49,168.67) -- cycle ;
\draw  [fill={rgb, 255:red, 0; green, 0; blue, 0 }  ,fill opacity=1 ] (603.01,222.58) .. controls (603.01,221.15) and (604.18,219.99) .. (605.61,219.99) .. controls (607.04,219.99) and (608.2,221.15) .. (608.2,222.58) .. controls (608.2,224.01) and (607.04,225.17) .. (605.61,225.17) .. controls (604.18,225.17) and (603.01,224.01) .. (603.01,222.58) -- cycle ;
\draw  [fill={rgb, 255:red, 0; green, 0; blue, 0 }  ,fill opacity=1 ] (546.85,251.17) .. controls (546.85,249.74) and (548.01,248.58) .. (549.44,248.58) .. controls (550.87,248.58) and (552.03,249.74) .. (552.03,251.17) .. controls (552.03,252.6) and (550.87,253.76) .. (549.44,253.76) .. controls (548.01,253.76) and (546.85,252.6) .. (546.85,251.17) -- cycle ;
\draw  [fill={rgb, 255:red, 0; green, 0; blue, 0 }  ,fill opacity=1 ] (532.6,64.1) .. controls (532.6,62.67) and (533.76,61.51) .. (535.2,61.51) .. controls (536.63,61.51) and (537.79,62.67) .. (537.79,64.1) .. controls (537.79,65.54) and (536.63,66.7) .. (535.2,66.7) .. controls (533.76,66.7) and (532.6,65.54) .. (532.6,64.1) -- cycle ;
\draw [color={rgb, 255:red, 0; green, 0; blue, 0 }  ,draw opacity=0.3 ][fill={rgb, 255:red, 0; green, 0; blue, 0 }  ,fill opacity=0.35 ] [dash pattern={on 0.84pt off 2.51pt}]  (412.39,175.2) -- (411.91,229.62) ;
\draw [color={rgb, 255:red, 0; green, 0; blue, 0 }  ,draw opacity=0.3 ] [dash pattern={on 0.84pt off 2.51pt}]  (570.4,143.34) -- (570.05,229.22) ;
\draw [color={rgb, 255:red, 0; green, 0; blue, 0 }  ,draw opacity=0.3 ] [dash pattern={on 0.84pt off 2.51pt}]  (616.09,168.67) -- (615.57,229.03) ;
\draw [color={rgb, 255:red, 0; green, 0; blue, 0 }  ,draw opacity=0.3 ] [dash pattern={on 0.84pt off 2.51pt}]  (468.55,146.61) -- (468.61,229.29) ;
\draw [color={rgb, 255:red, 0; green, 0; blue, 0 }  ,draw opacity=0.3 ] [dash pattern={on 0.84pt off 2.51pt}]  (514.24,171.93) -- (514.41,228.84) ;
\draw[-stealth]    (641.86,55.52) -- (651.53,90.06) ;
\draw[-stealth]    (651.5,62.71) -- (616.65,62.71) ;

\draw [line width=2.25]    (401.52,129.5) -- (401.42,186.26) ;
\draw    (412.39,175.2) -- (427.55,229.15) ;
\draw    (468.55,146.61) -- (491.72,229.29) ;
\draw    (514.24,171.93) -- (530.51,229.3) ;
\draw    (570.4,143.34) -- (595.16,229.22) ;
\draw    (616.09,168.67) -- (633.65,229.21) ;
\draw  [color={rgb, 255:red, 128; green, 128; blue, 128 }  ,draw opacity=1 ][fill={rgb, 255:red, 128; green, 128; blue, 128 }  ,fill opacity=1 ] (427.55,226.96) -- (429.74,229.15) -- (427.55,231.34) -- (425.35,229.15) -- cycle ;
\draw  [color={rgb, 255:red, 128; green, 128; blue, 128 }  ,draw opacity=1 ][fill={rgb, 255:red, 128; green, 128; blue, 128 }  ,fill opacity=1 ] (491.72,227.1) -- (493.91,229.29) -- (491.72,231.49) -- (489.53,229.29) -- cycle ;
\draw  [color={rgb, 255:red, 128; green, 128; blue, 128 }  ,draw opacity=1 ][fill={rgb, 255:red, 128; green, 128; blue, 128 }  ,fill opacity=1 ] (530.51,227.11) -- (532.7,229.3) -- (530.51,231.49) -- (528.32,229.3) -- cycle ;
\draw  [color={rgb, 255:red, 128; green, 128; blue, 128 }  ,draw opacity=1 ][fill={rgb, 255:red, 128; green, 128; blue, 128 }  ,fill opacity=1 ] (595.16,227.03) -- (597.35,229.22) -- (595.16,231.41) -- (592.97,229.22) -- cycle ;
\draw  [color={rgb, 255:red, 128; green, 128; blue, 128 }  ,draw opacity=1 ][fill={rgb, 255:red, 128; green, 128; blue, 128 }  ,fill opacity=1 ] (633.65,227.01) -- (635.84,229.21) -- (633.65,231.4) -- (631.46,229.21) -- cycle ;

\draw (285,70) node [anchor=north west][inner sep=0.75pt]    {$\scriptstyle \pi $};
\draw (255,45) node [anchor=north west][inner sep=0.75pt]    {$\scriptstyle\pi^{}_{\mathrm{int}}$};
\draw (15,150) node [anchor=north west][inner sep=0.75pt]    {$W$};
\draw (618,45) node [anchor=north west][inner sep=0.75pt]    {$\scriptstyle\pi^{}_{\mathrm{int}}$};
\draw (378,150) node [anchor=north west][inner sep=0.75pt]    {$W$};
\draw (652,70) node [anchor=north west][inner sep=0.75pt]    {$\scriptstyle \pi'$};

\end{tikzpicture}

%% file: CAPgen.tex
\tikzset{every picture/.style={line width=0.75pt}} 

\begin{tikzpicture}[x=0.7pt,y=0.7pt,yscale=-1,xscale=1]

\draw [color={rgb, 255:red, 0; green, 50; blue, 250 }  ,draw opacity=1 ]   (207.96,168.29) -- (265.38,168.62) ;
\draw [color={rgb, 255:red, 0; green, 50; blue, 250 }  ,draw opacity=1 ]   (251.57,109.21) -- (312.26,108.94) ;
\draw [color={rgb, 255:red, 0; green, 50; blue, 250 }  ,draw opacity=1 ]   (178.46,98.93) -- (235.87,99.26) ;
\draw [color={rgb, 255:red, 0; green, 50; blue, 250 }  ,draw opacity=1 ]   (161.8,109.46) -- (189.83,158.53) ;
\draw [color={rgb, 255:red, 0; green, 50; blue, 250 }  ,draw opacity=1 ]   (189.62,178.78) -- (218.33,228.81) ;
\draw [color={rgb, 255:red, 0; green, 50; blue, 250 }  ,draw opacity=1 ]   (312.36,127.9) -- (282.35,178.59) ;
\draw [color={rgb, 255:red, 245; green, 166; blue, 35 }  ,draw opacity=1 ]   (207.96,168.29) -- (236.82,118.68) ;
\draw [color={rgb, 255:red, 245; green, 166; blue, 35 }  ,draw opacity=1 ]   (265.38,168.62) -- (236.82,118.68) ;
\draw [color={rgb, 255:red, 245; green, 166; blue, 35 }  ,draw opacity=1 ]   (171.69,210.84) -- (143.13,160.91) ;
\draw [color={rgb, 255:red, 245; green, 166; blue, 35 }  ,draw opacity=1 ]   (235.34,238.48) -- (265.35,188.13) ;
\draw [color={rgb, 255:red, 245; green, 166; blue, 35 }  ,draw opacity=1 ]   (216.12,65.56) -- (276.13,65.31) ;
\draw  [color={rgb, 255:red, 0; green, 50; blue, 250 }  ,draw opacity=1 ][fill={rgb, 255:red, 0; green, 50; blue, 250 }  ,fill opacity=1 ] (196.81,153.33) .. controls (196.81,152.24) and (197.69,151.35) .. (198.78,151.35) .. controls (199.87,151.35) and (200.76,152.24) .. (200.76,153.33) .. controls (200.76,154.42) and (199.87,155.3) .. (198.78,155.3) .. controls (197.69,155.3) and (196.81,154.42) .. (196.81,153.33) -- cycle ;
\draw  [color={rgb, 255:red, 0; green, 50; blue, 250 }  ,draw opacity=1 ][fill={rgb, 255:red, 0; green, 50; blue, 250 }  ,fill opacity=1 ] (253.23,118.89) .. controls (253.23,117.8) and (254.11,116.92) .. (255.2,116.92) .. controls (256.29,116.92) and (257.17,117.8) .. (257.17,118.89) .. controls (257.17,119.98) and (256.29,120.86) .. (255.2,120.86) .. controls (254.11,120.86) and (253.23,119.98) .. (253.23,118.89) -- cycle ;
\draw  [color={rgb, 255:red, 0; green, 50; blue, 250 }  ,draw opacity=1 ][fill={rgb, 255:red, 0; green, 50; blue, 250 }  ,fill opacity=1 ] (224.2,222.62) .. controls (224.2,221.53) and (225.08,220.65) .. (226.17,220.65) .. controls (227.26,220.65) and (228.14,221.53) .. (228.14,222.62) .. controls (228.14,223.71) and (227.26,224.59) .. (226.17,224.59) .. controls (225.08,224.59) and (224.2,223.71) .. (224.2,222.62) -- cycle ;
\draw  [color={rgb, 255:red, 245; green, 166; blue, 35 }  ,draw opacity=1 ][fill={rgb, 255:red, 245; green, 166; blue, 35 }  ,fill opacity=1 ] (253.23,185.3) .. controls (253.23,184.21) and (254.11,183.33) .. (255.2,183.33) .. controls (256.29,183.33) and (257.17,184.21) .. (257.17,185.3) .. controls (257.17,186.39) and (256.29,187.28) .. (255.2,187.28) .. controls (254.11,187.28) and (253.23,186.39) .. (253.23,185.3) -- cycle ;
\draw  [color={rgb, 255:red, 189; green, 16; blue, 224 }  ,draw opacity=1 ][fill={rgb, 255:red, 189; green, 16; blue, 224 }  ,fill opacity=1 ] (207.13,178.86) .. controls (207.13,177.77) and (208.01,176.89) .. (209.1,176.89) .. controls (210.19,176.89) and (211.08,177.77) .. (211.08,178.86) .. controls (211.08,179.95) and (210.19,180.83) .. (209.1,180.83) .. controls (208.01,180.83) and (207.13,179.95) .. (207.13,178.86) -- cycle ;
\draw  [color={rgb, 255:red, 189; green, 16; blue, 224 }  ,draw opacity=1 ][fill={rgb, 255:red, 189; green, 16; blue, 224 }  ,fill opacity=1 ] (271.12,162.08) .. controls (271.12,160.99) and (272,160.1) .. (273.09,160.1) .. controls (274.18,160.1) and (275.06,160.99) .. (275.06,162.08) .. controls (275.06,163.17) and (274.18,164.05) .. (273.09,164.05) .. controls (272,164.05) and (271.12,163.17) .. (271.12,162.08) -- cycle ;
\draw  [color={rgb, 255:red, 189; green, 16; blue, 224 }  ,draw opacity=1 ][fill={rgb, 255:red, 189; green, 16; blue, 224 }  ,fill opacity=1 ] (225.01,113.35) .. controls (225.01,112.26) and (225.89,111.38) .. (226.98,111.38) .. controls (228.07,111.38) and (228.96,112.26) .. (228.96,113.35) .. controls (228.96,114.44) and (228.07,115.32) .. (226.98,115.32) .. controls (225.89,115.32) and (225.01,114.44) .. (225.01,113.35) -- cycle ;
\draw  [color={rgb, 255:red, 126; green, 211; blue, 33 }  ,draw opacity=1 ][fill={rgb, 255:red, 126; green, 211; blue, 33 }  ,fill opacity=1 ] (177.74,108.41) .. controls (177.74,107.32) and (178.63,106.43) .. (179.72,106.43) .. controls (180.81,106.43) and (181.69,107.32) .. (181.69,108.41) .. controls (181.69,109.5) and (180.81,110.38) .. (179.72,110.38) .. controls (178.63,110.38) and (177.74,109.5) .. (177.74,108.41) -- cycle ;
\draw  [color={rgb, 255:red, 126; green, 211; blue, 33 }  ,draw opacity=1 ][fill={rgb, 255:red, 126; green, 211; blue, 33 }  ,fill opacity=1 ] (299.84,125.04) .. controls (299.84,123.95) and (300.72,123.07) .. (301.81,123.07) .. controls (302.9,123.07) and (303.78,123.95) .. (303.78,125.04) .. controls (303.78,126.13) and (302.9,127.02) .. (301.81,127.02) .. controls (300.72,127.02) and (299.84,126.13) .. (299.84,125.04) -- cycle ;
\draw [color={rgb, 255:red, 126; green, 211; blue, 33 }  ,draw opacity=1 ]   (265.38,168.62) -- (271.03,178.41) ;
\draw [color={rgb, 255:red, 126; green, 211; blue, 33 }  ,draw opacity=1 ]   (271.03,178.41) -- (265.35,188.13) ;
\draw [color={rgb, 255:red, 126; green, 211; blue, 33 }  ,draw opacity=1 ]   (282.35,178.59) -- (271.03,178.41) ;
\draw [color={rgb, 255:red, 126; green, 211; blue, 33 }  ,draw opacity=1 ]   (189.83,158.53) -- (195.48,168.32) ;
\draw [color={rgb, 255:red, 126; green, 211; blue, 33 }  ,draw opacity=1 ]   (195.48,168.32) -- (189.62,178.78) ;
\draw [color={rgb, 255:red, 126; green, 211; blue, 33 }  ,draw opacity=1 ]   (207.96,168.29) -- (195.48,168.32) ;
\draw [color={rgb, 255:red, 126; green, 211; blue, 33 }  ,draw opacity=1 ]   (235.87,99.26) -- (241.52,109.05) ;
\draw [color={rgb, 255:red, 126; green, 211; blue, 33 }  ,draw opacity=1 ]   (241.52,109.05) -- (236.82,118.68) ;
\draw [color={rgb, 255:red, 126; green, 211; blue, 33 }  ,draw opacity=1 ]   (251.57,109.21) -- (241.52,109.05) ;
\draw [color={rgb, 255:red, 126; green, 211; blue, 33 }  ,draw opacity=1 ]   (312.26,108.94) -- (317.91,118.73) ;
\draw [color={rgb, 255:red, 126; green, 211; blue, 33 }  ,draw opacity=1 ]   (317.41,117.97) -- (312.36,127.9) ;
\draw [color={rgb, 255:red, 126; green, 211; blue, 33 }  ,draw opacity=1 ]   (178.46,98.93) -- (167.3,98.96) ;
\draw [color={rgb, 255:red, 126; green, 211; blue, 33 }  ,draw opacity=1 ]   (167.3,98.96) -- (161.8,109.46) ;
\draw [color={rgb, 255:red, 126; green, 211; blue, 33 }  ,draw opacity=1 ]   (218.33,228.81) -- (223.98,238.6) ;
\draw [color={rgb, 255:red, 126; green, 211; blue, 33 }  ,draw opacity=1 ]   (235.34,238.48) -- (223.98,238.6) ;
\draw [color={rgb, 255:red, 126; green, 211; blue, 33 }  ,draw opacity=1 ]   (163.47,55.57) -- (151.03,55.71) ;
\draw [color={rgb, 255:red, 126; green, 211; blue, 33 }  ,draw opacity=1 ]   (211.13,56.14) -- (201.08,55.98) ;
\draw [color={rgb, 255:red, 126; green, 211; blue, 33 }  ,draw opacity=1 ]   (287.52,65.38) -- (276.13,65.31) ;
\draw [color={rgb, 255:red, 126; green, 211; blue, 33 }  ,draw opacity=1 ]   (165.95,220.8) -- (171.36,211.38) ;
\draw [color={rgb, 255:red, 126; green, 211; blue, 33 }  ,draw opacity=1 ]   (170.95,230.22) -- (165.95,220.8) ;
\draw [color={rgb, 255:red, 126; green, 211; blue, 33 }  ,draw opacity=1 ]   (196.05,274.26) -- (191.06,264.84) ;
\draw [color={rgb, 255:red, 126; green, 211; blue, 33 }  ,draw opacity=1 ]   (292.51,74.8) -- (287.52,65.38) ;
\draw [color={rgb, 255:red, 126; green, 211; blue, 33 }  ,draw opacity=1 ]   (216.12,65.56) -- (211.13,56.14) ;
\draw [color={rgb, 255:red, 126; green, 211; blue, 33 }  ,draw opacity=1 ]   (332.34,178.96) -- (322.3,178.8) ;
\draw [color={rgb, 255:red, 126; green, 211; blue, 33 }  ,draw opacity=1 ]   (337.79,169.56) -- (332.34,178.96) ;
\draw [color={rgb, 255:red, 126; green, 211; blue, 33 }  ,draw opacity=1 ]   (362.53,126.12) -- (357.08,135.52) ;
\draw [color={rgb, 255:red, 155; green, 155; blue, 155 }  ,draw opacity=1 ]   (357.08,135.52) -- (337.79,169.56) ;
\draw [color={rgb, 255:red, 155; green, 155; blue, 155 }  ,draw opacity=1 ]   (189.62,178.78) -- (171.36,211.38) ;
\draw [color={rgb, 255:red, 155; green, 155; blue, 155 }  ,draw opacity=1 ]   (161.8,109.46) -- (143.54,142.07) ;
\draw [color={rgb, 255:red, 126; green, 211; blue, 33 }  ,draw opacity=1 ]   (138.13,151.49) -- (143.54,142.07) ;
\draw [color={rgb, 255:red, 126; green, 211; blue, 33 }  ,draw opacity=1 ]   (143.13,160.91) -- (138.13,151.49) ;
\draw [color={rgb, 255:red, 155; green, 155; blue, 155 }  ,draw opacity=1 ]   (216.12,65.56) -- (235.87,99.26) ;
\draw [color={rgb, 255:red, 155; green, 155; blue, 155 }  ,draw opacity=1 ]   (292.51,74.8) -- (312.26,108.94) ;
\draw [color={rgb, 255:red, 155; green, 155; blue, 155 }  ,draw opacity=1 ]   (170.95,230.22) -- (191.06,264.84) ;
\draw [color={rgb, 255:red, 155; green, 155; blue, 155 }  ,draw opacity=1 ]   (282.35,178.59) -- (322.3,178.8) ;
\draw [color={rgb, 255:red, 155; green, 155; blue, 155 }  ,draw opacity=1 ]   (163.47,55.57) -- (201.08,55.98) ;
\draw [color={rgb, 255:red, 189; green, 16; blue, 224 }  ,draw opacity=1 ]   (147.03,83.07) -- (167.3,98.96) ;
\draw [color={rgb, 255:red, 189; green, 16; blue, 224 }  ,draw opacity=1 ]   (342.75,108.61) -- (362.53,126.12) ;
\draw [color={rgb, 255:red, 189; green, 16; blue, 224 }  ,draw opacity=1 ]   (196.05,274.26) -- (220.33,264.93) ;
\draw [color={rgb, 255:red, 189; green, 16; blue, 224 }  ,draw opacity=1 ]   (220.33,264.93) -- (223.98,238.6) ;
\draw [color={rgb, 255:red, 189; green, 16; blue, 224 }  ,draw opacity=1 ]   (342.75,108.61) -- (318.21,118.22) ;
\draw [color={rgb, 255:red, 189; green, 16; blue, 224 }  ,draw opacity=1 ]   (151.03,55.71) -- (147.03,83.07) ;
\draw [line width=0.75]    (226.17,222.62) -- (199.89,156.12) ;
\draw [shift={(198.78,153.33)}, rotate = 68.44] [fill={rgb, 255:red, 0; green, 0; blue, 0 }  ][line width=0.08]  [draw opacity=0] (7.14,-3.43) -- (0,0) -- (7.14,3.43) -- (4.74,0) -- cycle    ;
\draw [line width=0.75]    (226.17,222.62) -- (271.25,164.45) ;
\draw [shift={(273.09,162.08)}, rotate = 127.78] [fill={rgb, 255:red, 0; green, 0; blue, 0 }  ][line width=0.08]  [draw opacity=0] (7.14,-3.43) -- (0,0) -- (7.14,3.43) -- (4.74,0) -- cycle    ;
\draw [color={rgb, 255:red, 208; green, 2; blue, 27 }  ,draw opacity=1 ][line width=0.75]    (226.17,222.62) -- (210.19,181.65) ;
\draw [shift={(209.1,178.86)}, rotate = 68.7] [fill={rgb, 255:red, 208; green, 2; blue, 27 }  ,fill opacity=1 ][line width=0.08]  [draw opacity=0] (7.14,-3.43) -- (0,0) -- (7.14,3.43) -- (4.74,0) -- cycle    ;
\draw [color={rgb, 255:red, 208; green, 2; blue, 27 }  ,draw opacity=1 ][line width=0.75]    (226.17,222.62) -- (253.36,187.67) ;
\draw [shift={(255.2,185.3)}, rotate = 127.88] [fill={rgb, 255:red, 208; green, 2; blue, 27 }  ,fill opacity=1 ][line width=0.08]  [draw opacity=0] (7.14,-3.43) -- (0,0) -- (7.14,3.43) -- (4.74,0) -- cycle    ;

\draw (169.65,158.04) node [anchor=north west][inner sep=0.75pt]    {$\scriptstyle u_{1}$};
\draw (183.65,198.88) node [anchor=north west][inner sep=0.75pt]    {$\textcolor[rgb]{0.82,0.01,0.11}{\scriptstyle u_{2}}$};
\draw (236.65,148.38) node [anchor=north west][inner sep=0.75pt]    {$\scriptstyle u_{3}$};
\draw (256.15,202.88) node [anchor=north west][inner sep=0.75pt]    {$\textcolor[rgb]{0.82,0.01,0.11}{\scriptstyle u_{4}}$};

\end{tikzpicture}

%% file: hats-final.bbl
\bigskip



\begin{thebibliography}{99}
		
\bibitem{BGG}
M.~Baake, F.~G\"{a}hler and P.~Gohlke,
Orbit separation dimension as complexity measure for primitive
inflation tilings, \textit{Ergod.\ Th.\ Dynam.\ Syst.}
\textbf{45} (2025) 2992--3020;
\texttt{arXiv:2311.03541}.

\bibitem{BGS2}
M.~Baake, F.~G\"{a}hler, J.~Maz\'{a}\v{c} and L.~Sadun,
On the long-range order of the Spectre tilings,
\textit{Discr.\ Comput.\ Geom.} (2025), in press;
\texttt{arXiv:2411.15503}.
		
\bibitem{BGS}
M.~Baake, F.~G\"{a}hler and L.~Sadun,
Dynamics and topology of the Hat family of tilings,
\textit{Israel J.\ Math.} (2025), in press;
\texttt{arXiv:2305.05639}.
		
\bibitem{BGM-Cant}
M.~Baake, A.~Gorodetski and J.~Maz\'{a}\v{c},
A naturally appearing family of Cantorvals,
\textit{Lett.\ Math.\ Phys.} \textbf{114} (2024) 101:1--12; 
\texttt{arXiv:2401.05372}. 
		
\bibitem{TAO}
M.~Baake and U.~Grimm,
\textit{Aperiodic Order.\ Vol.~1:\ A Mathematical Invitation},
Cambridge University Press, Cambridge (2013).
		
\bibitem{BG-Rauzy}
M.~Baake and U.~Grimm,
Fourier transform of Rauzy fractals and point spectrum of 1D
Pisot inflation tilings,
\textit{Docum.\ Math.} \textbf{25} (2020) 2303--2337;
\texttt{arXiv:1907.11012}. 

\bibitem{BG-Rauzy2}
M.~Baake and U.~Grimm, 
Diffraction of a model set with complex windows,
\textit{J.\ Phys.:\ Conf.\ Ser.} \textbf{1458} (2020) 012006:1-6;
\texttt{arXiv:1904.08285}.  
		
\bibitem{BH}
M.~Baake and A.~Haynes,
Convergence of Fourier--Bohr coefficients for regular Euclidean model sets,
\textit{preprint}; \texttt{arXiv:2308.07105}.
		
\bibitem{BL}
M.~Baake and D.~Lenz,
Dynamical systems on translation bounded measures:\
Pure point dynamical and diffraction spectra,
\textit{Ergod.\ Th.\ Dynam.\ Syst.} \textbf{24} (2004) 1867--1893;
\texttt{arXiv:math.DS/0302061}.
		
\bibitem{BL05}
M.~Baake and D.~Lenz,
Deformation of Delone dynamical systems and pure point diffraction,
\textit{J.\ Fourier Anal.\ Appl.} \textbf{11} (2005) 125--150;
\texttt{arXiv:math.DS/0404155}.
		
\bibitem{BerDun00}
G.~Bernuau and M.~Dunau,
Fourier analysis of deformed model sets, in
\textit{Directions in Mathematical Quasicrystals},
eds.\ M.~Baake and R.V.~Moody, Fields Institute Monographs, vol. 13,
Amer.\ Math.\ Society, Providence, RI (2000), pp. 43--60.
		
\bibitem{BEIR07}
V.~Berth\'{e}, H.~Ei, S.~Ito and H.~Rao,
On substitution invariant Sturmian words:\
An application of Rauzy fractals,
\textit{RAIRO -- Theor.\ Inform.\ Appl.} \textbf{41} (2007) 329--349.
		
\bibitem{CS1}
A.~Clark and L.~Sadun,
When size matters, 
\textit{Ergod.\ Th.\ Dynam.\ Syst.} \textbf{23} (2003) 1043--1057;
\newline \texttt{arXiv:math.DS/0201152}.
		
\bibitem{CS2}
A.~Clark and L.~Sadun,
When shape matters,
\textit{Ergod.\ Th.\ Dynam.\ Syst.} \textbf{26} (2006) 69--86;
\newline \texttt{arXiv:math.DS/0306214}.
		
\bibitem{FFIW}
D.-J.~Feng, M.~Furukado, S.~Ito and J.~Wu,
Pisot substitutions and Hausdorff dimension of boundaries of atomic surfaces,
\textit{Tsukuba J.\ Math.} \textbf{30} (2006) 195--223.
		
\bibitem{Hutch81}
J.E.~Hutchinson,
Fractals and self-similarity,
\textit{Indiana Univ.\ Math.\ J.} \textbf{30} (1981) 713--747.

\bibitem{Kap1}
C.S.~Kaplan, M.~O’Keeffe and M.M.J.~Treacy,
Periodic diffraction from an aperiodic monohedral tiling,
\textit{Acta Cryst.\ A} \textbf{80} (2024) 72–78.

\bibitem{Kap2}
C.S.~Kaplan, M.~O’Keeffe and M.M.J.~Treacy,
Periodic diffraction from an aperiodic monohedral tiling
--- The Spectre tiling. Addendum,
\textit{Acta Cryst.\ A} \textbf{80} (2024) 460–463.

\bibitem{KelSad}
J.~Kellendonk and L.~Sadun,
Conjugacies of model sets,
\textit{Discr.\ Cont.\ Dynam.\ Syst.\ A}
\textbf{37} (2017) 3805--3830;
\texttt{arxiv:1406.3851}.

\bibitem{LMFDB}
The LMFDB Collaboration,
\textit{The L-functions and modular forms database}, 
\texttt{https://www.lmfdb.org} (2024).
 
\bibitem{Lenz09}
D.~Lenz, 
Continuity of eigenfunctions of uniquely ergodic dynamical systems and
intensity of Bragg peaks, 
\textit{Commun.\ Math.\ Phys.} \textbf{287} (2009) 225--258;
\texttt{arXiv:math-ph/0608026}. 

\bibitem{LSS-long}
D.~Lenz, T.~Spindeler and N.~Strungaru,  
Pure Point Diffraction and mean, Besicovitch and Weyl almost periodicity,
\textit{Ergod.\ Th.\ Dynam.\ Syst.} \textbf{44} (2024) 524--568;
\texttt{arXiv:2006.10821}. 
		
\bibitem{LGJJ93}
J.M.~Luck, C.~Godr\`{e}che, A.~Janner and T.~Janssen,
The nature of the atomic surfaces of quasiperiodic self-similar structures,
\textit{J.\ Phys.\ A:\ Math.\ Gen.} \textbf{26} (1993) 1951--1999.
		
\bibitem{MW}
R.D.~Mauldin amd S.C.~Williams, 
Hausdorff dimension in graph directed constructions,
\textit{Trans.\ Amer.\ Math.\ Soc.} \textbf{309} (1988) 811--829.

\bibitem{Mat}
P.~Mattila,
\textit{Fourier Analysis and Hausdorff Dimension},
Cambridge University Press, Cambridge (2015).
  
\bibitem{Jan}
J.~Maz\'{a}\v{c},
\textit{Fractal and Statistical Phenomena in Aperiodic Order},
PhD thesis (Bielefeld University, 2025);
available electronically from the author.

\bibitem{Mey72}
Y.~Meyer,
\textit{Algebraic Numbers and Harmonic Analysis},
North Holland, Amsterdam (1972).
		
\bibitem{Moo97}
R.V.~Moody,
Meyer sets and their duals,
in \textit{The Mathematics of Long-Range Aperiodic Order},
ed.\  R.~V.~Moody, NATO ASI Series C 489, Kluwer, Dordrecht (1997),
pp. 403--441.
		
\bibitem{Moody}
R.V.~Moody,
Uniform distribution in model sets,
\textit{Can.\ Math.\ Bull.} \textbf{45} (2002) 123--130.		
		
\bibitem{PyFo}
N.~Pytheas Fogg,
\textit{Substitutions in Dynamics, Arithmetics and Combinatorics},
eds.\ V.~Berth\'{e}, S.~Ferenczi, C.~Mauduit and A.~Siegel,
LNM 1794, Springer, Berlin (2002).
		
\bibitem{Schl98}
M.~Schlottmann,
Cut-and-project sets in locally compact Abelian groups,
in \textit{Quasicrystals and Discrete Geometry},
ed.\ J.~Patera, Fields Institute Monographs, vol. 10,
Amer.\ Math.\ Society, Providence, RI (1998), pp. 247--264.
		
\bibitem{ST}
A.~Siegel and J.~Thuswaldner,
\textit{Topological properties of Rauzy fractals},
M\'{e}moires SMF \textbf{118} (2009). 
		
\bibitem{Bernd}
B.~Sing,
\textit{Pisot Substitutions and Beyond},
PhD thesis (Bielefeld University, 2007);
available electronically at \texttt{urn:nbn:de:hbz:361-11555}.
		
\bibitem{Hat}
D.~Smith, J.S.~Myers, C.S.~Kaplan and C.~Goodman-Strauss,
An aperiodic monotile, \textit{Combin.\ Th.} \textbf{4} (2024) 6:1--91;
\texttt{arXiv:2303.10798}.
		
\bibitem{Spectre}
D.~Smith, J.S.~Myers, C.S.~Kaplan and C.~Goodman-Strauss,
A chiral aperiodic monotile,
\textit{Combin.\ Th.} \textbf{4} (2024) 13:1--25;
\texttt{arXiv:2305.17743}.
		
\bibitem{Soc}
J.E.S.~Socolar,
Quasicrystalline structure of the hat monotile tilings,
\textit{Phys.\ Rev.\ B} \textbf{108} (2023) 224109:1--12;
\texttt{arXiv:2305.01174}.
		
\bibitem{Wicks}
K.R.~Wicks,
\textit{Fractals and Hyperspaces},
LNM 1492, Springer, Berlin (1991).
		
\end{thebibliography}
